\documentclass[10pt]{article}
\usepackage[T1]{fontenc}
\usepackage[utf8]{inputenc}
\usepackage[french,english]{babel}
\usepackage{enumerate,amsmath,amsthm}
\usepackage{amssymb}
                                                                                        \newcounter{n}

\newtheorem{lemm}{Lemma}[section]
\newtheorem{theo}{Theorem}[section]
\newtheorem{defi}{Definition}[section]
\newtheorem{rem}{Remark}[section]
\newtheorem{prop}{Proposition}[section]
\newtheorem{cor}{Corollary}[section]

\newcommand{\IN}{{\mathrm I}\!{\mathrm N}}

\title{An arithmetic site of Connes-Consani type for imaginary quadratic fields with class number 1}
\author{Aur\'{e}lien Sagnier}
\begin{document}
\maketitle
\tableofcontents

\section{Introduction (fran\c{c}ais)}
\subsection{Bref rappel des travaux de A.Connes et de C.Consani sur le site arithmetique}

L'analogie entre les corps de fonctions (ie extension algébrique finie de $\mathbb{F}_{q}(T)$ pour $q$ une puissance d'un nombre premier) et les corps de nombres (ie extension algébrique finie de $\mathbb{Q}$) a été et demeure un principe très fécond en géométrie arithmétique. Comme le raconte A.Weil dans \cite{metaphysique}, gr\^{a}ce à cette analogie, l'analogue de la conjecture de Riemann pour les corps de fonctions a pu être démontré dans \cite{weilriemann} et \cite{gmt}. L'espoir a alors depuis été de s'inspirer de ce qui a été fait pour les corps de fonctions pour démontrer la vraie conjecture de Riemann. Pour cela il a longtemps fait partie du folklore qu'il faudrait "faire tendre $q$ vers $1$" et donc "travailler en caractéristique $1$". Cela n'a aucune signification au sens strict et de nombreuses tentatives ont déjà été effectuées pour essayer de donner autant que possible un sens rigoureux à ces phrases "faire tendre $q$ vers $1$" et  "travailler en caractéristique $1$" comme \cite{soule}, \cite{durov},\cite{leichtnam}, \cite{borger}, \cite{lorscheid}, \cite{lescot}, \cite{haran}, \cite{deitmar}, \cite{kurokawa}, \cite{connesconsanimarcolli}, \cite{connesconsani1}, \cite{connesconsani12}, \cite{manin1}, \cite{tits}, \cite{toenvaquie}. Notre principale inspiration dans cette thèse est la tentative la plus récente de A.Connes et C.Consani développée dans \cite{sitearithcr}, \cite{ccsitearith}, \cite{scalingcr}, \cite{scaling}.

En 1995, A.Connes (\cite{trace}) donna une interpr\'{e}tation spectrale des z\'{e}ros de la fonction z\^{e}ta de Riemann en utilisant l'espace des classes d'ad\`{e}les $\mathbb{A}_{\mathbb{Q}}/\mathbb{Q}^{\star}$. En mai 2014, A.Connes et C.Consani r\'{e}ussirent (\cite{sitearithcr}, \cite{ccsitearith}) \`{a} trouver une structure de g\'{e}om\'{e}trie alg\'{e}brique sous-jacente \`{a} cet espace en construisant ce qu'ils ont baptis\'{e} le site arithm\'{e}tique. Cet espace est en fait un topos muni d'une faisceau structural qui a pour propriété d'être de "caractéristique 1" au sens d'être un semi-anneau idempotent. Pour introduire ce dernier ils se sont inspirés de ce qui a été développé dans le domaine max-plus par l'école de Maslov (\cite{maslov}, \cite{litvinov}) et par l'école de l'INRIA (\cite{gaubert1}, \cite{gaubert2}).

Pour construire ce site arithm\'{e}tique, ils consid\`{e}rent la petite cat\'{e}gorie, not\'{e}e $\mathbb{N}^{\times}$, n'ayant qu'un seul objet $\star$ et les fl\`{e}ches index\'{e}es par $\mathbb{N}^{\times}=\mathbb{N}\backslash\{ 0\}$, la composition des fl\'{e}ches \'{e}tant donn\'{e}e par la multiplication sur $\mathbb{N}^{\times}$.

Puis ils consid\`{e}rent $\widehat{\mathbb{N}^{\times}}$ (appel\'{e} ensuite le site arithm\'{e}tique), le topos de pr\'{e}faisceaux associ\'{e} \`{a} cette petite cat\'{e}gorie munie de la topologie chaotique (cf \cite{sga4}), ce qui autrement dit est la cat\'{e}gorie des foncteurs contravariants de la  cat\'{e}gorie $\mathbb{N}^{\times}$ dans la cat\'{e}gorie des ensembles.

Ensuite ils calculent la cat\'{e}gorie des points (au sens de \cite{sga4}) du topos $\widehat{\mathbb{N}^{\times}}$ et ils trouvent  (th\'{e}or\`{e}me 2.1 de \cite{ccsitearith}) que cette derni\`{e}re est \'{e}quivalente \`{a} la cat\'{e}gorie des groupes totalement ordonn\'{e}s isomorphes aux sous groupes ordonnés non triviaux de $(\mathbb{Q},\mathbb{Q}_{+})$ avec comme morphismes les morphismes injectifs de groupes ordonn\'{e}s.

Ils montrent ensuite (proposition 2.5 de \cite{ccsitearith}) que l'ensemble des classes d'isomorphie des points du topos $\widehat{\mathbb{N}^{\times}}$ est en bijection naturelle avec l'espace quotient $\mathbb{Q}^{\times}_{+}\backslash\mathbb{A}^{f}_{\mathbb{Q}}/\hat{\mathbb{Z}}^{\times}$.

Cet espace est une composante de l'espace des classes d'ad\`{e}les $\mathbb{Q}^{\times}_{+}\backslash\mathbb{A}_{\mathbb{Q}}/\hat{\mathbb{Z}}^{\times}$
 d\'{e}j\`{a} utilis\'{e} par Connes (\cite{trace}) pour donner une interpr\'{e}tation spectrale des z\'{e}ros de la fonction z\^{e}ta de Riemann. Connes et Consani \'{e}quipent ensuite leur site arithm\'{e}tique du faisceau structural $(\mathbb{Z}\cup\{-\infty\},\max, +)$ (vu comme semi anneau) et ils montrent alors  (théorème 3.8 de \cite{ccsitearith}) que les points du site arithm\'{e}tique $(\widehat{\mathbb{N}^{\times}},\mathbb{Z}_{\max})$ sur $\mathbb{R}_{\max}$ est l'ensemble des classes d'ad\`{e}les $\mathbb{Q}^{\times}_{+}\backslash\mathbb{A}_{\mathbb{Q}}/\hat{\mathbb{Z}}^{\times}$.

Connes et Consani concluent leur article \cite{ccsitearith} en explicitant la relation qu'entretiennent les topos de Zariski $\text{Spec}(\mathbb{Z})$ et le site arithm\'{e}tique et en construisant le carr\'{e} du site arithm\'{e}tique. Ceci est une \'{e}tape importante si on souhaite adapter \`{a} la fonction z\^{e}ta de Riemann la preuve donn\'{e}e par Weil puis raffinée par Grothendieck (\cite{gmt}) pour l'analogue, dans le cas des corps de fonctions sur un corps fini, de l'hypoth\`{e}se de Riemann.

\subsection{Description des résultats principaux}
Dans cette th\`{e}se, nous essayons de g\'{e}n\'{e}raliser les constructions ci-dessus de Connes et Consani \`{a} d'autres anneaux d'entiers de corps de nombres. Nous avons d'abord considéré l'anneau des entiers de Gauss $\mathbb{Z}[\imath]$ qui est l'anneau d'entiers le plus simple \`{a} consid\'{e}rer apr\`{e}s $\mathbb{Z}$ et il se trouve que ce que nous avons pu faire pour $\mathbb{Z}[\imath]$ reste valable pour les 8 autres anneaux d'entiers de corps quadratiques imaginaires avec un nombre de classes \'{e}gal \`{a} 1.

Dans cette th\`{e}se, nous suivons la strat\'{e}gie g\'{e}n\'{e}rale qui a \'{e}t\'{e} emprunt\'{e}e par Connes et Consani dans \cite{ccsitearith} pour d\'{e}velopper le site arithm\'{e}tique mais la \textit{principale difficult\'{e}} dans la généralisation de leur travail est que leurs constructions et une partie de leurs r\'{e}sultats reposent de mani\`{e}re cruciale sur l'ordre total naturel $<$ pr\'{e}sent sur $\mathbb{R}$ et qui est compatible avec les op\'{e}rations arithm\'{e}tiques de base $+$ et $\times$. H\'{e}las rien de tel n'existe pour $\mathbb{Z}[\imath]$ ainsi la plus grande partie de notre travail a \'{e}t\'{e} de trouver les bons objets \`{a} \'{e}tudier.

Le point de d\'{e}part de notre \'{e}tude est, pour $K$ un corps de nombres quadratique imaginaire avec un nombre de classes \'{e}gal \`{a} 1, la petite cat\'{e}gorie not\'{e}e $\mathcal{O}_{K}$. Elle est constitu\'{e}e d'un unique objet $\star$ et de fl\`{e}ches index\'{e}es par $\mathcal{O}_{K}$, l'anneau des \'{e}l\'{e}ments entiers de $K$ où la loi de composition des fl\`{e}ches est donn\'{e}e par la multiplication $\times$.

Dans ce travail, nous avons déterminé (théorème \ref{theo:ptsscd}) la cat\'{e}gorie des points du topos $\widehat{\mathcal{O}_{K}}$ (ie le topos des pr\'{e}faisceaux sur la petite cat\'{e}gorie $\mathcal{O}_{K}$ munie de la topologie chaotique) et nous avons montr\'{e} que celle-ci est \'{e}quivalente \`{a} la cat\'{e}gorie des sous-$\mathcal{O}_{K}$-modules de $K$. C'est pour avoir cette équivalence qu'il nous faut supposer que le nombre de classes de K est égal à 1, sinon le calcul des points nous donnerait l'équivalence avec la catégorie des sous-$\mathcal{O}_{K}$-modules de $K$ de rang 1 (au sens que deux éléments distincts sont commensurables) ce qui n'est pas intéressant ensuite dans la perspective de relier les points à un quotient des adèles finies.

Nous montrons (théorème \ref{theo:adlfst}) que nous avons, de mani\`{e}re analogue \`{a} Connes et Consani, une interpr\'{e}tation ad\'{e}lique de l'ensemble des classes d'isomorphie des points du topos $\widehat{\mathcal{O}_{K}}$. Cet ensemble est en bijection avec l'espace quotient $\frac{\mathbb{A}^{f}_{K}}{(K^{\star}(\prod\mathcal{O}_{\mathfrak{p}}^{\star}\times \{1\}))}$ ce qui généralise la proposition 2.5 de \cite{ccsitearith} de Connes et Consani.

Un autre d\'{e}fi majeur est de trouver un faisceau structural pour le topos $\widehat{\mathcal{O}_{K}}$. Il fallait que ce soit un semi-anneau li\'{e} d'une certaine fa\c{c}on \`{a} $\mathcal{O}_{K}$. Dans ce travail, nous proposons l'ensemble $\mathcal{C}_{\mathcal{O}_{K}}$ des polygones convexes du plan d'intérieur non vide dont le centre est $0$, qui sont invariants par l'action par similitudes directes des unit\'{e}s de $\mathcal{O}_{K}$ et dont les sommets sont dans $\mathcal{O}_{K}$ (quelques restrictions suppl\'{e}mentaires doivent \^{e}tre faites dans les cas o\`{u} $K$ est diff\'{e}rent de $\mathbb{Q}(\imath)$ et $\mathbb{Q}(j)$). Nous le munissons des op\'{e}rations $\text{Conv}(\bullet\cup\bullet )$ et $+$ (la somme de Minskowski), ce qui en fait un semi-anneau idempotent, qu'on définit comme le faisceau structural sur $\widehat{\mathcal{O}_{K}}$.

Nous définissons ensuite $\mathcal{C}_{K,\mathbb{C}}$ comme l'ensemble des polygones convexes du plan d'intérieur non vide dont le centre est $0$, qui sont invariants par l'action par similitudes directes des unit\'{e}s de $\mathcal{O}_{K}$ et dont les sommets sont dans $\mathbb{C}$ (quelques restrictions suppl\'{e}mentaires doivent \^{e}tre faites dans les cas o\`{u} $K$ est diff\'{e}rent de $\mathbb{Q}(\imath)$ et $\mathbb{Q}(\imath\sqrt{3})$) auxquels on rajoute aussi $\{ 0\}$ et $\emptyset$. On le munit des op\'{e}rations $\text{Conv}(\bullet\cup\bullet )$ et $+$ (la somme de Minskowski)), ce qui en fait un anneau idempotent. On peut remarquer que $\emptyset$, l'élément neutre de $\text{Conv}(\bullet\cup\bullet )$, est un \'{e}l\'{e}ment absorbant pour $+$. Nous prouvons ensuite que $\text{Aut}_{\mathbb{B}}^{+}(\mathcal{C}_{K,\mathbb{C}})$, l'ensemble des $\mathbb{B}$-automorphismes de $\mathcal{C}_{K,\mathbb{C}}$ qu'on dira directs, est égal à $\mathbb{C}^{\star}/\mathcal{U}_{K}$. L'ensemble de tous les $\mathbb{B}$-automorphismes de $\mathcal{C}_{K,\mathbb{C}}$ a une structure un peu plus compliquée. Cela suggère que, heuristiquement, $\mathcal{C}_{K,\mathbb{C}}$ est de dimension tropicale 2, ce qui est une différence avec ce qu'ont fait A.Connes et C.Consani dans \cite{ccsitearith} et suggère déjà que l'interprétation spectrale sera un peu différente de la leur.

Nous pouvons alors prouver (cf \ref{theo:ptsckc}) que l'ensemble points du site arithm\'{e}tique $(\widehat{\mathcal{O}_{K}},\mathcal{C}_{\mathcal{O}_{K}})$ au dessus de $\mathcal{C}_{K,\mathbb{C}}$ est en bijection naturelle avec $\frac{\mathbb{A}_{K}}{\left( K^{\star}\left(\prod_{\mathfrak{p}}\mathcal{O}_{\mathfrak{p}}^{\star}\times\{ 1\}\right)\right)}$. Ceci généralise le théorème 3.8 de \cite{ccsitearith} de Connes et Consani.

Notons $\mathcal{H}$ l'espace de Hilbert associ\'{e} par Connes à $\frac{\mathbb{A}_{K}^{f}\times\mathbb{C}}{K^{\star}}$ dans \cite{trace} pour définir la réalisation spectrale des fonctions L de Hecke de $K$. Posons $G=\frac{K^{\star}\times\left(\prod_{\mathfrak{p}\,\text{prime}}\mathcal{O}_{\mathfrak{p}}^{\star}\times 1\right)}{K^{\star}}$. L'espace de Hilbert associé à notre espace $\frac{\mathbb{A}_{K}^{f}\times\mathbb{C}}{(K^{\star}(\prod\mathcal{O}_{\mathfrak{p}}^{\star}\times \{1\}))}$ est $\mathcal{H}^{G}$ (cf \ref{theo:intsp}).

Posons $C_{K,1}$ le groupe des classes d'adèles de norme 1. L'espace de Hilbert associé dans \cite{trace} à $\zeta_{K}$, la fonction zêta de Dedekind de $K$, est $\mathcal{H}^{C_{K,1}}$. On peut alors observer dans notre cas que $\frac{C_{K,1}}{G}=\frac{\mathbb{S}^{1}}{\mathcal{U}_{K}}$. Nous prouvons alors (cf théorème \ref{theo:intsp}) que $\mathcal{H}^{G}=\bigoplus_{\chi\in\widehat{\mathbb{S}^{1}/\mathcal{U}_{K}}}\mathcal{H}_{\chi}^{G}$   et que l'interprétation spectrale dit alors que les valeurs propres $\lambda$ du générateur infinitésimal de l'action de $1\times\mathbb{R}^{\star}_{+}$ sur $\mathcal{H}_{\chi}^{G}$ sont les $z-\frac{1}{2}\in\imath\mathbb{R}$ où $L(\chi, z )=0$. En particulier lorsque $\chi$ est trivial, on obtient ainsi une interprétation spectrale des zéros de la fonction zêta de Dedekind de $K$. La nuance avec ce qu'ont fait A.Connes et C.Consani dans \cite{ccsitearith} est que dans l'interprétation spectrale, on obtient certaines fonctions $L$ de Hecke pour des caractères non triviaux en plus de la fonction zêta (ici de Dedekind et non de Riemann). Cela est dû au fait que $\mathcal{C}_{K,\mathbb{C}}$ est heuristiquement de dimension tropicale 2. Notre travail donne donc une famille d'exemples où le topos associé prend en compte certaines fonctions $L$ non triviales, cela donnera peut-être une piste pour atteindre dans le futur plus de fonctions $L$ de Hecke.

Nous étendons ensuite à $K$ le théorème 5.3 de \cite{ccsitearith} de Connes et Consani qui établit un lien entre $\text{Spec}(\mathbb{Z})$ et le topos $\left(\widehat{\mathbb{N}},\mathbb{Z}_{\max}\right)$. Plus précisément (cf theorèmes \ref{theo:liensiteun} et \ref{theo:liensitedeux}), nous construisons un morphisme géométrique $T:\text{Spec}(\mathcal{O}_{K})\to\widehat{\mathcal{O}_{K}}$ et montrons que pour $\mathfrak{p}$ idéal premier de $\mathcal{O}_{K}$, la fibre $T^{\star}(\mathcal{C}_{\mathcal{O}_{K}})_{\mathfrak{p}}$ est le semi-anneau $\mathcal{C}_{H_{\mathfrak{p}}}$. De plus, au point générique, la fibre de $T^{\star}(\mathcal{C}_{\mathcal{O}_{K}})$ est $\mathbb{B}$.

Enfin dans le paragraphe 10, nous supposons que $K=\mathbb{Q}(\imath)$. Nous commençons (cf proposition \ref{prop:faiscstrucfct}) par donner une description fonctionnelle $\mathcal{F}_{\mathbb{Z}[\imath]}$ du faisceau structural $\mathcal{C}_{\mathbb{Z}[\imath]}$ de $\widehat{\mathbb{Z}[\imath]}$. Ceci nous permet alors (cf définition \ref{defi:defprodtensnred}) de définir le $\mathbb{B}$-module $\mathcal{F}_{\mathbb{Z}[\imath]}\otimes_{\mathbb{B}}\mathcal{F}_{\mathbb{Z}[\imath]}$ et de montrer (cf propositions \ref{prop:prodtsemrg} et \ref{prop:actprodtens}) que $\mathcal{F}_{\mathbb{Z}[\imath]}\otimes_{\mathbb{B}}\mathcal{F}_{\mathbb{Z}[\imath]}$ peut être muni de manière naturelle d'une structure de semi-anneau sur lequel $\mathbb{Z}[\imath]\times\mathbb{Z}[\imath]$ agit. Cela nous permet alors (cf définition \ref{defi:defprodtensnred}) de définir le carré non réduit $\left(\widehat{\mathbb{Z}[\imath]\times\mathbb{Z}[\imath]}, \mathcal{F}_{\mathbb{Z}[\imath]}\otimes_{\mathbb{B}}\mathcal{F}_{\mathbb{Z}[\imath]}\right)$. Il semblerait que ce semi-anneau n'est pas simplifiable. Nous lui associons alors (cf définition \ref{defi:prdtensred}) son semi-anneau simplifiable canonique $\mathcal{F}_{\mathbb{Z}[\imath]}\hat{\otimes}_{\mathbb{B}}\mathcal{F}_{\mathbb{Z}[\imath]}$, ce qui nous permet de considérer (cf définition \ref{defi:redsquare}) le carré réduit $\left(\widehat{\mathbb{Z}[\imath]\times\mathbb{Z}[\imath]}, \mathcal{F}_{\mathbb{Z}[\imath]}\hat{\otimes}_{\mathbb{B}}\mathcal{F}_{\mathbb{Z}[\imath]}\right)$.

\subsection{D\'{e}veloppements futurs}

Dans la construction du carré du site arithmétique $(\widehat{\mathcal{O}_{\mathbb{Z}[\imath]}},\mathcal{C}_{\mathcal{O}_{\mathbb{Z}[\imath]}})$, je suis déjà passé du point de vue de certains polygones convexes ($\mathcal{C}_{\mathbb{Z}[\imath]}$) au point de vue de certaines fonctions convexes affines par morceaux sur $[1,\imath]/(1\sim\imath)$ vu comme une courbe tropicale ($\mathcal{F}_{\mathbb{Z}[\imath]}$). Dans ma thèse, j'ai défini abstraitement et formellement le produit tensoriel $\mathcal{F}_{\mathbb{Z}[\imath]}\otimes_{\mathbb{B}}\mathcal{F}_{\mathbb{Z}[\imath]}$. Il serait intéressant d'avoir une description explicite de ce produit tensoriel car dans le cas de $\mathbb{Z}\otimes_{\mathbb{B}}\mathbb{Z}$, la description explicite de ce dernier produit tensoriel a des applications aux systèmes dynamiques à événements discrets comme montré dans \cite{gaubert}. Je suis actuellement en train d'essayer de trouver une description explicite du produit tensoriel $\mathcal{F}_{\mathbb{Z}[\imath]}\otimes_{\mathbb{B}}\mathcal{F}_{\mathbb{Z}[\imath]}$ et nous pouvons espérer qu'en plus des applications au carré du site arithmétique, une telle description pourrait être utile en mathématiques appliquées.

Je compte aussi poursuivre mes recherches dans une autre direction tout d'abord en remarquant que ce qui a été fait dans ma thèse jusqu'au faisceau structural peut être facilement généralisé à tout corps de nombres dont le nombre de classes est 1. La principale difficulté sera alors de trouver un faisceau structural adéquat. Dans le cas où $K$ était un corps de nombres imaginaire quadratique et de nombre de classe 1, je n'ai pas eu besoin de passer au point de vue fonctionnel dès le début car le groupe $\mathcal{U}_{K}$ des unités de $\mathcal{O}_{K}$, qui doit être vu comme le groupe des symétries, est fini. Cependant ceci n'est plus vrai même pour un corps de nombres quadratique réel. En étudiant le cas de $\mathbb{Z}[\sqrt{2}]$, il me semble que dans le cas général, pour un corps de nombres avec nombre de classes égal à 1, il faille regarder les fonctions rationnelles (du point de vue de la géométrie tropicale) sur le quotient (du point de vue de la géométrie tropicale) d'un sous ensemble bien choisi de $\mathbb{R}^{[K:\mathbb{Q}]}$ par l'action multiplication des éléments de $\mathcal{U}_{K}$. Le cas de $K=\mathbb{Q}(\sqrt{2})$ est en cours de réalisation.

Pour un corps de nombres $K$ avec nombre de classes différent de $1$, il semble qu'il serait raisonnable d'essayer d'étudier le topos $\widehat{\mathcal{I}_{K}}$ : le topos de préfaisceaux sur le site défini par la petite catégorie $\mathcal{I}_{K}$, la catégorie avec un seul objet $\star$ et dont les flèches sont indexées par $\mathcal{I}_{K}$ (le mono\"{i}de des idéaux entiers) et la loi de composition des flèches donnée par la multiplication des idéaux et sur lequel on met la topologie chaotique au sens de \cite{sga4}. Il semble que la catégorie des points de ce topos est un quotient intéressant des adèles finies, je suis en train de le calculer. La principale difficulté sera alors de trouver un faisceau structural adéquat, il faudrait que ce dernier soit de dimension tropicale 1, car grâce à l'interprétation spectrale de la fonction zêta de Dedekind de $K$, nous savons que nous devons quotienter l'espace des classes d'adèles par l'action des classes d'idèles de norme 1 (ces dernières pouvant être vues comme le noyau du module) pour qu'il ne reste plus que l'action par $\mathbb{R}^{\star}_{+}$. Ainsi, il me parait raisonnable de penser qu'heuristiquement le semi-anneau dans lequel les points prennent leur valeur et le faisceau structural doivent être de dimension tropicale 1.

Comme $\widehat{\mathcal{I}_{K}}$ a l'air d'être un candidat intéressant pour faire office de site arithmétique d'un corps de nombres $K$ quelconque et puis que $DR_{K}$, le mono\"{i}de de Deligne-Ribet de $K$, est intimement relié au mono\"{i}de $\mathcal{I}_{K}$ et que $DR_{K}$ joue un rôle important dans la structure des sytèmes de Bost-Connes comme expliqué dans \cite{bora}, il pourrait être intéressant et probablement difficile de calculer la catégorie des points du topos $\widehat{DR_{K}}$. Cela permettrait peut-être d'établir un lien entre le site arithmétique et le système de Bost-Connes et donner une meilleure compréhension des deux.

Il serait aussi intéressant de voir s'il serait possible de développer un analogue du site arithmétique dans le cas des corps de fonctions sur un corps fini et de voir si la preuve existante de Weil et Grothendieck de l'analogue de l'hypothèse de Riemann peut aussi être réalisée dans ce cadre.

En mars 2016 dans \cite{scalingcr} et in \cite{scaling}, A.Connes et C.Consani ont construit par extension des scalaires un site des fréquences pour $(\widehat{\mathbb{N}^{\times}},\mathbb{Z}_{\max})$ et ont montré que l'espace des classes d'adèles de $\mathbb{Q}$ qui est si important dans l'interprétation spectrale des zéros de la fonction zêta de Riemann admet une structure de courbe tropicale. Dans le futur, j'ai l'intention de construire des sites des fréquences similaires pour d'autres corps de nombres. \`{A} propos du site des fréquences, il serait aussi intéressant de voir si le morphisme géométrique existant de $\text{Spec}(\mathbb{Z})$ dans $(\widehat{\mathbb{N}^{\times}},\mathbb{Z}_{\max})$ peut être étendu et aller de la compactification d'Arakelov de $\text{Spec}(\mathbb{Z})$ dans le site des fréquences défini par A.Connes et C.Consani. Il serait ensuite intéressant de voir si un tel phénomène se produit aussi pour des corps de nombres plus généraux. Pour cela, il faudra utiliser le formalisme de $\mathbb{S}$-algèbres développé par A.Connes et C.Consani dans \cite{gamma}.

\subsection{Remerciements}
Je tiens à remercier Eric Leichtnam, mon directeur de thèse, pour sa patience, ses conseils et la liberté qu'il m'a laissée dans l'orientation de mes recherches. Je tiens aussi à remercier Caterina Consani pour ses encouragements, ses conseils et l'intér\^{e}t qu'elle a eu pour mes travaux tout au long de ma thèse. Je tiens enfin à exprimer ma profonde gratitude à Alain Connes pour ses conseils inspirants, sa disponibilité et pour m'avoir signalé une erreur dans la première version du théorème \ref{theo:intsp}.

\section{Introduction (english)}
\subsection{Brief summary of the work of A.Connes and C.Consani on the arithmetic site}

The analogy between function fields (ie finite algebraic extension of $\mathbb{F}_{q}(T)$  for $q$ a power of a prime number) and number fields (ie finite algebraic extension of $\mathbb{Q}$) has been and remains a fruitful principle in arithmetic geometry. As A.Weil tells in \cite{metaphysique}, thanks to this analogy, the analogous for function fields of the Riemann conjecture was proved in \cite{weilriemann} and \cite{gmt}. Since then, the hope has been to get inspiration from what happens in the function field case in order to try to prove the Riemann conjecture. For a long time the folklore has been to say that in order to achieve this, one should try to make "$q$ tend to $1$" and so work in "characteristic $1$". Rigorously speaking it doesn't make any sense but many people since then have tried to give a reasonable meaning to the sentences "$q$ tend to $1$" and "characteristic $1$" like in \cite{soule}, \cite{durov}, \cite{leichtnam}, \cite{borger}, \cite{lorscheid}, \cite{lescot}, \cite{haran}, \cite{deitmar}, \cite{kurokawa}, \cite{connesconsanimarcolli}, \cite{connesconsani1}, \cite{connesconsani12}, \cite{manin1}, \cite{tits}, \cite{toenvaquie}. In this thesis, our main inspiration is coming from the last approach of A.Connes and C.Consani on this problem as developped in \cite{sitearithcr}, \cite{ccsitearith}, \cite{scalingcr}, \cite{scaling}. 

In 1995, A.Connes (\cite{trace}) gave a spectral interpretation of the zeroes of the Riemann zeta function using the adele class space $\mathbb{A}_{\mathbb{Q}}/\mathbb{Q}^{\star}$. In May 2014, A.Connes and C.Consani (\cite{sitearithcr}, \cite{ccsitearith}) found  for this space $\mathbb{A}_{\mathbb{Q}}/\mathbb{Q}^{\star}$ an underlying structure coming from algebraic geometry by building what they have called the arithmetic site. This space is in fact a topos with a structural sheaf which has the property to be of "caracteristic $1$" in the sense that it is an idempotent semiring. To introduce this structural sheaf, they drew their inspiration from what has been developped in the max-plus area by Maslov's school (\cite{maslov}, \cite{litvinov}) and by the school of the INRIA (\cite{gaubert1}, \cite{gaubert2}).

To construct this arithmetic site, they consider the small category, denoted $\mathbb{N}^{\times}$, with only one object $\star$ and the arrows indexed by $\mathbb{N}^{\times}=\mathbb{N}\backslash\{ 0\}$. The composition law of arrows is given by the multiplication on $\mathbb{N}^{\times}$.

Then they consider $\widehat{\mathbb{N}^{\times}}$, later called the arithmetic site, the presheaf topos associated to this small category considered with the chaotic topology (cf \cite{sga4}), in other words it is the category of contravariant functors from the category $\mathbb{N}^{\times}$ into the category of sets.

Then they show (theorem 2.1 of \cite{ccsitearith}) that the category of points (in the sense of \cite{sga4}) of the topos $\widehat{\mathbb{N}^{\times}}$ is equivalent to the category of totally ordered groups isomorphic to non trivial subgroups of $(\mathbb{Q},\mathbb{Q}_{+})$ with morphisms in the category being injective morphisms of ordered groups.

Then they show (proposition 2.5 of \cite{ccsitearith}) that the set of classes of isomorphic points of the topos $\widehat{\mathbb{N}^{\times}}$ is in natural bijection with the quotient space $\mathbb{Q}^{\times}_{+}\backslash\mathbb{A}^{f}_{\mathbb{Q}}/\hat{\mathbb{Z}}^{\times}$.

This space is a component of the adele class space $\mathbb{Q}^{\times}_{+}\backslash\mathbb{A}_{\mathbb{Q}}/\hat{\mathbb{Z}}^{\times}$ already used by Connes (\cite{trace}) to give a spectral interpretation of the zeroes of the Riemann zeta function. Connes and Consani then put on the arithmetic site as a structural sheaf the idempotent semiring $(\mathbb{Z}\cup\{-\infty\},\max, +)$. They show then in theorem 3.8 of \cite{ccsitearith} that the points of the arithmetic site $(\widehat{\mathbb{N}^{\times}},\mathbb{Z}_{\max})$ over $\mathbb{R}^{\max}$ is the adele class space $\mathbb{Q}^{\times}_{+}\backslash\mathbb{A}_{\mathbb{Q}}/\hat{\mathbb{Z}}^{\times}$.

Connes and Consani end their article \cite{ccsitearith} by describing precisely the relation between the Zariski topos $\text{Spec}(\mathbb{Z})$ and the arithmetic site, and by building the square of the arithmetic site. Building the square of the arithmetic site is important in the hope of adapting to the Riemann zeta function the proof given by Weil and refined by Grothendieck in \cite{gmt} of the analogue of the Riemann hypothesis in the case of function fields.

\subsection{Description of the main results}

In this thesis, we try to generalize the constructions of Connes and Consani mentionned above to other rings of integers of number fields. We have first considered $\mathbb{Z}[\imath]$ the ring of Gaussian integers which is the simplest ring of integers to look after $\mathbb{Z}$ and it turns out that what we have done for $\mathbb{Z}[\imath]$ remains true for the 8 other rings of integers of imaginary quadratic number fields of class number 1.

In this thesis, we follow the general strategy adopted by Connes and Consani in \cite{ccsitearith} to develop the arithmetic site but the \emph{main difficulty} in generalizing their work is that their constructions and part of their results strongly rely on the natural total order $<$ existing on $\mathbb{R}$ which is compatible with basic arithmetic operations $+$ and $\times$. Of course nothing of this sort exists in the case of $\mathbb{Z}[\imath]$ and the main part of my work has been to find the good objects to study.

The starting point of my study is, for $K$ an imaginary quadratic field with class number 1, the small category denoted $\mathcal{O}_{K}$ with only one object $\star$ and arrows indexed by $\mathcal{O}_{K}$, the ring of integers of the number field $K$, the composition law of the arrows being given by the multiplication law $\times$.

In this thesis, we have shown (cf \ref{theo:ptsscd}) that the category of points of the topos $\widehat{\mathcal{O}_{K}}$ (ie the presheaf topos on the small category $\mathcal{O}_{K}$ endowed with the chaotic topology) is equivalent to the category of sub-$\mathcal{O}_{K}$-modules of $K$.

We  have shown (cf\ref{theo:adlfst}), in the same way as Connes and Consani, that we have an adelic interpretation of the set of classes of isomorphic points of the topos $\widehat{\mathcal{O}_{K}}$. This set is in bijection with $\frac{\mathbb{A}^{f}_{K}}{(K^{\star}(\prod\mathcal{O}_{\mathfrak{p}}^{\star}\times \{1\}))}$, which generalizes the proposition 2.5 of \cite{ccsitearith} of Connes and Consani.

Another great difficulty is to find a structural sheaf for the topos $\widehat{\mathcal{O}_{K}}$. It needs to be an idempotent semiring somehow linked to $\mathcal{O}_{K}$. In this work, we propose the set $\mathcal{C}_{\mathcal{O}_{K}}$ of convex polygons of the plane whose interior is non empty, invariants by the action by direct similitudes of the units $\mathcal{U}_{K}$ of $\mathcal{O}_{K}$ and whose summits are in $\mathcal{O}_{K}$ (some restrictions have to be made when $K$ is not equal to $\mathbb{Q}(\imath)$ or $\mathbb{Q}(\imath\sqrt{3})$). We endow it with the operations $\text{Conv}(\bullet\cup\bullet )$ and $+$ (the Minkowski sum). These laws turn $\mathcal{C}_{\mathcal{O}_{K}}$ into an idempotent semiring which we define to be the structural sheaf on $\widehat{\mathcal{O}_{K}}$.

Then we define $\mathcal{C}_{K,\mathbb{C}}$ as the set of convex polygons of the plane whose interior is non empty, invariants by the action by direct similitudes of the units $\mathcal{U}_{K}$ of $\mathcal{O}_{K}$ and whose summits are in $\mathbb{C}$ (some restrictions have to be made when $K$ is not equal to $\mathbb{Q}(\imath)$ or $\mathbb{Q}(\imath\sqrt{3})$) and we endow it with the operations $\text{Conv}(\bullet\cup\bullet )$ and $+$ (the Minkowski sum) to which we also add the sets $\{0\}$ and $\emptyset$. These laws turn $\mathcal{C}_{K,\mathbb{C}}$ into an idempotent semiring. We can already remark that $\emptyset$, the neutral element of the law $\text{Conv}(\bullet\cup\bullet)$, is an absorbant element for the law $+$. We then prove that $\text{Aut}_{\mathbb{B}}^{+}(\mathcal{C}_{K,\mathbb{C}})$, the set of $\mathbb{B}$-automorphisms of $\mathcal{C}_{K,\mathbb{C}}$ which we will call direct, is equal to $\mathbb{C}^{\star}/\mathcal{U}_{K}$. The set of all $\mathbb{B}$-automorphisms of $\mathcal{C}_{K,\mathbb{C}}$ has a more complicated structure. This suggests heuristically that $\mathcal{C}_{K,\mathbb{C}}$ is of tropical dimension 2 which is  different from what A.Connes and C.Consani did in \cite{ccsitearith} and already suggests that our spectral interpretation will be different from the one they obtained.

We prove then (cf \ref{theo:ptsckc}) that the set of points of the arithmetic site $(\widehat{\mathcal{O}_{K}},\mathcal{C}_{\mathcal{O}_{K}})$ over $\mathcal{C}_{K,\mathbb{C}}$ is in natural bijection with $\frac{\mathbb{A}_{K}}{\left( K^{\star}\left(\prod_{\mathfrak{p}}\mathcal{O}_{\mathfrak{p}}^{\star}\times\{ 1\}\right)\right)}$. This generalizes the theorem 3.8 of \cite{ccsitearith} of Connes and Consani.

Let us now denote by $\mathcal{H}$ the Hilbert space associated by Connes to $\frac{\mathbb{A}_{K}^{f}\times\mathbb{C}}{K^{\star}}$ in \cite{trace} to build the spectral realization of Hecke L functions of $K$. Let us denote $G=\frac{K^{\star}\times\left(\prod_{\mathfrak{p}\,\text{prime}}\mathcal{O}_{\mathfrak{p}}^{\star}\times 1\right)}{K^{\star}}$. The Hilbert space associated to our space $\frac{\mathbb{A}_{K}^{f}\times\mathbb{C}}{(K^{\star}(\prod\mathcal{O}_{\mathfrak{p}}^{\star}\times \{1\}))}$ is $\mathcal{H}^{G}$ (cf \ref{theo:intsp}).

Let us denote $C_{K,1}$ the group of adele classes of norm 1. The Hilbert space associated in \cite{trace} to $\zeta_{K}$, the Dedekind zeta function of $K$, is $\mathcal{H}^{C_{K,1}}$. But in our case, we can notice that $\frac{C_{K,1}}{G}=\frac{\mathbb{S}^{1}}{\mathcal{U}_{K}}$. We can therefore prove (cf theorem \ref{theo:intsp}) that $\mathcal{H}^{G}=\bigoplus_{\chi\in\widehat{\mathbb{S}^{1}/\mathcal{U}_{K}}}\mathcal{H}_{\chi}^{G}$  and that the spectral interpretation tells us that the eigenvalues of the infinitesimal generator of the action of $1\times\mathbb{R}^{\star}_{+}$ on $\mathcal{H}_{\chi}^{G}$ are exactly the $z-\frac{1}{2}$ such that $L(\chi,z)=0$. In particular when $\chi$ is trivial we get a spectral interpretation of the zeroes of the Dedekind zeta function of $K$. The slight difference here with what A.Connes and C.Consani did in \cite{ccsitearith} is that the spectral interpretation gives us not only the zeta function (here of Dedekind and not of Riemann) but also some Hecke $L$ functions. The reason for this is that heuristically $\mathcal{C}_{K,\mathbb{C}}$ is of tropical dimension 2. Our work is thus giving a family of examples where the associated topos encodes some non trivial $L$ functions, it may give a hint on how to take into account more Hecke $L$ functions in the future.

Then we extend to the case of $K$ the theorem 5.3 of \cite{ccsitearith} of Connes and Consani which establish a link between $\text{Spec}(\mathbb{Z})$ and the topos $\left(\hat{\mathbb{N}},\mathbb{Z}_{\max}\right)$. More precisely (cf theorems \ref{theo:liensiteun} and \ref{theo:liensitedeux}), we build a geometric morphism $T:\text{Spec}(\mathcal{O}_{K})\to\hat{\mathcal{O}_{K}}$ and show that for $\mathfrak{p}$ a prime ideal of $\mathcal{O}_{K}$, the fiber $T^{\star}(\mathcal{C}_{\mathcal{O}_{K}})_{\mathfrak{p}}$ is the semiring $\mathcal{C}_{H_{\mathfrak{p}}}$. Moreover at the generic point, the fiber of $T^{\star}(\mathcal{C}_{\mathcal{O}_{K}})$ is $\mathbb{B}$.

Lastly, in section 10, we assume that $K=\mathbb{Q}(\imath)$. We begin (cf proposition \ref{prop:faiscstrucfct}) with giving a functional description $\mathcal{F}_{\mathbb{Z}[\imath]}$ of the structural sheaf $\mathcal{C}_{\mathbb{Z}[\imath]}$ of $\widehat{\mathcal{O}_{K}}$. This allows us (cf definition \ref{defi:defprodtensnred}) to define the $\mathbb{B}$-module $\mathcal{F}_{\mathbb{Z}[\imath]}\otimes_{\mathbb{B}}\mathcal{F}_{\mathbb{Z}[\imath]}$ and to show (cf propositions \ref{prop:prodtsemrg} and \ref{prop:actprodtens})that it can be naturally endowed with a structure of semiring on which $\mathbb{Z}[\imath]\times\mathbb{Z}[\imath]$ acts. It allows us then (cf definition \ref{defi:defprodtensnred}) to define the non reduced square $\left(\widehat{\mathbb{Z}[\imath]\times\mathbb{Z}[\imath]}, \mathcal{F}_{\mathbb{Z}[\imath]}\otimes_{\mathbb{B}}\mathcal{F}_{\mathbb{Z}[\imath]}\right)$. It seems that this semiring is not multiplicatively cancellative. Therefore we associate to it (cf definition \ref{defi:prdtensred}) its canonical multiplicatively cancellative semiring $\mathcal{F}_{\mathbb{Z}[\imath]}\hat{\otimes}_{\mathbb{B}}\mathcal{F}_{\mathbb{Z}[\imath]}$, which allows us to define (cf \ref{defi:redsquare}) the reduced square $\left(\widehat{\mathbb{Z}[\imath]\times\mathbb{Z}[\imath]}, \mathcal{F}_{\mathbb{Z}[\imath]}\hat{\otimes}_{\mathbb{B}}\mathcal{F}_{\mathbb{Z}[\imath]}\right)$.

\subsection{Future projects}

In the construction of the square of the arithmetic site  $(\widehat{\mathcal{O}_{\mathbb{Z}[\imath]}},\mathcal{C}_{\mathcal{O}_{\mathbb{Z}[\imath]}})$ I have already switched viewpoints from the set $\mathcal{C}_{\mathbb{Z}[\imath]}$ of convex polygons with some special hypothesis to the set $\mathcal{F}_{\mathbb{Z}[\imath]}$ of some special convex affine by parts functions on $[1,\imath]/(1\sim\imath)$ seen as a tropical curve. In my thesis, I have defined abstractly the tensor product over $\mathbb{B}$ : $\mathcal{F}_{\mathbb{Z}[\imath]}\otimes_{\mathbb{B}}\mathcal{F}_{\mathbb{Z}[\imath]}$. As shown in \cite{gaubert}, the concrete description of $\mathbb{Z}\otimes_{\mathbb{B}}\mathbb{Z}$ has applications to discrete event dynamic systems. I am currently trying to find a concrete description of $\mathcal{F}_{\mathbb{Z}[\imath]}\otimes_{\mathbb{B}}\mathcal{F}_{\mathbb{Z}[\imath]}$ and one could hope, as in the case of $\mathbb{Z}\otimes_{\mathbb{B}}\mathbb{Z}$, that the concrete description of $\mathcal{F}_{\mathbb{Z}[\imath]}\otimes_{\mathbb{B}}\mathcal{F}_{\mathbb{Z}[\imath]}$ could be useful too in applied mathematics.

Another direction of research could consist first by noticing that what has been done in my thesis until the structural sheaf of the arithmetic site could be generalized easily for a $K$ a number field whose ring of integers is principal. The main difficulty will be then to find an adequate structural sheaf. In the case where $K$ was imaginary quadratic and of class number 1, we could have prevented us from going to the functionnal viewpoint at the begining because the group  $\mathcal{U}_{K}$ of units of $\mathcal{O}_{K}$, which must be seen as a symmetry group, is finite. However this is no longer true even for a quadratic real number field. By studying the case of $\mathbb{Z}[\sqrt{2}]$, it seems to me that in the general case, one should try to look at rationnal functions (from the viewpoint of tropical geometry) on the quotient (seen from the viewpoint of tropical geometry) of a well chosen subset of $\mathbb{R}^{[K:\mathbb{Q}]}$ by the multiplicative action as elements of $K$ of the elements of the group $\mathcal{U}_{K}$. The case of $K=\mathbb{Q}(\sqrt{2})$ is under realisation.

For number fields $K$ with class number different than $1$, it seems that a reasonable topos to study, would be the topos $\widehat{\mathcal{I}_{K}}$ : the presheaf topos on the site defined by the small category $\mathcal{I}_{K}$, the category with one object $\star$ with arrows indexed by the elements of $\mathcal{I}_{K}$ (the monoid of integral ideals) and the law of composition of arrows given by the multiplication of ideals, and with the chaotic topology in the sense of \cite{sga4}. It seems that the category of points of this topos is an interesting quotient of finite adeles, I am computing it now. The main difficulty will be to find a suitable structural sheaf. We would have to try to find a structural sheaf of tropical dimension 1, because thanks to the spectral interpretation, we know that in order to get in this spectral interpretation the Dedekind zeta function of $K$, we have to divide the adèle class space by the idèles classes of norm 1 (ie the kernel of the module map) and what is left is only an action by $\mathbb{R}^{\star}_{+}$.

Since $\widehat{\mathcal{I}_{K}}$ seems to be an interesting candidate for the arithmetic site for a general number field $K$, and since $DR_{K}$, the monoid of Deligne-Ribet of $K$, is closely linked to the monoid $\mathcal{I}_{K}$ and is playing a crucial role in the structure of Bost-Connes systems as shown in \cite{bora}, it would be interesting and difficult to try to compute the category of points of the presheaf topos associated to the small category with only one object $\star$ and the arrows indexed by the elements of $DR_{K}$ and the law of composition of arrows given by the law of the monoid $DR_{K}$, the delicate thing will be to put an adequate topology on this category and compute the points. One could hope that it could provide a link and maybe a better understanding between Bost-Connes systems and arithmetic sites.

It would be also interesting to see if it is possible to develop an analogue of the arithmetic site in the case of function fields on a finite field and see if the already existing proof of Weil and Grothendieck for the analogue of the Riemann hypothesis can be done also in this framework.

In march 2016 in \cite{scalingcr} and in \cite{scaling}, A.Connes and C.Consani constructed by extensions of scalar a scaling site for $(\widehat{\mathbb{N}^{\times}},\mathbb{Z}_{\max})$ and so showed that the adèle class space of $\mathbb{Q}$ which is so important in the spectral interpretation of the zeroes of the Riemann z\^{e}ta function admits a natural structure of tropical curve. In the future, I intend to build similar scaling sites for more general number fields. Also regarding the scaling site, it would be interesting to see if the geometric morphism $\text{Spec}(\mathbb{Z})$ and the arithmetic site $(\widehat{\mathbb{N}^{\times}},\mathbb{Z}_{\max})$, could be extended from the Arakelov compactification of $\text{Spec}(\mathbb{Z})$ to the scaling site and see if a similar situation occurs also for more general number fields. In order to achieve this, one would have to use the formalism of $\mathbb{S}$-algebras developped by A.Connes and C.Consani in \cite{gamma}.

\subsection{Acknowledgements}
I would like to thank Eric Leichtnam, my advisor, for his patience, his advice and the freedom he gave me in the orientation of my research. I would also like to thank Caterina Consani for her encouragements, her advice and the interest she had for my work all along my PhD. I would like finally to express my deep gratitude towards Alain Connes for his inspiring advice and for pointing it out a mistake in the first version of the theorem \ref{theo:intsp} and the time he took to help me correcting it.

\section{Notations}
The starting point of Alain Connes and Caterina Consani's construction is the topos $\widehat{\IN ^{\star}}$. Here we shall use the topos $\widehat{\mathcal{O}_{K}}$ where :
\begin{itemize}
\item[$\bullet$] $K$ is a number field whose ring of integers $\mathcal{O}_{K}$ is principal
\item[$\bullet$] $\mathcal{O}_{K}$ is a written shortcut for the little category which has only one object noted $\star$ and arrows indexed by the elements of $\mathcal{O}_{K}$ and the law of composition of arrows is determined by the multiplication law of $\mathcal{O}_{K}$ ($\mathcal{O}_{K}$ is a mono\"{i}d for the multiplication law)
\item[$\bullet$]Let denote $\widehat{\mathcal{O}_{K}}$ the presheaf topos associated to the small category $\mathcal{O}_{K}$, ie the category of contravariant functors from the small category $\mathcal{O}_{K}$ to $\mathfrak{Sets}$ the category of sets.
\item[$\bullet$] $\mathcal{U}_K$ is the set of the units of $\mathcal{O_{K}}$ the ring of integers of $K$
\item[$\bullet$] $\mathbb{S}^{1}$ the circle, ie the set of complex number with modulus equal to $1$
\end{itemize}
\section{Geometric points of $\widehat{\mathcal{O}_{K}}$}
As recalled by Alain Connes and Caterina Consani in \cite{ccsitearith} and proved by MacLane and Moerdijk in \cite{mlm} : in topos theory, the category of geometric points of a presheaf topos $\widehat{\mathcal{C}}$, with $\mathcal{C}$ being a small category, is canonically equivalent to the category of covariant flat functors from $\mathcal{C}$ to $\mathfrak{Sets}$. Let us also recall that a covariant flat functor $F:\mathcal{C}\to\mathfrak{Sets}$ is said to be flat if and only if it is filtering which means :
\begin{enumerate}
\item $F(C)\neq\emptyset$ for at least one object $C$ of $\mathcal{C}$
\item Given two objects $A$ and $B$ of $\mathcal{C}$ and two elements $a\in F(A)$ and $b\in F(B)$, then there exists an object $Z$ of $\mathcal{C}$, two morphisms $u:Z\to A$, $v:Z\to B$ and an element $z\in F(Z)$ such that $F(u)z=a$ and $F(v)z=b$
\item Given two objects $A$ and $B$ of $\mathcal{C}$, two arrows $u,v:A\to B$ and $a\in F(A)$ with $F(u)a=F(v)a$, then there exists an object $Z$ of $\mathcal{C}$, an arrow $w:Z\to A$ and an element $z\in F(Z)$ such that $F(w)z=a$ and $u\circ w =v\circ w \in\text{Hom}_{\mathcal{C}}(Z,B)$.
\end{enumerate}
Here in the case of $\widehat{\mathcal{O}_{K}}$, we deduce that a covariant functor $F:\mathcal{O}_{K}\to\mathfrak{Sets}$ is flat if and only if 
\begin{enumerate}
\item $X:=F(\star )$ is a non empty set
\item Given two elements $a,b\in X$, then there exists $u,v\in \mathcal{O}_{K}$ and $z\in X$ such that $F(u)z=a$ and $F(v)z=b$
\item Given two elements $u,v\in\mathcal{O}_{K}$ and $a\in X$ with $F(u)a=F(v)a$, then there exists $w\in\mathcal{O}_{K}$ and $z\in X$ such that $F(w)z=a$ and $u\times w =v\times w \in\mathcal{O}_K$.
\end{enumerate}
Then we have :
\begin{theo}\label{theo:ptsfirst}
Let $F:\mathcal{O}_{K}\to\mathfrak{Sets}$ be a flat covariant functor. Then $X:=_{def} F(\star )$ can be naturally endowed with the structure of an $\mathcal{O}_{K}$-module which is isomorphic (not in a canonical way) to an $\mathcal{O}_{K}$-module included in $K$
\end{theo}
Let us now prove this theorem with a long serie of lemma:

Let $F:\mathcal{O}_{K}\to\mathfrak{Sets}$ be a flat covariant functor.

Let us denote $X:=F(\star)$, the image by $F$ in $\mathfrak{Sets}$ of $\star$ the only object of the small category $\mathcal{O}_{K}$.

\begin{lemm}
The group law $+$ of $\mathcal{O}_K$ will induce through $F$ an intern law on $X$
\end{lemm}
\begin{proof}
The group law $+$ of $\mathcal{O}_K$ will induce through $F$ an intern law on $X$ in the following way :

Let $x,\tilde{x}\in X$ be two elements of $X$.

By the property (ii) of the flatness of $F$,

Let $u,v\in\mathcal{O}_K$ and $z\in X$ such that  $F(u)z=x$ and $F(v)z=\tilde{x}$.

Then we take as definition $x+\tilde{x}:=F(u+v)z$.

We must now check that this definition is independent of the choices made for $u,v$ and $z$.

Indeed let $u',v'\in\mathcal{O}_K$ and $z'\in X$ (not necessarily equal to $u,v$ and $z$ respectively) sucht that $F(u')z'=x$ and $F(v')z'=\tilde{x}$.

Then by property (ii) of the flatness of $F$,

Let $\alpha,\alpha '\in\mathcal{O}_K$ and $\hat{z}\in X$ such that $F(\alpha )\hat{z}=z$ and $F(\alpha ')\hat{z}=z'$.

Then $F(u\alpha )\hat{z}=x=F(u'\alpha ')\hat{z}$ and $F(v\alpha )\hat{z}=\tilde{x}=F(v'\alpha ')\hat{z}$.

So by property (iii) of the flatness of $F$.

Let $\beta\in\mathcal{O}_{K}$ and $\gamma\in X$ such that $F(\beta )\gamma =\hat{z}$ and $u\alpha\beta =u'\alpha '\beta$

And let $\tilde{\beta }\in\mathcal{O}_{K}$ and $\gamma '\in X$ such that $F(\tilde{\beta} )\gamma ' =\hat{z}$ and $v\alpha\tilde{\beta }=u'\alpha '\tilde{\beta }$.

From here there are several possibilities : 

\begin{itemize}
\item[$\bullet$]$\beta =0$ and $\tilde{\beta }=0$

Then $F(0)\gamma =\hat{z}=F(0)\gamma '$

And so $z=F(\alpha )\hat{z}=F(0)\gamma$ and $z'=F(\alpha ')\hat{z}=F(0)\gamma '$

So finally $F(u+v)z=x+\tilde{x}=F((u+v)0)\gamma =F(0)\gamma =\hat{z}=F(0)\gamma '=F((u'+v')0)\gamma '=F(u'+v')z'$
\item[$\bullet$]$\beta =0$ and $\tilde{\beta }\neq 0$ or $\beta\neq 0$ and $\tilde{\beta }=0$  (as $\beta$ and $\tilde{\beta }$ have symmetric roles we will just look the case $\beta =0$ and $\tilde{\beta }\neq 0$)

Then $F(0)\gamma =\hat{z}=F(\tilde{\beta })\gamma '$

And so $z=F(\alpha )\hat{z}=F(0)\gamma$.

So $x=F(u)z=F(0)\gamma$ and $\tilde{x}=F(v)z=F(0)\gamma$.

But since $x=F(u')z'$ and $\tilde{x}=F(v')z'$, by property (ii) of the flatness of $F$,

Let $\lambda ,\mu\in\mathcal{O}_{K}$ and $z''\in X$ such that $F(\lambda )z''=\gamma$ and $F(\mu )z''=z'$.

So $F(0)z''=x=F(\mu '\mu)z''$.

And so by property (iii) of flatness of $F$, let $\nu\in\mathcal{O}_K$ and $z'''\in X$ such that $F(\nu )z'''=z''$ and $0\nu =\mu '\mu\nu$ and then 

\begin{itemize}
\item[$\star$]either $\mu\nu =0$

Then $z'=F(0)z''$ and so $x+\tilde{x}=F(u+v)z=F(0)\gamma$ and $F(u'+v')z'=F((u'+v')0)z''=F(0)z''$.

So by property (ii) of flatness of $F$, let $l,m\in\mathcal{O}_K$ and $\bar{z}\in X$ such that $F(l)\bar{z}=\gamma$ and $F(m)\bar{z}=z''$

And so finally $F(u+v)z=x+\tilde{x}=F(0)\gamma =F(0)\bar{z}=F(0m)\bar{z}=F(0)z''=F(u'+v')z'$.
\item[$\star$]either $u'=0$

Then $x=F(0)z'$

So $F(u'+v')z'=F(v')z'=\tilde{x}$

And we will still have $x+\tilde{x}=F(0)\gamma =\tilde{x}$

So finally $x+\tilde{x}=F(u'+v')z'$.
\end{itemize}
\item[$\bullet$]$\beta\neq 0$ and $\beta '\neq 0$

Then $u\alpha =u'\alpha '$ and $v\alpha =v'\alpha '$

So finally $x+\tilde{x}=F(u+v)z=F(u\alpha +v\alpha )\hat{z}=F(u'\alpha '+v'\alpha ')\hat{z}=F(u'+v')z$
\end{itemize}

Therefore the definition of the law $+$ on $X$ is independent of the choices made.

\end{proof}

In fact, more is true: 
\begin{lemm}
The set $(X,+)$ with the intern law defined as before is an abelian group.
\end{lemm}

\begin{proof}
Let us now prove $(X,+)$ that is an abelian group.
\begin{itemize}
\item[$\bullet$]let us first check the associativity of $+$

It follows from the associativity of $+$ on $\mathcal{O}_{K}$ and more precisely : 

let $x,x',x''\in X$, let us now apply the property (ii) of the flatness of $F$ two times in a row,

let us then take $a,a',a''\in\mathcal{O}_{K}$ and $z\in X$ such that $x=F(a)z$, $x'=F(a')z$ and $x''=F(a'')z$.

Then $x+(x'+x'')=x+F(a'+a'')z=F(a)z+F(a'+a'')z=F(a+(a'+a''))z=F((a+a')+a'')z=(x+x')+x''$.

So the law $+$ is indeed associative.
\item[$\bullet$]let us now check the commutativity of $+$

Here again it follows from the commutativity of $+$ on $\mathcal{O}_{K}$ and more precisely :

let us then take $a,a'\in\mathcal{O}_{K}$ and $z\in X$ such that $x=F(a)z$ and  $x'=F(a')z$.

Then $x+x'=F(a+a')z=F(a'+a)z=x'+x$.

So the law $+$ is indeed commutative.
\item[$\bullet$]Let us now find the neutral element of $(X,+)$.

We denote $0_{X}:=F(0)x$ where $x\in X$ is any element of $X$.

Let us first show that $0_X$ is well defined :

Let $x,x'\in X$, by property (ii) of the flatness of $F$, let $a,a'\in\mathcal{O}_{K}$ and $z\in X$ such that $x=F(a)z$ and $x'=F(a')z$.

Then $F(0)x=F(0a)z=F(0)z=F(0a')z=F(0)x'$.

So $0_{X}$ is indeed well defined, let us now show that it is the neutral element of $(X,+)$.

Let $y\in X$, by property (ii) of flatness of $F$, let $\tilde{a},\tilde{a}'\in\mathcal{O}_{K}$ and $\tilde{z}\in X$ such that $x=F(\tilde{a})\tilde{z}$ and $0_{X}=F(\tilde{a}')\tilde{z}$.

But then by the definition of $0_{X}$, we have $F(\tilde{a}')\tilde{z}=0_{X}=F(0)\tilde{z}$.

So by property (iii) of flatness of $F$, let $w\in\mathcal{O}_K$ and $\hat{z}\in X$ such that $\tilde{z}=F(w)\hat{z}$ and $\tilde{a}'w=0w=0$.

All in all we have $y=F(\tilde{a}w)\hat{z}$ and $0_{X}=F(0)\hat{z}$.

And so $y+0_{X}=F(\tilde{a}w+0)\hat{z}=F(\tilde{a}w)\hat{z}=y$ and by commutativity $0_{X}+y=y+0_{X}=y$.

So $0_{X}$ is the neutral element of $(X,+)$.
\item[$\bullet$]Let us finally show that each element of $X$ admits a symmetric for the law $+$.

Let $x\in X$, as above by property (ii) and (iii) of flatness of $F$, let $a\in\mathcal{O}_K$ and $z\in X$ such that $x=F(a)z$ and $0_{X}=F(0)z$.

We denote $-x:=F(-a)z$, then $x+(-x)=F(a+(-a))z=F(0)z=0_X$ and by commutativity $(-x)+x=x+(-x)=0_X$.

But before concluding we must check that our definition of $-x$ is independent of the choices made for $a$ and $z$.

Let $a'\in\mathcal{O}_K$ and $z'\in X$ such that $x=F(a')z'$ and $0_{X}=F(0)z'$ too.

By property (ii) of flatness of $F$, let $b,b'\in\mathcal{O}_{K}$ and $z''\in X$ such that $z=F(b)z''$ and $z'=F(b')z''$.

Then $F(ab)z''=x=F(a'b')z''$, so by property (iii) of flatness of $F$, let $c\in\mathcal{O}_{K}$ and $z'''\in X$ such that $z''=F(c)z'''$ and $abc=a'b'c$.

So $-x=F(-a)z=F(-abc)z'''=F(-a'b'c)z'''=F(-a')z'$.

So $-x$ is well defined and is the symmetric of $x$ for the law $+$
\end{itemize}

Therefore $(X,+)$ is an abelian group. 
\end{proof}

In fact we have a better result :
\begin{lemm}
We can endow $X$ with the structure of an $\mathcal{O}_K$-module.
\end{lemm} 
 
\begin{proof} 
Let us now show that we can endow $X$ with the structure of an $\mathcal{O}_K$-module.
Let us define $\bullet$ the action law by $\bullet\left\{\begin{aligned}\mathcal{O}_{K}\times X&\longrightarrow & X\\ (\alpha ,x) & \longmapsto & \alpha\bullet x:=F(\alpha u )z\end{aligned}\right.$

Where $u\in\mathcal{O}_K$ and $z\in X$ are such that $x=F(u)z$ and $0_{X}=F(0)z$ (by property (ii) and (iii) of flatness of $F$ such elements always exist).

Let us first check that $\bullet$ is well defined, ie independent of the choices made to define it:

Let $(\alpha ,x)\in\mathcal{O}_{K}\times X$, by property (ii) and (iii) of flatness of $F$, let $u,u'\in\mathcal{O}_{K}$ and $z,z'\in X$ such that $F(u)z=x=F(u')z'$ and $F(0)z=0_{X}=F(0)z'$.

To show that $\bullet$ is well defined, let us show that $F(\alpha u)z=F(\alpha u')z'$.

Since $F(u)z=x=F(u')z'$, by property (ii) and (iii) of $F$, let $w,w'\in\mathcal{O}_{K}$ and $\hat{z}$, $\hat{w}\in\mathcal{O}_{K}$ and $\hat{\hat{z}}\in X$ such that $z=F(w)\hat{z}$, $z'=F(w')\hat{z}$, $\hat{z}=F(\hat{w)})\hat{\hat{z}}$ and $uw\hat{w}=u'w'\hat{w}$.

Then $F(\alpha u)z=F(\alpha uw\hat{w})\hat{\hat{z}}=F(\alpha u'w'\hat{w})\hat{\hat{z}}=F(\alpha u')z'$.

So $\bullet$ is indeed well defined.

Let us now check the following relations :
\begin{itemize}
\item[$\star$]$\forall\alpha\in\mathcal{O}_{K},\forall (x,y)\in X^{2}, \alpha\bullet(x+y)=\alpha\bullet x+\alpha\bullet y$

Indeed, let $\alpha\in\mathcal{O}_{K}$ and $x,y\in X$,

by property (ii) of flatness of $F$, let $u,v\in\mathcal{O}_{K}$ and $z\in X$ such that $x=F(u)z$ and $y=F(v)z$.

Then $\alpha\bullet (x+y)=\alpha\bullet (F(u+v)z)=F(\alpha (u+v))z=F(\alpha u+\alpha v)z=F(\alpha u)z+F(\alpha v)z=\alpha\bullet x+\alpha\bullet y$.

\item[$\star$]$\forall (\alpha ,\beta )\in\mathcal{O}_{K}^{2},\forall x\in X, (\alpha +\beta )\bullet x=\alpha\bullet x+\beta\bullet x$

Indeed, let $\alpha ,\beta\in\mathcal{O}_{K}$ and $x\in X$, 

By property (ii) and (iii) of flatness of $F$, let $u\in\mathcal{O}_{K}$ and $z\in X$ such that $x=F(u)z$ and $0_{X}=F(0)z$.

Then $(\alpha +\beta )\bullet x=F((\alpha +\beta)u)z=F(\alpha u+\beta u)z=F(\alpha u)z+F(\beta u)z=\alpha\bullet x+\beta\bullet x$

\item[$\star$]$\forall (\alpha ,\beta )\in\mathcal{O}_{K}^{2},\forall x\in X,\alpha\bullet (\beta\bullet x)=(\alpha\beta )\bullet x$

Indeed, let $\alpha ,\beta\in\mathcal{O}_{K}$ and $x\in X$, by property (ii) and (iii) of flatness of $F$, let $u\in\mathcal{O}_{K}$ and $z\in X$ such that $x=F(u)z$ and $0_{X}$.

Then $\alpha\bullet (\beta\bullet x)=\alpha\bullet (F(\beta u)z)=F(\alpha\beta u)z=(\alpha\beta )\bullet F(u)z=(\alpha\beta )\bullet x$

\item[$\bullet$]$\forall x\in X, 1\bullet x=x$

Indeed, let $x\in X$,  by property (ii) and (iii) of flatness of $F$, let $u\in\mathcal{O}_{K}$ and $z\in X$ such that $x=F(u)z$ and $0_{X}$.

Then $1\bullet x=1\bullet F(u)z=F(1\times u)z=F(u)z=x$
\end{itemize}
So finally $X$ is indeed an $\mathcal{O}_{K}$-module.

\end{proof}

We end the proof of the theorem \ref{theo:ptsfirst} thanks to the following lemma:
\begin{lemm}
The $\mathcal{O}_{K}$-module $X$ is isomorphic (in a non canonical way) to an $\mathcal{O}_{K}$-module included in $K$.
\end{lemm}

\begin{proof}
We have two possibilities

\begin{itemize}
\item[$\star$]$X=\{0_{X}\}$ then obviously $X\simeq\{0_{K}\}$
\item[$\star$]$\{0_{X}\}\subsetneqq X$

Then let us take $x\in X\backslash\{ 0_{X}\}$ and let us thus note $j_{X,x}\left\{\begin{aligned} X&\longrightarrow & K\\ \tilde{x} & \longmapsto & \frac{\tilde{k}}{k}\end{aligned}\right.$ where we have $k,\tilde{k}\in\mathcal{O}_{K}$ and $z\in X$ such that $x=F(k)z$ and $\tilde{x}=F(\tilde{k})z$ (there always exists such elements by property (ii) of flatness of $F$).

Let us first show that $j_{X,x}$ is well defined.

Let $\tilde{x}\in X$, then by property (ii) of flatness of $F$, let $k,\tilde{k},k',\tilde{k}'\in\mathcal{O}_{K}$ and $z,z'\in X$ such that $F(k)z=x=F(k')z'$ and $F(\tilde{k})z=\tilde{x}=F(\tilde{k}')z'$.

According to what we have shown earlier on $0_{X}$, $k\neq 0$ and $k'\neq 0$ (otherwise we would have $x=0_{X}$ which is impossible).

So now to show that $j_{X,x}$ is well defined, we only have left to show that $\frac{\tilde{k}}{k}=\frac{\tilde{k}'}{k'}$.

By property (ii) and (iii) of flatness of $F$, let $w,w',\hat{w}\in\mathcal{O}_{K}$ and $\hat{z},\hat{\hat{z}}\in X$ such that $z=F(w)\hat{z},z'=F(w')\hat{z},\hat{z}=F(\hat{w})\hat{\hat{z}}$ and $kw\hat{w}=k'w'\hat{w}$, so since $k\neq 0$, $w\hat{w}=\frac{k'}{k}w'\hat{w}$.

So $0_{X}\neq x=F(kw\hat{w})\hat{\hat{z}}=F(k'w'\hat{w})\hat{\hat{z}}$.

So $kw\hat{w}\neq 0$ and $k'w'\hat{w}\neq 0$, so $w'\hat{w}\neq 0$.

We also have $F(\tilde{k}w\hat{w})\hat{\hat{z}}=\tilde{x}=F(\tilde{k}'w'\hat{w})\hat{\hat{z}}$.

So by property (iii) of flatness of $F$, let $\hat{\hat{w}}\in\mathcal{O}_{K}$ and $\hat{\hat{\hat{z}}}\in X$ such that $\hat{\hat{z}}=F(\hat{\hat{w}})\hat{\hat{\hat{z}}}$ and $\tilde{k}w\hat{w}\hat{\hat{w}}=\tilde{k}w\hat{w}\hat{\hat{w}}$.

So $\tilde{k}\frac{k'}{k}w'\hat{w}\hat{\hat{w}}=\tilde{k}'w'\hat{w}\hat{\hat{w}}$.

And since $w'\hat{w}\neq 0$ and $k'\neq 0$, we thus have $\frac{\tilde{k}}{k}\hat{\hat{w}}=\frac{\tilde{k}'}{k'}\hat{\hat{w}}$.

And $\hat{\hat{w}}\neq 0$ because otherwise we would have $\hat{\hat{z}}=F(0)\hat{\hat{\hat{z}}}$ and so $x=F(kw\hat{w}0)\hat{\hat{\hat{z}}}=0_{X}$ which is impossible.

And so we get that $\frac{\tilde{k}}{k}=\frac{\tilde{k}'}{k'}$ and so $j_{X,x}$ is well defined.

Let us now check that $j_{X,x}$ is a linear map from $X$ to $K$.

Let $x',x''\in X$ and $\lambda\in\mathcal{O}_{K}$. By property (ii) of flatness of $F$ there exists $k,k',k''\in\mathcal{O}_{K}$ and $z\in X$ such that $X=F(k)z$, $x'=F(k')z$ and $x''=F(k'')z$ and as $x\neq 0$, $k\neq 0$.

Then $\lambda x'+x''=F(\lambda k'+k'')z$.

And so $j_{X,x}(\lambda x'+x'')=\frac{\lambda k'+k''}{k}=\lambda\frac{k'}{k}+\frac{k''}{k}=\lambda j_{X,x}(x')+j_{X,x}(x'')$ and so $j_{X,x}$ is linear.

Let us now show that $j_{X,x}$ is injective.

Indeed, let $x'\in X$ such that $j_{X,x}(x')=0$, then by property (ii) of flatness of $F$, let $k,k'\in\mathcal{O}_{K}$ and $z\in X$ such that $x=F(k)z$ and $x'=F(k')z$.

Since $x\neq 0$, we have $k\neq 0$ and then $0=j_{X,x}(x')=\frac{k'}{k}$.

So $k'=0$ and finally $x'=F(0)z=0_{X}$.

So finally we have $X\simeq\text{Im}(j_{X,x})$ and of course the dependance in $x$ makes it non canonical.
\end{itemize}
\end{proof}
Thanks to the theorem \ref{theo:ptsfirst} we get that :
\begin{theo}\label{theo:ptsscd}
The category of (geometric) points of the topos $\widehat{\mathcal{O}_{K}}$ is canonically equivalent to the category of sub $\mathcal{O}_{K}$-modules of $K$ and morphisms of $\mathcal{O}_{K}$-modules.
\end{theo}
\begin{proof}
By theorem VII.5.2bis p 382 of \cite{mlm} the category of geometric ponts of $\widehat{\mathcal{O}_{K}}$ and natural transformations is equivalent to the category $\mathfrak{Flat}(\mathcal{O}_{K})$ of the covariant flat functors from the small category $\mathcal{O}_{K}$ to $\mathfrak{Sets}$ and natural transformations.

Now we just have to prove that $\mathfrak{Flat}(\mathcal{O}_{K})$ is equivalent to the category $\mathcal{O}_{K}-\mathfrak{Mod}_{\simeq\subset K}$ of $\mathcal{O}_{K}$-modules isomorphic to sub $\mathcal{O}_{K}$-modules of $K$ and morphisms of $\mathcal{O}_{K}$-modules.

First we can define \[\mathcal{E}\left\{\begin{aligned} \mathfrak{Flat}(\mathcal{O}_{K})&\longrightarrow & \mathcal{O}_{K}-\mathfrak{Mod}_{\simeq\subset K}\\ F:\mathcal{O}_{K}\to\mathfrak{Sets} & \longmapsto & F(\star ) \\ \Phi :F\to G & \longmapsto & \Phi :F(\star )\to G(\star )\end{aligned}\right.\]

Let us first check that $\mathcal{E}$ is well defined :
\begin{itemize}
\item[$\bullet$]on the objects $\mathcal{E}$ is well defined as shown by the last lemma (1.1)

\item[$\bullet$]let $\Phi$ be a natural transformation from $F$ to $G$ with $F,G:\mathcal{O}_{K}\to\mathfrak{Sets}$ two flat covariant functors from the small category $\mathcal{O}_{K}$ to $\mathfrak{Sets}$, then by definition $\Phi$ can also be seen as an application from $F(\star )$ to $G(\star )$ since the small category $\mathcal{O}_{K}$ has only one object noted $\star$.

Let us now show that $\Phi:F(\star )\to G(\star )$ is linear.

Let $\lambda\in\mathcal{O}_{K}$ and $x,y\in F(\star )$, then by property (ii) of flatness of $F$, let $u,v\in\mathcal{O}_{K}$ and $z\in F(\star )$ such that $x=F(u)z$ and $y=F(v)z$.

Then since $\Phi$ is more than a mere application from $F(\star )$ to $G(\star )$ but also a natural transformation from $F$ to $G$, we have $\Phi (x)=G(u)\Phi (z)$ and $\Phi (y)=G(v)\Phi  (z)$.

And so $\Phi (\lambda x+y)=\Phi (\lambda F(u)z+F(v)z)=\Phi (F(\lambda u+v)z)=G(\lambda u+v)\Phi (z)=\lambda G(u)\Phi (z)+G(v)\Phi (z)=\lambda\Phi (x)+\Phi (y)$.

So it means that $\Phi :F(\star )\to G(\star )$ is indeed linear.
\end{itemize}
So $\mathcal{E}$ is well defined and is in fact a covariant functor, indeed for all $F$ flat covariant functor from the small category $\mathcal{O}_{K}$ to $\mathfrak{Sets}$ we have almost by definition  $\mathcal{E}(\text{id}_{F})=\text{id}_{F(\star )}$ and also for any $F,G,H:\mathcal{O}_{K}\to\mathfrak{Sets}$ three flat covariant functors from the small category $\mathcal{O}_{K}$ to $\mathfrak{Sets}$ and any $\Phi$ be a natural transformation from $F$ to $G$ and $\Psi$ be a natural transformation from $G$ to $H$, we have then immediately $\mathcal{E}(\Psi\circ\Phi )=\mathcal{E}(\Psi )\circ\mathcal{E}(\Phi )$.

Let us now show that $\mathcal{E}$ is fully faithful. Indeed let $F,G:\mathcal{O}_{K}\to\mathfrak{Sets}$ two flat covariant functors from the small category $\mathcal{O}_{K}$ to $\mathfrak{Sets}$, as the small category $\mathcal{O}_{K}$ has only one object we deduce when we look closely at the definitions that it is rigorously the same to consider a natural transformation from $F$ to $G$ than a linear application from $F(\star )$ to $G(\star )$.

Let us now finally check that $\mathcal{E}$ is essentially surjective then $\mathcal{E}$ will induce an equivalence of categories. As for essential surjectivity we work up to isomorphism we can directly take $X$ an $\mathcal{O}_{K}$-module included in $K$.

Let us then define $F$ as follows \[F\left\{\begin{aligned} \mathcal{O}_{K}&\longrightarrow & \mathfrak{Sets}\\ \star & \longmapsto & X \\ \lambda\in\mathcal{O}_{K} & \longmapsto & \mu_{\lambda}:X\to X;x\mapsto\lambda x\end{aligned}\right.\]

$F$ is obviously a covariant functor. One now has to check that it is flat.

Of course $X\neq\emptyset$ so property (i) of flatness is satified by $F$.

Let us now check that property (ii) is satisfied: 

let $x,y\in X$, as $X\subset K$, then :
\begin{itemize}
\item[$\bullet$]first case, $x=0=y$, then we have $x=0\times 0$ and $y=0\times 0$

\item[$\bullet$]second case ($x=0$ and $y\neq 0$) or ($y=0$ and $x\neq 0$), without loss of generality, let us just study the case $x=0$ and $y\neq 0$, then $x=0y$ and $y=1y$

\item[$\bullet$]third case $x\neq 0$ and $y\neq 0$

let us write $x$ and $y$ as irreducible fractions : $x=\frac{x_{n}}{x_{d}}$ and $y=\frac{y_{n}}{y_{d}}$.

Let us note $<x,y>$ the $\mathcal{O}_{K}$-module generated by $x$ and $y$ and $t:=x_{d}y_{d}$.

Then $t<x,y>$ is an $\mathcal{O}_{K}$-module included in $\mathcal{O}_K$ ie an ideal of $\mathcal{O}_K$.

Since $\mathcal{O}_K$ is principal, let us note $\delta\in\mathcal{O}_K$ a generator of $t<x,y>$.

So let $u,v\in\mathcal{O}_{K}$ such that $tx=u\delta$ and $ty=v\delta$.

Now let us note $z:=\frac{\delta }{t}$, and so we have $x=uz$, $y=vz$ and $z\in <x,y>\subset X$
\end{itemize}

Let us finally check that property (iii) is satisfied by $F$ : let $a\in X$ and let $u,v\in\mathcal{O}_{K}$ such that $ua=va$.
\begin{itemize}
\item[$\bullet$]First case $a=0$, then let us take $w=0\in\mathcal{O}_{K}$ and $z=0=a\in X$, and so $wz=a$ and $uw=vw$.

\item[$\bullet$]Second case $a\neq 0$, then let us take $w=1\in\mathcal{0}_{K}$ and $z=a\in X$, and so $wz=a$ and $uw=vw$.
\end{itemize}
And so the theorem \ref{theo:ptsscd} is proved.
\end{proof}

\section{Adelic interpretation of the geometric points of $\widehat{\mathcal{O}_{K}}$}
Let us denote $\mathbb{A}^{f}_{K}:=_{def} \prod^{'}_{\mathfrak{p} \enskip\text{prime}} K_{\mathfrak{p}}$  (restricted product) the ring of finite adeles of $K$ and $\widetilde{\mathcal{O}_{K}}:=\prod_{\mathfrak{p}\enskip \text{prime}} \mathcal{O}_{\mathfrak{p}}$ its maximal compact subring.
 
Let us recall the definition of Dedekind's complementary module (or inverse different), it is the fractionnal ideal $\mathfrak{D}_{\mathcal{O}_{K}}:=\{x\in K / \text{tr}(x.\mathcal{O}_{K})\in\mathbb{Z}\}$ of $\mathcal{O}_{K}$  denoted $\mathfrak{D}_{\mathcal{O}_{K}}$.

Then we have the following lemmas :
\begin{lemm}\label{lemm:dual1}
A closed sub-$\mathcal{O}_{K}$-module of $\widetilde{\mathcal{O}_{K}}$ is an ideal of the ring $\widetilde{\mathcal{O}_{K}}$.
\end{lemm}
\begin{proof}
Let $J$ be a closed sub-$\mathcal{O}_{K}$-module of $\widetilde{\mathcal{O}_{K}}$.

To prove that $J$ is an ideal of the ring $\widetilde{\mathcal{O}_{K}}$, we only have to check that $\forall\alpha\in\widetilde{\mathcal{O}_{K}},\forall j\in J, \alpha j\in J$.

Since $J$ is a sub-$\mathcal{O}_{K}$-module of $\widetilde{\mathcal{O}_{K}}$, we already have that $\forall\alpha\in\mathcal{O}_{K},\forall j\in J, \alpha j\in J$.

Now let $\alpha\in\widetilde{O}_{K}$.

Since $\mathcal{O}_{K}$ is dense in $\widetilde{\mathcal{O}_{K}}$ (thanks to strong approximation theorem), let $(\alpha_{n})_{n\in\mathbb{N}}\in(\mathcal{O}_{K})^{\mathbb{N}}$ such that $\alpha_{n}\xrightarrow[n\to\infty]{} \alpha$.

Then, since we have $\alpha_{n}j\xrightarrow[n\to\infty]{} \alpha j$ and $\forall n\in\mathbb{N}, \alpha_{n}j\in J$ and also $J$ closed, we get that $\alpha j\in J$.

Therefore $J$ is an ideal of the ring $\widetilde{\mathcal{O}_{K}}$.
\end{proof}
\begin{lemm}\label{lemm:dual2}
Let $J$ be a closed sub-$\mathcal{O}_{K}$-module of $\widetilde{\mathcal{O}_{K}}$. For each prime ideal $\mathfrak{p}$ of $\mathcal{O}_{K}$ the projection $\pi_{\mathfrak{p}}(J)\subset \mathcal{O}_{\mathfrak{p}}$ coincides with the intersection $(\{0\}\times\mathcal{O}_{K,\mathfrak{p}})\cap J$ and is a closed ideal of $J_{\mathfrak{p}}\subset\mathcal{O}_{K,\mathfrak{p}}$, moreover one has $x\in J \Leftrightarrow \forall \mathfrak{p}\quad\text{prime}, \pi_{\mathfrak{p}}(x)\in J_{\mathfrak{p}}$
\end{lemm}
\begin{proof}
Let $J$ be a closed sub-$\mathcal{O}_{K}$-module of $\widetilde{\mathcal{O}_{K}}$.

Thanks to the preceding lemma, $J$ is an ideal of the ring $\widetilde{\mathcal{O}_{K}}$.

Let $\mathfrak{p}$ be a prime ideal of $\mathcal{O}_{K}$ and let us note $a_{\mathfrak{p}}$ the finite adele which is zero everywhere except in $\mathfrak{p}$ where it is equal to 1.

Then we have that $\pi_{\mathfrak{p}}(J)=a_{\mathfrak{p}}J\subset J$ since $J$ is an ideal.

And by definition $\pi_{\mathfrak{p}}(J)\subset\mathcal{O}_{K,\mathfrak{p}}\simeq\{0\}\times\mathcal{O}_{K,\mathfrak{p}}$.

So $\pi_{\mathfrak{p}}(J)\subset\left(\{0\}\times\mathcal{O}_{K,\mathfrak{p}}\right)\cap J$.

The converse inclusion is obvious, so all in all we have indeed that $$\pi_{\mathfrak{p}}(J)=\left(\{0\}\times\mathcal{O}_{K,\mathfrak{p}}\right)\cap J$$

Now since $\mathcal{O}_{K,\mathfrak{p}}\simeq\{0\}\times\mathcal{O}_{K,\mathfrak{p}}\subset{\widetilde{O}_{K}}$ and since $J$ is an ideal of $\widetilde{\mathcal{O}_{K}}$, we have that $\mathcal{O}_{K,\mathfrak{p}}.J\subset J$.

So $\mathcal{O}_{K,\mathfrak{p}}.\pi_{\mathfrak{p}}=\mathcal{O}_{K,\mathfrak{p}}.\left(\left(\{0\}\times\mathcal{O}_{K,\mathfrak{p}}\right)\cap J\right)\subset\left(\left(\{0\}\times\mathcal{O}_{K,\mathfrak{p}}\right)\cap J\right)=\pi_{\mathfrak{p}}(J)$.

And  since $\{0\}\times\mathcal{O}_{K,\mathfrak{p}}$ and $J$ are closed subgroups of $\widetilde{\mathcal{O}_{K}}$, we get that $\pi_{\mathfrak{p}}(J)$ is a closed subgroup of $\mathcal{O}_{K,\mathfrak{p}}$, and so all in all $\pi_{\mathfrak{p}}(J)$ is a closed ideal of $\mathcal{O}_{K,\mathfrak{p}}$.

Finally the implication $x\in J\Rightarrow\forall\mathfrak{p},\pi_{\mathfrak{p}}(x)\in\pi_{\mathfrak{p}}(J)$, is obvious.

Conversely let $x\in\widetilde{\mathcal{O}_{K}}$ such that $\forall\mathfrak{p},\pi_{\mathfrak{p}}(x)\in\pi_{\mathfrak{p}}(J)$.

Since $\forall\mathfrak{p}, \pi_{\mathfrak{p}}(x)\in\pi_{\mathfrak{p}}(J)=\left(\{0\}\times\mathcal{O}_{K,\mathfrak{p}}\right)\cap J\subset J$.

Let us note for each prime ideal $\mathfrak{p}$ of $\mathcal{O}_{K}$, $a_{\mathfrak{p}}\in\widetilde{\mathcal{O}_{K}}$ the finite adele which is zero everywhere except in $\mathfrak{p}$ where it is equal to 1.

So $x=\sum_{\mathfrak{p}}a_{\mathfrak{p}}\pi_{\mathfrak{p}}\in J$ since $J$ is an ideal of $\widetilde{\mathcal{O}_{K}}$ and the sum is finite because $\pi_{\mathfrak{p}}(x)=0$ almost everywhere because $x\in\widetilde{O}_{K}$.

Therefore we have proved that $x\in J\Leftrightarrow\forall\mathfrak{p},\pi_{\mathfrak{p}}(x)\in\pi_{\mathfrak{p}}(J)$
\end{proof}
\begin{lemm}\label{lemm:dual3}
For any prime ideal $\mathfrak{p}$ of $\mathcal{O}_{K}$, any ideal of $\mathcal{O}_{K,\mathfrak{p}}$ is principal.
\end{lemm}
\begin{proof}
Since $\mathcal{O}_{K,\mathfrak{p}}$ is a complete discrete valuation ring, every ideal of $\mathcal{O}_{K,\mathfrak{p}}$ is of the form $\pi^{n}\mathcal{O}_{K,\mathfrak{p}}$ with $n\in\mathbb{N}$ and $\pi$ an element of valuation $1$.
\end{proof}
With the lemmas \ref{lemm:dual1}, \ref{lemm:dual2} and \ref{lemm:dual3} and Pontrjagin duality, one can show that :
\begin{theo}\label{theo:adlfst}
Any non trivial sub-$\mathcal{O}_K$-module of $K$ is uniquely of the form $H_{a}:=\{q\in K | aq\in\widehat{\mathfrak{D}_{\mathcal{O}_{K}}}\}$ where $a\in\mathbb{A}^{f}_{K}/\widetilde{\mathcal{O}_{K}}^{\times}$ and $\widehat{\mathfrak{D}_{\mathcal{O}_{K}}}$ denotes the profinite completion of the different.
\end{theo}
\begin{proof}
Let us recall that Tate's character for finite adeles is $$\chi_{Tate}:=\left\{\begin{aligned}\mathbb{A}^{f}_{K} & \longrightarrow & \mathbb{S}^{1} \\
\left( x_{\mathfrak{p}}\right) & \longmapsto & \prod_{\mathfrak{p}}\psi_{\text{char}(K_{\mathfrak{p}})}(\text{tr}(x_{\mathfrak{p}}))\end{aligned}\right.$$ 
where for all prime number $p$, $\psi_{p}$ is $\psi_{p}:=\mathbb{Q}_{p}\xrightarrow{\text{can}}\mathbb{Q}_{p}/\mathbb{Z}_{p}
\to\mathbb{Q}/\mathbb{Z}\xrightarrow{\exp(2\pi\imath\bullet )}\mathbb{S}^{1}$.

Then the pairing $\langle k,a\rangle =\chi_{\text{Tate}}(ka), \forall k\in K/\mathcal{O}_{K}, \forall a\in\mathfrak{D}_{\mathcal{O}_{K}}$ identifies $\widehat{\mathfrak{D}_{\mathcal{O}_{K}}}$ with the Pontrjagin dual of $K/\mathcal{O}_{K}$ by a direct application of proposition VIII.4.12 of \cite{weilbasic}.

Let us now prove the theorem.

Let $H$ be a non trivial $\mathcal{O}_K$-module included in $K$. If $H$ is included in $\mathcal{O}_{K}$, then $H$ is an integral ideal and the result is obvious.

Let $H$ be a non trivial $\mathcal{O}_K$-module included in $K$ and containing $\mathcal{O}_{K}$, it is completely determined by its image $\mathcal{H}$ in $K/\mathcal{O}_{K}$.

Then the Pontrjagin duality implies that : $\mathcal{H}=(\mathcal{H}^{\bot})^{\bot}=\{k\in K/\mathcal{O}_{K}/ \forall x\in\mathcal{H}^{\bot}, \langle k,x\rangle =1\}$ and $\mathcal{H}^{\bot}=\{x\in\mathfrak{D}_{\mathcal{O}_{K}}/\forall k\in K/\mathcal{O}_{K}, \langle k,x\rangle =1 \}$.

We also have that $\mathfrak{D}_{\mathcal{O}_{K}}=\prod_{\mathfrak{p}\ni F}\mathcal{O}_{K,\mathfrak{p}}\times\prod_{\mathfrak{p}\in F}\mathfrak{p}^{-n_{\mathfrak{p}}}$ where $F$ is a finite set included in $\text{Spec}(\mathcal{O}_{K})$.

Since $\mathcal{H}^{\bot }\subset\mathfrak{D}_{\mathcal{O}_{K}}$, we have that $\left(\prod_{\mathfrak{p}\in F}\mathfrak{p}^{-n_{\mathfrak{p}}}\right) .\mathcal{H}^{\bot }\subset \widetilde{\mathcal{O}_{K}}$ and $\left(\prod_{\mathfrak{p}\in F}\mathfrak{p}^{-n_{\mathfrak{p}}}\right) .\mathcal{H}^{\bot }$ is also a closed sub-$\mathcal{O}_{K}$-module of $\widetilde{\mathcal{O}_{K}}$.

So thanks to the three last lemmas, we know that a closed sub-$\mathcal{O}_{K}$-module of $\widetilde{\mathcal{O}_{K}}$ can be written in the form $a.\widetilde{\mathcal{O}_{K}}$ with $a\in\widetilde{\mathcal{O}_{K}}$ and this $a$ is unique up to multiplication by an element of $\widetilde{\mathcal{O}_{K}}^{\star}$, so we get the result.
\end{proof}
And so we get that:
\begin{cor}
There is a canonical bijection between the quotient space $\mathbb{A}^{f}_{K}/(K^{\star}(\prod\mathcal{O}_{\mathfrak{p}}^{\star}\times \{1\}))$ and the isomorphisms classes of the (geometric) points of the topos $\widehat{\mathcal{O}_{K}}$
\end{cor}
\begin{proof}
Thanks to theorem \ref{theo:ptsscd} and \ref{theo:adlfst}, any non trivial point of the topos $\widehat{\mathcal{O}_{K}}$ is  obtained from a $\mathcal{O}_{K}$-module $H_{a}$ of rank 1 included in $K$ where $a\in\mathbb{A}^{f}_{K}/\widetilde{\mathcal{O}_{K}}^{\times}$. Two elements $a,b\in\mathbb{A}^{f}_{K}/\widetilde{\mathcal{O}_{K}}^{\times}$ determine isomorphic $\mathcal{O}_{K}$-modules $H_{a}$ and $H_{b}$ of rank 1 included in $K$. An isomorphism between $\mathcal{O}_{K}$-modules of rank 1 included in $K$ is given by the multiplication by an element   $k\in K^{\star}$ so that $H_{b}=k.H_{a}$ and then by theorem 3.1, $a=kb$ in $\mathbb{A}^{f}_{K}/\widetilde{\mathcal{O}_{K}}^{\times}$. Therefore the result is proved.
\end{proof}

\section{The geometric points of the arithmetic site for an imaginary quadratic field with class number 1}

In the rest of this paper we will restrict our attention to the simple case where we have only a finite group of symetries, therefore in the sequel we assume that $K$ is an imaginary quadratic number field. Moreover we assume its class number is 1.

Let us denote $\widetilde{\mathcal{C}_{\mathcal{O}_{K}}}$ the set of $\emptyset$, $\{0\}$ and the convex polygons of the real plane (identified with $\mathbb{C}$) with non empty interior, with center $0$, whose summits have affix in $\mathcal{O}_{K}$ and who are invariant by the action of the elements $\mathcal{U}_{K}$.

\begin{lemm}
$\left(\widetilde{\mathcal{C}_{\mathcal{O}_{K}}},\text{Conv}(\bullet \cup \bullet), +\right)$ is an (idempotent) semiring whose neutral element for the first law is $\emptyset$ and for the second law is $\{0\}$
\end{lemm}
\begin{proof}
A convex polygon is the convex hull of a finite number of points. Let $\mathcal{P}:=\text{Conv}\left(\bigcup_{i\in [|1,n|]}P_{i}\right)$ and $\widetilde{\mathcal{P}}:=\text{Conv}\left(\bigcup_{j\in [|1,m|]}\tilde{P}_{j}\right)$ be two polygones whose summits have their affix in $\mathcal{O}_{K}$ and which are invariant by the action of the elements $\mathcal{U}_{K}$ (here $P_{i}$ and $\tilde{P}_{j}$ mean both the points of the plane and their affixes).

Then $\text{Conv}(\mathcal{P}\cup\widetilde{\mathcal{P}})=\text{Conv}(\{P_{i},i\in [|1,n|]\})$ and $\mathcal{P}+\widetilde{\mathcal{P}}=\text{Conv}(\{P_{i}+\tilde{P_{j}},(i,j)\in [|1,n|]\times[|1,m|]\})$ are also polygons. From those formulae one sees also immediately that they have summits wich have affix in $\mathcal{O}_{K}$ and who are invariant by the action of the elements $\mathcal{U}_{K}$ and that they have non empty interior and center $0$. These formulae still work when one is either $\emptyset$ or $\{0\}$.

So $\widetilde{\mathcal{C}_{\mathcal{O}_{K}}}$ is a sub semiring of the well known semiring of the convex sets of the plane with the operations convex hull of the union and the Minkowski operation, so we have the result.
\end{proof}
\begin{lemm}
$\mathcal{O}_{K}$ acts multiplicatively by direct complex similitudes on $\widetilde{\mathcal{C}_{\mathcal{O}_{K}}}$, that is to say that $\alpha\in\mathcal{O}_{K}\backslash\{0\}$ acts as the direct similitude $(\mathbb{C}\to\mathbb{C}, z\mapsto \alpha z)$ and $\emptyset$ is sent to $\emptyset$ and $0$ sends everything to $\{0\}$ except $\emptyset$ which is sent to $\emptyset$.
\end{lemm}
\begin{proof}
Direct similitudes preserve extremal points of convex sets and so we get the result.
\end{proof}
\begin{lemm}\label{lemm:Csimple}
For $K=\mathbb{Q}(\imath )$ and $K=\mathbb{Q}(\imath\sqrt{3})$ (in other words the only cases where $\mathcal{U}_{K}$ is greater than $\{1,-1\}$),let us denote $D_{K}$ the convex polygon (with center $0$) whose summits are the elements of $\mathcal{U}_{K}$.
Then  $$\widetilde{\mathcal{C}_{\mathcal{O}_{K}}}=\text{Semiring}(\{h.D_{K},h\in\mathcal{O}_{K}\})\cup\{\emptyset\}$$ ie the semiring generated by $\{h.D_{K},h\in\mathcal{O}_{K}\}$ (to which we add $\emptyset$), it has also an action of $\mathcal{O}_K$ on it by direct similitudes of the complex plane.
\end{lemm}
\begin{proof}
Let us first note $\omega_{K}:=\left\{\begin{aligned} \imath\quad &\:\text{if} & K=\mathbb{Q}(\imath )\\ \frac{1+\imath\sqrt{3}}{2}  &\:\text{if} & K=\mathbb{Q}(\imath\sqrt{3})\end{aligned}\right.$ and $\theta_{K}:=\text{Arg}(\omega_{K} )=\left\{\begin{aligned} \frac{\pi}{2} &\:\text{if} & K=\mathbb{Q}(\imath )\\ \frac{\pi}{3}  &\:\text{if} & K=\mathbb{Q}(\imath\sqrt{3})\end{aligned}\right.$ and $\sigma_{K}=\left\{\begin{aligned} 4 &\:\text{if} & K=\mathbb{Q}(\imath )\\ 6  &\:\text{if} & K=\mathbb{Q}(\imath\sqrt{3})\end{aligned}\right.$. 

Of course $\text{Semiring}(\{h.D_{K},h\in\mathcal{O}_{K}\})\subset\widetilde{\mathcal{C}_{\mathcal{O}_{K}}}$

By definition of $\text{Semiring}(\{h.D_{K},h\in\mathcal{O}_{K}\})$, $\{\emptyset, \{ 0\}\}\subset\text{Semiring}(\{h.D_{K},h\in\mathcal{O}_{K}\})$.

Let $C\in\widetilde{\mathcal{C}_{\mathcal{O}_{K}}}\backslash\{\emptyset, \{0\}\}$.

Let us note $\mathcal{S}_{+,C}$ the set of summits of $C$ whose arguments (modulo $2\pi$) belong to $\left[ 0,\theta_{K}\right[$ and $\mathcal{S}_{C}$ the set of all summits of $C$.

Let us now show that $\mathcal{S}_{C}=\bigcup_{u\in\mathcal{U}_{K}}u\mathcal{S}_{+,C}$

Indeed since $C$ is invariant under the action of $\mathcal{U}_{K}$ we have that $\bigcup_{u\in\mathcal{U}_{K}}u\mathcal{S}_{+,C}\subset\mathcal{S}_{C}$

Now let $s\in\mathcal{S}_{C}$, then let $k\in [|0,\sigma_{K}|]$ and $\alpha\in [0,\theta_{K}[$ such that $\text{Arg}(s)\equiv k\theta_{K}+\alpha (2\pi )$.

So $\text{Arg}(\omega_{K}^{-k}s)\equiv\alpha (2\pi )$ and $\omega_{K}^{-k}\in\mathcal{U}_{K}$

And since $C$ is invariant by the action of $\mathcal{U}_{K}$, $\zeta_{K}^{-k}s\in\mathcal{S}_{+,C}$

So $s\in\omega_{K}^{k}\mathcal{S}_{+,C}\subset\bigcup_{u\in\mathcal{U}_{K}}u\in\mathcal{S}_{+,C}$

And so $\mathcal{S}_{C}\subset\bigcup_{u\in\mathcal{U}_{K}}u\mathcal{S}_{+,C}$ which ends the proof of $\mathcal{S}_{C}=\bigcup_{u\in\mathcal{U}_{K}}u\mathcal{S}_{+,C}$

And so finally we get that $C=\text{Conv}\left(\bigcup_{s\in\mathcal{S}_{+,C}}sD_{K}\right)\in\text{Semiring}(\{h.D_{K},h\in\mathcal{O}_{K}\})$

So we conclude that $\text{Semiring}(\{h.D_{K},h\in\mathcal{O}_{K}\})\supset\widetilde{\mathcal{C}_{\mathcal{O}_{K}}}$ and so $\widetilde{\mathcal{C}_{\mathcal{O}_{K}}}=\text{Semiring}(\{h.D_{K},h\in\mathcal{O}_{K}\})$
\end{proof}
\begin{rem}
The definition of $D_{K}$ in the preceding lemma does not make any sense for the seven other quadratic imaginary number fields with class number 1.
\end{rem}

Then for $K:=\mathbb{Q}(\sqrt{d})$ with $\mathcal{U}_{K}=\{1,-1\}$, we adopt the following definition :
\begin{itemize}
\item[$\bullet$] when $d\equiv 2,3 (4)$ define $\mathcal{O}_{K}=\mathbb{Z}[\sqrt{d}]$ , define for $D_{K}$ to be the convex polygone whose summits are $1, \sqrt{d}, -1, -\sqrt{d}$
\item[$\bullet$] when $d\equiv 1 (4)$ define $\mathcal{O}_{K}=\mathbb{Z}[\frac{1+\sqrt{d}}{2}]$ , define for $D_{K}$ to be the convex polygone whose summits are $1, \frac{1+\sqrt{d}}{2}, -1, -\frac{1+\sqrt{d}}{2}$.
\end{itemize}

\begin{defi}
Let us then denote $$\mathcal{C}_{\mathcal{O}_{K}}:=\text{Semiring}(\{h.D_{K},h\in\mathcal{O}_{K}\})\cup\{\emptyset\}$$ the semiring generated by $\{h.D_{K},h\in\mathcal{O}_{K}\}$ .
\end{defi}

\begin{lemm}
Let $K$ be a quadratic imaginary number field with class number one. Then $\widetilde{\mathcal{C}_{\mathcal{O}_{K}}}=\mathcal{C}_{\mathcal{O}_{K}}$ if and only if $K=\mathbb{Q}(\imath )$ or $K=\mathbb{Q}(\imath\sqrt{3})$, ie if and only if $\{1,-1\}\subsetneqq\mathcal{U}_{K}$, ie if and only if "we have enough symmetries".
\end{lemm}
\begin{proof}
$\Leftarrow$ was shown in lemma 4.3

Let us now prove $\Rightarrow$.

Let $K$ be a quadratic imaginary number field with class number one different from $K=\mathbb{Q}(\imath )$ and $K=\mathbb{Q}(\imath\sqrt{3})$.

So $K$ is one of these number fields $\mathbb{Q}(\imath\sqrt{2})$, $\mathbb{Q}(\imath\sqrt{7})$, $\mathbb{Q}(\imath\sqrt{11})$, $\mathbb{Q}(\imath\sqrt{19})$, $\mathbb{Q}(\imath\sqrt{43})$, $\mathbb{Q}(\imath\sqrt{67})$ and $\mathbb{Q}(\imath\sqrt{163})$ and for all of these the group of units is reduced to $\{\pm 1\}$

\begin{itemize}
\item[$\bullet$]for $K=\mathbb{Q}(\imath\sqrt{2})$, $D_{K}$ is the polygon with summits $1,\imath\sqrt{2},-1,-\imath\sqrt{2}$.

Let us then note $P$ the polygon with summits $3,\imath\sqrt{2},-3,-\imath\sqrt{2}$, we have immediately $P\in\widetilde{\mathcal{C}_{\mathcal{O}_{K}}}$.

But $P\notin\mathcal{C}_{\mathcal{O}_{K}}$, indeed the only polygones in $\mathcal{C}_{\mathcal{O}_{K}}$ which have $3$ as a summit are $3D_{K}$, $D_{K}+D_{K}+D_{K}$, $D_{K}+\imath\sqrt{2}D_{K}$ but none of them have $\imath\sqrt{2}$ as a summit so none of them is equal to $P$

\item[$\bullet$]for $K=\mathbb{Q}(\imath\sqrt{7})$, $D_{K}$ is the polygon with summits $1,\frac{1+\imath\sqrt{7}}{2},-1,\frac{1-\imath\sqrt{7}}{2}$.

Let us then note $P$ the polygon with summits $2,\frac{1+\imath\sqrt{7}}{2},-2,\frac{1-\imath\sqrt{7}}{2}$, we have immediately $P\in\widetilde{\mathcal{C}_{\mathcal{O}_{K}}}$.

But $P\notin\mathcal{C}_{\mathcal{O}_{K}}$, indeed the only polygones in $\mathcal{C}_{\mathcal{O}_{K}}$ which have $2$ as a summit are $2D_{K}$, $D_{K}+D_{K}$, $\frac{1-\imath\sqrt{7}}{2}D_{K}$ but none of them have $\frac{1+\imath\sqrt{7}}{2}$ as a summit so none of them is equal to $P$

\item[$\bullet$]for the other five cases let us write $K=\mathbb{Q}(\imath\sqrt{d})$ with $d\in\{ 11,19,43,67,163 \}$, then  $D_{K}$ is the polygon with summits $1,\frac{1+\imath\sqrt{d}}{2},-1,\frac{1-\imath\sqrt{d}}{2}$.

Let us then note $P$ the polygon with summits $2,\frac{1+\imath\sqrt{d}}{2},-2,\frac{1-\imath\sqrt{d}}{2}$, we have immediately $P\in\widetilde{\mathcal{C}_{\mathcal{O}_{K}}}$.

But $P\notin\mathcal{C}_{\mathcal{O}_{K}}$, indeed the only polygones in $\mathcal{C}_{\mathcal{O}_{K}}$ which have $2$ as a summit are $2D_{K}$ and $D_{K}+D_{K}$ but none of them have $\frac{1+\imath\sqrt{d}}{2}$ as a summit so none of them is equal to $P$.
\end{itemize}
\end{proof}

\begin{rem}
Why this choice of structural sheaf? 

It is  because following the strategy of \cite{ccsitearith}, we would like now to put a structural sheaf on the topos $\widehat{\mathcal{O}_K}$ which is an idempotent  semiring, and that the points of this semiringed topos with values in something to be isomorphic to $\mathbb{A}^{f}_{K}\times\mathbb{C}/(K^{\star}(\prod\mathcal{O}_{\mathfrak{p}}^{\star}\times \{1\}))$.

 However the set $\mathcal{U}_K$ acts trivially $\mathbb{A}^{f}_{K}/(K^{\star}(\prod\mathcal{O}_{\mathfrak{p}}^{\star}\times \{1\}))$ but not on $\mathbb{C}$.
 
  This has the following consequence : let $(H,\lambda)\in ((\mathbb{A}^{f}_{K}/(\prod\mathcal{O}_{\mathfrak{p}}^{\star})\times \{1\})\times \mathbb{C})/K^{\star}$, we have $(H,\lambda)\simeq (r.H,r\lambda)$ for $r\in K^{\star}$, so for $r_{0}\in \mathcal{U}_{K}$ and so $(H,\lambda)=(H,r_{0}\lambda)$, so $\lambda$ and $r_{0}\lambda$ induce the same embedding of the fiber of the structural sheaf in $H$ into something. So $1$ and $r_{0}$ should have the same action the fiber of the structural sheaf in $H$ and more generally, $1$ and $r_{0}$ should have the same action on the structural sheaf but since the action of $1$ is the identity.
  
   We come to the conclusion that the set $\mathcal{U}_{K}$  of units of $\mathcal{O}_{K}$ should be seen as the set of \emph{symmetries} of the structural sheaf.
\end{rem}

\begin{rem}
Why this choice of $D_{K}$ for $K=\mathbb{Q}(\sqrt{d})$? An heuristic explanation could be that since $\mathcal{O}_{K}$ is a lattice in the plane, one should view it as a tiling puzzle and one of the smallest tile is the triangle $0,1,\sqrt{d}$ or $\frac{1+\sqrt{d}}{2}$ depending on $d$ (the last two elements form a base of $\mathcal{O}_{K}$ viewed as a $\mathbb{Z}$ module). Then we let the elements of $\mathcal{U_{K}}$ act on this tile, and the union of all tiles. We get this way $D_{K}$ when $K=\mathbb{Q}(\imath )$ or $K=\mathbb{Q}(\imath\sqrt{3})$ (there were enough symmetries), in the other cases in order to get $D_{K}$ we have to get the convex enveloppe of the union of all the tiles.
\end{rem}

\begin{defi}
Let $K$ be a quadratic imaginary number field with class number 1.The arithmetic site for $K$ is the datum $\left(\widehat{\mathcal{O}_{K}},\mathcal{C}_{\mathcal{O}_{K}}\right)$ where the topos $\widehat{\mathcal{O}_{K}}$ is endowed with the structure sheaf $\mathcal{C}_{\mathcal{O}_{K}}$ viewed as a semiring in the topos using the action of $\mathcal{O}_{K}$ by similitudes.
\end{defi}

\begin{theo}\label{theo:strshptone}
The stalk of the structure sheaf $\mathcal{C}_{\mathcal{O}_K}$ at the point of the topos $\widehat{\mathcal{O}_K}$ associated with the $\mathcal{O}_{K}$ module $H$ is canonically isomorphic to $\mathcal{C}_{H}:=\text{Semiring}(\{h.D_{K},h\in H\})$ the semiring generated by the polygons $h\times D_{K}$ with $h\in H$ (to which we add $\emptyset$) viewed here in this context as a semiring. 
\end{theo}

\begin{proof}
By theorem \ref{theo:ptsscd} to the point of $\widehat{\mathcal{O}_{K}}$ associated to the $\mathcal{O}_{K}$ module $H\subset K$ corresponds to the flat functor $F_{H}:\mathcal{O}_{K}$ which associates to the only object $\star $ of the small category $\mathcal{O}_{K}$ the $\mathcal{O}_{K}$ module $H$ and the endomorphism indexed by $k$ the multiplication $F(k)$ by $k$ in $H\subset K$.

As said in \cite{ccsitearith} and shown in \cite{mlm}, the inverse image functor associated to this point is the functor which associates to any $\mathcal{O}_{K}$ equivariant set its geometric realization
$$|-|_{F_{H}}\left\{\begin{aligned} \mathcal{O}_{K}-\text{equivariant sets} &\longrightarrow & \mathfrak{Sets}\\ C  &\longmapsto & |C|_{F_{H}}:=(C\times_{\mathcal{O}_{K}}H)/\sim\end{aligned}\right.$$
where $\sim$ is the equivalence relation stating the equivalence of the couples $(C,F(k)h)\sim (kC,h)$.

Let us recall as in \cite{ccsitearith} that thanks to the property (ii) of flatness of $F_{H}$, we have $$(C,h)\sim (C',h')\Leftrightarrow\exists\hat{h}\in H,\exists k,k'\in\mathcal{O}_{K},\quad\text{such that}\quad k\hat{h}=h\quad\text{and}\quad k'\hat{h}=h'\quad\text{and}\quad kC=k'C'$$

But we would like to have a better understanding and description of the fiber.

The natural candidate we imagine to be the fiber is $\mathcal{C}_{H}:=\text{Semiring}\{hD_{K}, h\in H\}$. Let us now show that this intuition is true.

Let us consider the map $\beta:\left\{\begin{aligned} \mathcal{C}_{\mathcal{O}_{K}}\times H &\to &\mathcal{C}_{H}\\ (C,h)  &\mapsto & hC\end{aligned}\right.$

Let us show that $\beta$ is compatible with the equivalence relation $\sim$.

Let $(C,h),(C',h')\in\mathcal{C}_{\mathcal{O}_{K}}\times H$ such that $(C,h)\sim (C',h')$.

So let $\hat{h}\in H$, $k,k'\in\mathcal{O}_{K}$ such that $k\hat{h}=h$ and $k'\hat{h}=h'$ and $kC=k'C'$.

Then $\beta (C,h)=hC=k\hat{h}C=\hat{h}kC=\hat{h}k'C'=k'\hat{h}C'=h'C'=\beta (C',h')$.

So $\beta$ is compatible with the equivalence relation.

And so $\beta$ induces another application again noted $\beta$ from $|\mathcal{C}_{\mathcal{O}_{K}}|_{F_{H}}$ to $\mathcal{C}_{H}$.

Let us now show that $\beta$ is surjective.

Let $C\in\mathcal{C}_{H}$, let us note $\mathcal{S}_{C}$ the set of summits of $C$. 

As $\mathcal{S}_{C}$ is a finite set, let $q\in\mathcal{O}_{K}\backslash\{0\}$ such that $\forall s\in\mathcal{S}_{C}, qs\in\mathcal{O}_{K}$.

Then $I=q<\mathcal{S}_{C}>\subset\mathcal{O}_{K}$ is an ideal where $<\mathcal{S}_{C}>$ means the sub-$\mathcal{O}_{K}$-module included in $K$ generated by the elements of $\mathcal{S}_{C}$.

Since $\mathcal{O}_{K}$ is principal, let $d\in\mathcal{O}_{K}$ such that $I=<d>$.

So $\exists \left(a_{s}\right)_{s\in\mathcal{S}_{C}}\in\left(\mathcal{O}_{K}\right)^{\mathcal{S}_{C}},d=\sum_{s\in\mathcal{S}_{C}}a_{s}qs$, so $\frac{d}{q}\in <\mathcal{S}_{C}>\subset H$.

And so $\beta(\text{Conv}\left(\bigcup_{s\in\mathcal{S}_{C}}a_{s}D_{K},\frac{d}{q}\right)=C$, so $\beta$ is surjective.

Let us now show that $\beta$ is injective.

Let $(C,h),(C',h')\in|\mathcal{C}_{\mathcal{O}_{K}}|_{F_{H}}$ such that $\beta (C,h)=\beta (C',h')$, ie such that $hC=h'C'$.

By property (ii) of flatness of $F_{H}$, there exists $\hat{h}\in H$, $k,k'\in\mathcal{O}_{K}$ such that $h=k\hat{h}$ and $h'=h'\hat{h}$. 

So we have immediately that $k\hat{h}C=k'\hat{h}C'$.

First case $\hat{h}\neq 0$ and so $kC=k'C'$ and so by definition of $\sim$, $(C,h)\sim (C',h')$.

Second case $\hat{h}=0$, then $h=0=h'$ and we have also $0.C=0.C'$ and so by definition of $\sim$, we have $(C,0)\sim (C',0)$.

Finally $\beta$ is injective and so bijective, so $|\mathcal{C}_{\mathcal{O}_{K}}|_{F_{H}}\simeq\mathcal{C}_{H}$ so we have a good understanding of the fiber now.

Now we would like to know what is the semiring structure on the fiber induced by the semiring structure on $\mathcal{C}_{\mathcal{O}_{K}}$.

Here we follow once again \cite{ccsitearith}. The two operations $\text{Conv}(\bullet \cup \bullet )$ and $+$ of $\mathcal{C}_{\mathcal{O}_{K}}$ determine canonically maps of $\mathcal{O}_{K}$-spaces $\mathcal{C}_{\mathcal{O}_{K}}\times\mathcal{C}_{\mathcal{O}_{K}}\to\mathcal{C}_{\mathcal{O}_{K}}$.

Applying the geometric realization functor, we get that the induced operations on the fiber correspond to the induced maps $|\mathcal{C}_{\mathcal{O}_{K}}\times\mathcal{C}_{\mathcal{O}_{K}}|_{F_{H}}\to|\mathcal{C}_{\mathcal{O}_{K}}|_{F_{H}}$.

But since the geometric realization functor commutes with finite limits, we get the following identification (one could also prove it by hand) :
$$\left\{\begin{aligned} |\mathcal{C}_{\mathcal{O}_{K}}\times\mathcal{C}_{\mathcal{O}_{K}}|_{F_{H}} &= &|\mathcal{C}_{\mathcal{O}_{K}}|_{F_{H}}\times |\mathcal{C}_{\mathcal{O}_{K}}|_{F_{H}}\\ (C,C',h)  &\mapsto & (C,h)\times (C',h)\end{aligned} \right.$$

And thanks to this identification one just needs to study the induced maps $$|\mathcal{C}_{\mathcal{O}_{K}}|_{F_{H}}\times |\mathcal{C}_{\mathcal{O}_{K}}|_{F_{H}}\to|\mathcal{C}_{\mathcal{O}_{K}}|_{F_{H}}$$

However we already have :

$\forall C,C'\in\mathcal{C}_{\mathcal{O}_{K}},\forall h\in H, h\text{Conv}(C\cup C')=\text{Conv}(hC\cup hC')$ and $ h(C+C')=hC+hC'$.

So we conclude that the semiring laws on the fiber induced by the semiring laws $\text{Conv}(\bullet \cup \bullet )$ and $+$ of $\mathcal{C}_{\mathcal{O}_{K}}$ are the laws $\text{Conv}(\bullet \cup \bullet )$ and $+$ on $\mathcal{C}_{H}$.

\end{proof}

\begin{rem}
With the same notations as in the last theorem, when $K=\mathbb{Q}(\imath )$ or $K=\mathbb{Q}(\imath\sqrt{3})$, thank to the symmetries, $\mathcal{C}_{H}$ is the semiring of convex compact polygons with non empty interior and center zero and summits with affixes in $H$ and symmetric under the action of $\mathcal{U}_K$ (and with $\emptyset$ and $\{0\}$).
\end{rem}

\begin{prop}
The set of global sections $\Gamma (\widehat{\mathcal{O}_{K}},\mathcal{C}_{\mathcal{O}_{K}})$ of the structure sheaf are given by the semiring $\{\emptyset,\{0\}\}\simeq\mathbb{B}$.
\end{prop}
\begin{proof}
As in \cite{ccsitearith} we recall that for a Grothendieck topos $\mathcal{T}$ the global section functor $\Gamma:\mathcal{T}\to\mathfrak{Sets}$ is given by $\Gamma (E):=\text{Hom}_{\mathcal{T}}(1,E)$ where $E$ is an object in the topos and $1$ the final object in the topos. In the special case of a topos of the form $\widehat{\mathcal{C}}$ where $\mathcal{C}$ is a small category, a global section of a contravariant functor $P:\mathcal{C}\to\mathfrak{Sets}$ is a function which assigns to each object $C$ of $\mathcal{C}$ an element $\gamma_{C}\in P(C)$ in such a way that for any morphism $f:D\to C\in\text{Hom}_{\mathcal{C}}(D,C)$ and $D$ any another object of $\mathcal{C}$ one has $P(f)\gamma_{C}=\gamma_{D}$ (as explained in \cite{mlm} Chap I.6.(9)).

When we apply this definition to our special case which is the small category $\mathcal{O}_{K}$, the global sections of sheaf $\mathcal{C}_{\mathcal{O}_{K}}$ are the elements of $\mathcal{C}_{\mathcal{O}_{K}}$ which are invariant by the action of $\mathcal{O}_{K}$, so the global sections of $\mathcal{C}_{\mathcal{O}_{K}}$ are only $\emptyset$ and $\{ 0\}$ and $\{\emptyset ,\{ 0\}\}\simeq\mathbb{B}$. So the set of global sections is isomorphic to $\mathbb{B}$.
\end{proof}

Let us now denote for $K$ an imaginary quadratic number field of class number 1 :
\begin{defi}
$\mathcal{C}_{K,\mathbb{C}}:=\text{Semiring}(\{h.D_{K},h\in\mathbb{C}\})\cup\{\emptyset\}$  the semiring generated by $\{h.D_{K},h\in\mathbb{C}\}$  (to which we add $\emptyset$ the neutral element for $\text{Conv}(\bullet\cup\bullet )$ which is absorbant for $+$).
\end{defi}

\begin{rem}
When $K=\mathbb{Q}(\imath )$ or $K=\mathbb{Q}(\imath\sqrt{3})$, due to the symmetries, $\mathcal{C}_{K,\mathbb{C}}$ is the semiring of convex compact polygons with non empty interior and center zero  and symmetric under the action of $\mathcal{U}_K$ (and with $\emptyset$ and $\{0\}$).
\end{rem}

We have the following interesting result on $\mathcal{C}_{K,\mathbb{C}}$ :

\begin{defi}
Let us denote $\text{Aut}_{\mathbb{B}}^{+}(\mathcal{C}_{K,\mathbb{C}})$ the group of direct $\mathbb{B}$-automorphisms of $\mathcal{C}_{K,\mathbb{C}}$.

Let us give the definition of direct $\mathbb{B}$-automorphisms of $\mathcal{C}_{K,\mathbb{C}}$ : an element $f$ of $\text{Aut}_{\mathbb{B}}^{+}(\mathcal{C}_{K,\mathbb{C}})$ is an application from $\mathcal{C}_{K,\mathbb{C}}$ to $\mathcal{C}_{K,\mathbb{C}}$ such that :

\begin{itemize}
\item[$\bullet$] $f$ is bijective
\item[$\bullet$] $\forall C,D\in\mathcal{C}_{K,\mathbb{C}}, f(\text{Conv}(C\cup D)=\text{Conv}(f(C)\cup f(D))$
\item[$\bullet$] $\forall C,D\in\mathcal{C}_{K,\mathbb{C}}, f(C+D)=f(C)+f(D)$
\item[$\bullet$] $\forall \mu\in\mathbb{S}^{1}, \forall C\in\mathcal{C}_{K,\mathbb{C}}, f(\mu.C)=\mu.f(C)$
\end{itemize}.
\end{defi}

Before showing the interesting result on $\text{Aut}_{\mathbb{B}}^{+}(\mathcal{C}_{K,\mathbb{C}})$, let us prove the following lemma :

\begin{lemm}\label{lemm:convlin}
We have that for any $f\in\text{Aut}_{\mathbb{B}}^{+}(\mathcal{C}_{K,\mathbb{C}})$:

\begin{itemize}
\item[$\bullet$] $f(\emptyset)=\emptyset$

\item[$\bullet$] $f(\{0\})=\{ 0\}$

\item[$\bullet$] $\forall C\in\mathcal{C}_{K,\mathbb{C}}, \forall\lambda\in \mathbb{C}, f(\lambda C)=\lambda f(C)$ 

In other words, the elements of $\text{Aut}_{\mathbb{B}}^{+}(\mathcal{C}_{K,\mathbb{C}})$ commute with complex homotheties.
\end{itemize}
\end{lemm}

\begin{proof}
Let $f\in\text{Aut}_{\mathbb{B}}^{+}(\mathcal{C}_{K,\mathbb{C}})$.

We can first deduce, from the definition of $f$, that $\forall C,D\in\mathcal{C}_{K,\mathbb{C}}, C\subset D\Rightarrow f(C)\subset f(D)$.

\begin{itemize}
\item[$\bullet$] Since $f$ is bijective, let $E\in\mathcal{C}_{K,\mathbb{C}}$ be the only element of $\mathcal{C}_{K,\mathbb{C}}$ such that $f(E)=\emptyset$.

Since $\emptyset\subset E$, we have that $f(\emptyset)\subset f(E)=\emptyset$. So $f(\emptyset)=\emptyset$.

Therefore we have in fact that $E=\emptyset$, so in other otherwords the only element of $\mathcal{C}_{K,\mathbb{C}}$ whose image by $f$ is $\emptyset$ is $\emptyset$.

\item[$\bullet$] Since $f$ is bijective and since $f(\emptyset )=\emptyset$, let $Z\in\mathcal{C}_{K,\mathbb{C}}\backslash\{\emptyset\}$ be the only element of $\mathcal{C}_{K,\mathbb{C}}$ such that $f(Z)=\{0\}$.

But $\{0\}\subset Z$ so $f(\{0\})\subset\{0\}$. But $f(\{0\})\neq\emptyset$ since $f$ is bijective and $f(\emptyset)=\emptyset$ so $f(\{0\})=\{0\}$.

Therefore $Z=\{0\}$ and so $\{0\}$ is the only element of $\mathcal{C}_{K,\mathbb{C}}$ whose image by $f$ is $\{0\}$.

\item[$\bullet$] Let $f\in\text{Aut}_{\mathbb{B}}^{+}(\mathcal{C}_{K,\mathbb{C}})$.

Then using the definition of $\text{Aut}_{\mathbb{B}}^{+}(\mathcal{C}_{K,\mathbb{C}})$, by induction, we can show that $$\forall C\in\mathcal{C}_{K,\mathbb{C}}, \forall\lambda\in \mathbb{N}, f(\lambda C)=\lambda f(C)$$

Since $\{\pm 1\}\subset\mathcal{U}_{K}$ and since the elements of $\mathcal{C}_{K,\mathbb{C}}$ are symmetric by the action of the elements of $\mathcal{U}_{K}$, we get that $$\forall C\in\mathcal{C}_{K,\mathbb{C}}, \forall\lambda\in \mathbb{Z}, f(\lambda C)=\lambda f(C)$$

Then classically if we take $C\in\mathcal{C}_{K,\mathbb{C}}$ and $\lambda\in\mathbb{Q}^{\star}$, we take $p,q\in\mathbb{Z}$ prime to each other such that $\lambda=\frac{p}{q}$,

Then $q.f(\lambda.C)=f(q\lambda.C)=f(p.C)=p.f(C)$, so $f(\lambda.C)=\lambda.C$ so finally we have that $$\forall C\in\mathcal{C}_{K,\mathbb{C}}, \forall\lambda\in \mathbb{Q}, f(\lambda C)=\lambda f(C)$$

Let $C\in\mathcal{C}$ and $\lambda\in\mathbb{R}^{+}$.

Let us denote $\mathcal{D}_{C}^{+}:=\{\mu.C,\mu\in\mathbb{R}^{+}\}$. Equipped with the inclusion relation $\subset$, $\mathcal{D}_{C}^{+}$ is a totally ordered set.

Since $\mathbb{R}^{+}$ has the least upper bound property and since $\forall \mu,\mu'\in\mathbb{R}^{+}, \mu.C\subset\mu'.C\Leftrightarrow \mu\leq\mu'$, $\mathcal{D}_{C}^{+}$ has the least upper bound property too, for an non empty strict subset $\mathcal{A}$ of $\mathcal{D}_{C}^{+}$, we will denote its least upper bound $\text{sup}^{\subset} (\mathcal{A})$.

But by its definition, $f$ is a non decreasing map for the order $\subset$. Moreover $f$ is bijective, so for every non empty strict subset $\mathcal{A}$ of $\mathcal{D}_{C}^{+}$, $f(\text{sup}^{\subset} (\mathcal{A}))=\text{sup}^{\subset}(f(\mathcal{A})$.

But we know that $\lambda=\text{sup}^{\leq }\{r\in\mathbb{Q}/r<\lambda\}$, so we have $f(\text{sup}^{\subset}\{r.C/r\in\mathbb{Q}\wedge r<\lambda\})=\text{sup}^{\subset}\{r.f(C)/r\in\mathbb{Q}\wedge r<\lambda\}$.

But we easily have that $f(\text{sup}^{\subset}\{r.C/r\in\mathbb{Q}\wedge r<\lambda\})
=f(\lambda.C)$ and that $\text{sup}^{\subset}\{r.f(C)/r\in\mathbb{Q}\wedge r<\lambda\}=\text{sup}^{\leq}\{r\in\mathbb{Q}/r<\lambda\}.C=\lambda.C$.

So we have that $f(\lambda.C)=\lambda.C$.

And so we have proved that $$\forall C\in\mathcal{C}_{K,\mathbb{C}}, \forall\lambda\in \mathbb{R}^{+}, f(\lambda C)=\lambda f(C)$$

And finally since $\forall \mu\in\mathbb{S}^{1}, \forall C\in\mathcal{C}_{K,\mathbb{C}}, f(\mu.C)=\mu.f(C)$, we can conclude that $$\forall C\in\mathcal{C}_{K,\mathbb{C}}, \forall\lambda\in \mathbb{C}, f(\lambda C)=\lambda f(C)$$
\end{itemize}.
\end{proof}

We can now state the interesting result on $\text{Aut}_{\mathbb{B}}^{+}(\mathcal{C}_{K,\mathbb{C}})$ :

\begin{prop}
We have that $\text{Aut}_{\mathbb{B}}^{+}(\mathcal{C}_{K,\mathbb{C}})=\mathbb{C}^{\star}/\mathcal{U}_{K}$.
\end{prop}

\begin{proof}

Let $f\in\text{Aut}_{\mathbb{B}}^{+}(\mathcal{C}_{K,\mathbb{C}})$.

We can first recall, from the definition of $f$, that $\forall C,D\in\mathcal{C}_{K,\mathbb{C}}, C\subset D\Rightarrow f(C)\subset f(D)$.

The convex sets $\lambda.D_{K}$ for $\lambda\in\mathbb{C}^{\star}$ are caracterised abstractly by the property that they cannot be decomposed non trivially in the semiring $\mathcal{C}_{K,\mathbb{C}}$. More precisely for every $\lambda\in\mathbb{C}^{\star}$, if $A,B\in\mathcal{C}_{K,\mathbb{C}}$ are such that $\lambda.D_{K}=\text{Conv}(A\cup B)$, then either $A$ or $B$ is equal to $\lambda.D_{K}$, let's say without loss of generality that it is $A$, and then $B\subset A$. 

Indeed let $\lambda\in\mathbb{C}^{\star}$ and let $A,B\in\mathcal{C}_{K,\mathbb{C}}$ such that $\lambda.D_{K}=\text{Conv}(A\cup B)$, then the set of extremal points of $\lambda.D_{K}$ is included in the union of the set of extremal points of $A$ and the sets of extremal points of $B$. Let $P$ be an extremal point of $\lambda.D_{K}$ of greater module, then $P$ is either an extremal point of $A$ or an extremal point of $B$. Without loss of generality, let's say that $P$ is an extremal point of $A$. Then because the definition of $\mathcal{C}_{K,\mathbb{C}}$ and because $A\subset\lambda.D_{K}$, we get that $A=\lambda.D_{K}$. And so because $\lambda.D_{K}=\text{Conv}(A\cup B)$, we get that $B\subset A$.

This abstract property of the $\lambda.D_{K}$ is preserved by any automorphism of $\mathcal{C}_{K,\mathbb{C}}$, so $f$ sends the sets of the form $\lambda.D_{K}$ to other sets of the form $\lambda.D_{K}$. So let $\lambda\in\mathbb{C}^{\star}$ such that $f(D_{K})=\lambda.D_{K}$. And since the elements of $\mathcal{C}_{K,\mathbb{C}}$ are symmetric with respect to $\mathcal{U}_{K}$, such a $\lambda$ is unique only up to multiplication by an element of $\mathcal{U}_{K}$.

Since $\mathcal{C}_{K,\mathbb{C}}=\text{Semiring}(\{h.D_{K},h\in\mathbb{C}\})\cup\{\emptyset\}$, $f$ is only determined by the image it gives to $D_{K}$.

Therefore $\text{Aut}_{\mathbb{B}}^{+}(\mathcal{C}_{K,\mathbb{C}})=\mathbb{C}^{\star}/\mathcal{U}_{K}$.

\end{proof}

\begin{defi}
A point of $\left(\widehat{\mathcal{O}_{K}},\mathcal{C}_{\mathcal{O}_{K}}\right)$ over $\mathcal{C}_{K,\mathbb{C}}$ is a pair $(p,f)$ given by a point $p$ of the topos $\widehat{\mathcal{O}_{K}}$ and a direct similitude $f$ (so it preserves the orientation and the number of summits) from $\mathbb{C}$ to $\mathbb{C}$ which induces a morphism $f^{\sharp}_{p}:\mathcal{C}_{\mathcal{O}_{K},p}\to\mathcal{C}_{K,\mathbb{C}}$ of semirings from the stalk of $\mathcal{C}_{\mathcal{O}_{K}}$ at the point $p$ into $\mathcal{C}_{K,\mathbb{C}}$.
\end{defi}

\begin{lemm}\label{lemm:bijmodad}
We follow the notations of theorem \ref{theo:strshptone}. Let us denote $\mathcal{S}_{mod}$ the set of sub-semirings of $\mathcal{C}_{K,\mathbb{C}}$ of the form $\text{Semiring}\{hD_{K},h\in H\}$ where $H$ is a sub-$\mathcal{O}_{K}$ module of $K$.

Let us note now $\Phi$ the map from $\left(\mathbb{A}^{f}_{K}/\widetilde{\mathcal{O}_{K}}^{\star}\right)\times\left(\mathbb{C}/\mathcal{U}_{K}\right)$ to $\mathcal{S}_{mod}$ defined by

$$\Phi\left\{\begin{aligned} \left(\mathbb{A}^{f}_{K}/\widetilde{\mathcal{O}_{K}}^{\star}\right)\times\left(\mathbb{C}/\mathcal{U}_{K}\right) &\to & \mathcal{S}_{mod}\\ (a,\lambda)  &\mapsto & \Phi (a,\lambda ):=\text{Semiring}\{h\lambda D_{K}, h\in H_{a}\}.\end{aligned}\right.$$

Then $\Phi$ induces a bijection between the quotient of $\left(\mathbb{A}^{f}_{K}/\widetilde{\mathcal{O}_{K}}^{\star}\right)\times\left(\mathbb{C}/\mathcal{U}_{K}\right)$ by the diagonal action of $K^{\star}$ and the set $\mathcal{S}_{mod}$.
\end{lemm}

\begin{proof}
Let us first show that the map $\Phi$ is invariant under the diagonal action of $K^{\star}$.

Let $k\in K^{\star}$ and $(a,\lambda )\in \left(\mathbb{A}^{f}_{K}/\widetilde{\mathcal{O}_{K}}^{\star}\right)\times\left(\mathbb{C}/\mathcal{U}_{K}\right)$

Then since $H_{ka}=k^{-1}H_{a}$, we have $$\begin{aligned} \Phi (ka,k\lambda ) & = & \text{Semiring}\{ hk\lambda D_{K}, h\in H_{ka}\} \\  & = & \text{Semiring}\{ h\lambda D_{K}, h\in H_{a}\} \\  & = & \Phi (a,\lambda )\end{aligned}$$

We also have immediately thanks to theorem \ref{theo:ptsscd} that $\Phi$ is surjective.

Let us now show that $\Phi$ is injective.

Let $(a,\lambda ), (b,\mu )\in\left(\mathbb{A}^{f}_{K}/\widetilde{\mathcal{O}_{K}}^{\star}\right)\times\left(\mathbb{C}/\mathcal{U}_{K}\right)$ such that $\Phi (a,\lambda )=\Phi (b,\mu )$.

So $\text{Semiring}\{h\lambda D_{K}, h\in H_{a}\}=\text{Semiring}\{\tilde{h}\mu D_{K}, \tilde{h}\in H_{b}\}$.

And since the summits of $D_{K}$ are $1, -1, s_{K}, -s_{K}$ where $s_{K}=\imath\sqrt{d}$ when $K=\mathbb{Q}(\sqrt{-d})$ with $-d\equiv 1,3(4)$ and $s_{K}=\frac{1+\imath\sqrt{d}}{2}$ when $K=\mathbb{Q}(\sqrt{-d})$ with $-d\equiv 2(4)$

We get $1.\lambda .H_{a}\subset 1\mu H_{b}+(-1).\mu .H_{b}+s_{K}\mu H_{b}+(-s_{K}).\mu .H_{b}$,

And $s_{K}.\lambda .H_{a}\subset 1\mu H_{b}+(-1).\mu .H_{b}+s_{K}\mu H_{b}+(-s_{K}).\mu .H_{b}$,

And $(-1).\lambda .H_{a}\subset 1\mu H_{b}+(-1).\mu .H_{b}+s_{K}\mu H_{b}+(-s_{K}).\mu .H_{b}$,

And $(-s_{K}).\lambda .H_{a}\subset1\mu H_{b}+(-1).\mu .H_{b}+s_{K}\mu H_{b}+(-s_{K}).\mu .H_{b}$.

So $1\lambda H_{a}+(-1).\lambda .H_{a}+s_{K}\lambda H_{a}+(-s_{K}).\lambda .H_{a}\subset 1\mu H_{b}+(-1).\mu .H_{b}+s_{K}\mu H_{b}+(-s_{K}).\mu .H_{b}$.

And by the same process we get the other way around so we have the equality $$\lambda (1 H_{a}+(-1).H_{a}+s_{K}.H_{a}+(-s_{K}).H_{a})=\mu (1.H_{b}+(-1).H_{b}+s_{K}.H_{b}+(-s_{K}).H_{b})$$

But $1,-1,s_{K},-s_{K}\in\mathcal{O}_{K}$ and $H_{a},H_{b}$ are $\mathcal{O}_{K}$-modules, so $\lambda H_{a}=\mu H_{b}$.

Since $H_{a}$ and $H_{b}$ are $\mathcal{O}_{K}$-modules of rank 1, there exists $k\in K^{\star}$ such that $\lambda =k\mu$.

So then one gets $\lambda H_{a}=\lambda kH_{b}$.

So $H_{a}=kH_{b}$, but $kH_{b}=H_{k^{-1}b}$.

So by theorem \ref{theo:ptsscd} we have $a=k^{-1} b$ in $\mathbb{A}^{f}_{K}/\widetilde{\mathcal{O}_{K}}^{\star}$ and then we get $\mu=k\lambda$ and $b=ka$ thus the injectivity of $\Phi$. The lemma is proved.

\end{proof}

\begin{theo}\label{theo:ptsckc}
The set of points of the arithmetic site $\left(\widehat{\mathcal{O}_{K}},\mathcal{C}_{\mathcal{O}_{K}}\right)$ over $\mathcal{C}_{K,\mathbb{C}}$ is naturally identified to $\left(\mathbb{A}^{f}_{K}\times\mathbb{C}/\mathcal{U}_{K}\right)/\left(K^{\star}\times\left(\prod_{\mathfrak{p}\enskip \text{prime}}\mathcal{O}^{\star}_{\mathfrak{p}}\times{1}\right)\right)$ which can also be identified to $\mathbb{A}_{K}/\left( K^{\star}\left(\prod_{\mathfrak{p}}\mathcal{O}_{\mathfrak{p}}^{\star}\times\{ 1\}\right)\right)$.
\end{theo}

\begin{rem}
This theorem is a generalization of  the theorem of Connes and Consani in\cite{ccsitearith} on the interpretation of the points of the arithmetic site over $\mathbb{R}_{\max}$.
\end{rem}

\begin{proof}
Let us consider a point of the arithmetic site with values in $\mathbb{C}_{K,\mathbb{C}}$ $(p,f_{p}^{\star})$.

By theorem \ref{theo:ptsscd}, to a point of the topos $\widehat{\mathcal{O}_{K}}$ is associated to $H$ an $\mathcal{O}_{K}$-module (of rank 1 since $\mathcal{O}_{K}$ is principal) included in $K$, and by theorem \ref{theo:strshptone} the stalk of $\mathcal{C}_{\mathcal{O}_{K}}$ at the point $p$ is $\mathcal{C}_{H}$.

As in \cite{ccsitearith} we consider the following two cases depending on the range of $f_{p}^{\star}$ :
\begin{enumerate}
\item The range of $f_{p}^{\star}$ is $\mathbb{B}\simeq\{\emptyset ,\{ 0\}\}\subset\mathcal{C}_{K,\mathbb{C}}$.

$f_{p}^{\star}$ sends non empty sets of $\mathcal{C}_{H}$ to $\{0\}$ and so the pair $(p,f_{p}^{\star})$ is uniquely determined by the point $p$ and so by theorem \ref{theo:ptsscd} the set of those kind of points is isomorphic to $K^{\star}\backslash\mathbb{A}_{K}^{f}/\widetilde{\mathcal{O}_{K}}^{\star}$.

\item The range of $f_{p}^{\star}$ is not contained in $\mathbb{B}$, then by the definition of a point ($f_{p}^{\star}$ is direct similitude) the range of $f_{p}^{\star}$ is of the form $\text{Semiring}\{h\lambda D_{K}, h\in H\}$ and $\lambda\in\mathbb{C}/\mathcal{U}_{K}$ with $f_{p}^{\star}(D_{K})=\lambda D_{K}$. 

So by lemma \ref{lemm:bijmodad} the set of those points is isomorphic to $K^{\star}\backslash\left(\mathbb{A}_{K}^{f}/\widetilde{\mathcal{O}_{K}}^{\star}\times\mathbb{C}/\mathcal{U}_{K}\right)$ 
\end{enumerate}
So all in all the set of the points of $\widehat{\mathcal{O}_{K}}$ with values in $\mathcal{C}_{K,\mathbb{C}}$ is isomorphic to $K^{\star}\backslash\left(\mathbb{A}_{K}^{f}\times\mathbb{C}/\mathcal{U}_{K}\right)/\left(\widetilde{\mathcal{O}_{K}}^{\star}\times\{1\}\right)$
\end{proof}

\section{Link with the Dedekind zeta function}
\subsection{The spectral realization of critical zeros of $L$-functions by Connes}
In this section, $K$ will denote an imaginary quadratic number field with class number 1.

Let us first recall some facts (\cite{connesmarcolli} and \cite{trace}) about homogeneous distributions on adeles and L-functions and then the construction of the Hilbert space $\mathcal{H}$ underlying the spectral realization of the critical zeroes.


Let $k$ be a local field and $\chi$ a quasi-character of $k^{\star}$, let $s\in\mathbb{C}$, we can write $\chi$ in the following form : $\forall x\in k^{\star}\chi(x)=\chi_{0}(x)|x|^{s}$ with $\chi_{0}:k^{\star}\to S^{1}$. Let $\mathcal{S}(k)$ denote the Schwartz Bruhat space on $k$.

\begin{defi}
We say that a distribution $D\in\mathcal{S}'(k)$ is homogeneous of weight $\chi$ if one has $\forall f\in\mathcal{S}(k),\forall a\in k^{\star}, \langle f^{a},D\rangle =\chi(a)^{-1}\langle f,D\rangle$ where by definition $f^{a}(x)=f(ax)$.
\end{defi}

We have the following property :

\begin{prop}  
For $\sigma=\Re(s)>0$, there exists up to normalization, only one homogeneous distribution  of weight $\chi$ on $k$ which is given by the absolutely convergent integral $\Delta_{\chi}(f)=\int_{k^{\star}}f(x)\chi(x)\text{d}^{\star}x$.
\end{prop}

\begin{proof}
Cf in \cite{weiltr}.
\end{proof}

If $k$ is non-archimedean field, and $\pi_{k}$ a uniformizer, let us now define and note for all $s\in\mathbb{C}$: 
\begin{defi}
The distribution $\Delta^{'}_{s}\in\mathcal{S}'(k)$ is defined for all $f\in\mathcal{S}(k)$ as $\Delta^{'}_{s}=\int_{k^{\star}}(f(x)-f(\pi_{k}x))|x|^{s}\text{d}^{\star}x$, with the multiplicative Haar measure $\text{d}^{\star}x$ normalized by $\langle 1_{\mathcal{O}^{\star}_{k}},\text{d}^{\star}x\rangle =1$
\end{defi}
The distribution $\Delta_{s}'$ is well defined because by the very definition of $\mathcal{S}(k)$, for $f\in\mathcal{S}(k)$, $f(x)-f(\pi_{k}x)=0$ for $x$ small enough.

This distribution has the following properties :

\begin{prop}
We have :
\begin{enumerate}
\item $\langle 1_{\mathcal{O}_{k}},\Delta^{'}_{s}\rangle =1$
\item $\langle f^{a}, \Delta^{'}_{s}\rangle = |a|^{-s}\langle f,\Delta^{'}_{s}\rangle$
\item $\Delta^{'}_{s}=(1-q^{-s})\Delta_{s}$ with $|\pi_{k}|=q^{-1}$
\end{enumerate}
\end{prop}

\begin{proof}
Cf §9.1 in \cite{connesmarcolli}.
\end{proof}

Let $\chi$ be now a quasi-character from the idele class group $C_{K}=\mathbb{A}_{K}^{\star}/K^{\star}$. We can note $\chi$ as $\chi =\prod_{\nu}$ and $\chi (x)=\chi_{0}(x)|x|^s$ with $s\in\mathbb{C}$ and $\chi_{0}$ a character of $C_{K}$. Let us note $P$ the finite set of places where $\chi_{0}$ is ramified. For any place $\nu\notin P$, let us denote $\Delta_{\nu}^{'}$ the unique homogeneous distribution of weight $\chi_{\nu}$ normalized by $\langle\Delta_{\nu}^{'},1_{\mathcal{O}_{\nu}}\rangle =1$. For any $\nu\in P$ or infinite place and for $\sigma=\Re (s)>0$ let us denote $\Delta_{\nu}^{'}$ the homogeneous distribution of weight $\chi_{\nu}$ given by proposition 5.1 (this one is unnormalized). Then the infinite tensor product $\Delta_{s}^{'}=\prod_{\nu}\Delta_{\nu}^{'}$ makes sense as a continuous linear form on $\mathcal{S}(\mathbb{A}_{K})$ and it is homogeneous of weight $\chi$, it is not equal to zero since $\Delta_{\nu}^{'}\neq 0$ for every $\nu$ and for infinite places as well and is finite by construction of the space $\mathcal{S}(\mathbb{A}_{K})$ as the infinite tensor product $\bigotimes_{\nu}\left(\mathcal{S}(K_{\nu}),1_{\mathcal{O}_{\nu}}\right)$

Then we can see the L functions appear as a normalization factor thanks to the following property:

\begin{prop}\label{prop:2.52}
For $\sigma =\Re (s)>1$, the following integral converges absolutely 
$$\forall f\in\mathcal{S}(\mathbb{A}_{K}), \int_{\mathcal{A}_{K}^{\star}}f(x)\chi_{0}(x)|x|^{s}\text{d}^{\star}x=\Delta_{s}(f)$$

and we have $\forall f\in\mathcal{S}(\mathbb{A}_{K}), \Delta_{s}(f)=L(\chi_{0},s)\Delta_{s}^{'}(f)$.
\end{prop}
\begin{proof}
Cf lemma 2.50 in \cite{connesmarcolli}
\end{proof}

Let us now note $\mathcal{S}(\mathbb{A}_{K})_{0}$ the following subspace of $\mathcal{S}(\mathbb{A}_{K})$ defined by $\mathcal{S}(\mathbb{A}_{K})_{0}=\{ f\in \mathcal{S}(\mathbb{A}_{K}) / f(0)=0, \int_{\mathbb{A}_{K}}f(x)\text{d}x=0 \}$. 
We can now define the operator $\mathcal{E}$:

\begin{defi}
Let $f\in\mathcal{S}(\mathbb{A}_{K})_{0}$ and $g\in C_{K}$, then we set $$\mathcal{E}(f)(g)=|g|^{\frac{1}{2}}\sum_{q\in K^{\star}}f(qg)$$.
\end{defi}
We have the following properties for $\mathcal{E}$:

\begin{prop} \label{prop:rapidecay} We have :
\begin{enumerate}
\item for all $f\in\mathcal{S}(\mathbb{A}_{K})_{0}$ and $g\in C_{K}$, the series $\mathcal{E}(f)(g)$ converges absolutely
\item $\forall f\in\mathcal{S}(\mathbb{A}_{K})_{0},\,\forall n\in\mathbb{N},\exists c>0,\forall g\in C_{K}, |\mathcal{E}(f)(g)|\leq c\exp (-n|\log |g| |)$
\item for all $f\in\mathcal{S}(\mathbb{A}_{K})_{0}$ and $g\in C_{K}$, $\mathcal{E}(\hat{f})(g)=\mathcal{E}(f)(g^{-1})$
\end{enumerate}
\end{prop}
\begin{proof}
Cf lemma 2.51 in \cite{connesmarcolli}.
\end{proof}
And so we can get that 

\begin{prop}
For $\sigma =\Re (s)>0$ and any character $\chi_{0}$ of $C_{K}$, we have that $$\forall f\in\mathcal{S}(\mathbb{A}_{K})_{0}, \int_{C_{K}}\mathcal{E}(f)(x)\chi_{0}(x)|x|^{s-\frac{1}{2}}\text{d}^{\star}x=cL(\chi_{0},s)\Delta_{s}^{'}(f)$$ where the non zero constant $c$ depends on the normalization of the Haar measure $\text{d}^{\star}x$ on $C_{K}$. 
\end{prop}

\begin{proof}
Cf lemma 2.52 in \cite{connesmarcolli}.
\end{proof}

And we also have the following lemma :

\begin{lemm}\label{lemm:approx}
There exists an approximate unit $(f_{n})_{n\in\mathbb{N}}$ such that for all $n\in\mathbb{N}$, $f_{n}\in\mathcal{S}(C_{K})$, $\hat{f_{n}}$ has compact support and there exists $C>0$ such that $||\theta_{m}(f_{n})||\leq C$ and that $\theta_{m}(f_{n})\to 1$ strongly in $L^{2}_{\delta}(C_{K})$ as $n\to\infty$.
\end{lemm}

Now we are able to define the Hilbert space $\mathcal{H}$. First on $\mathcal{S}(\mathbb{A}_{K})_{0}$ we can put the inner product corresponding to the norm  $\Vert f\Vert_{\delta }^{2}=\int_{C_{K}} |\mathcal{E}(f)(x)|^{2}(1+(\log |x|)^{2})^{\frac{\delta }{2}}\text{d}^{\star}x$.

Let us denote $L_{\delta}^{2}(\mathbb{A}_{K}/K^{\star})_{0}$ the separated completion of $\mathcal{S}(\mathbb{A}_{K})_{0}$ with respect to the inner product defined ealier. Let us also define $\theta_{a}$ the representation of $C_{K}$ on $\mathcal{S}(\mathbb{A}_{K})$ given by for $\xi\in\mathcal{S}(\mathbb{A}_{K})$, $\forall\alpha\in C_{K}, \forall x\in\mathbb{A}_{K}, (\theta_{a}(\alpha)\xi)(x)=\xi (\alpha^{-1}x)$.

We can also put the following Sobolev norm on $L_{\delta}^{2}(C_{K})$, $\Vert\xi\Vert_{\delta}^{2}=\int_{C_{K}}|\xi (x)|^{2}(1+(\log |x|)^{2})^{\frac{\delta}{2}}\text{d}^{\star}x$.

Then by construction, the linear map $\mathcal{E}:\mathcal{S}(\mathbb{A}_{K})_{0}\to L_{\delta}^{2}(C_{K})$ satisfies 
for all $f\in\mathcal{S}(\mathbb{A}_{K})_{0}$, $\Vert f\Vert_{\delta}^{2}=\Vert\mathcal{E}(f)\Vert_{\delta}^{2}$. Thus this map extends to an isometry still denoted $\mathcal{E}:L_{\delta}^{2}(\mathbb{A}_{K}/K^{\star})_{0}\to L_{\delta}^{2}(C_{K})$.

Let us also denote $\theta_{m}$ the regular representation of $C_{K}$ on $L_{\delta}^{2}(C_{K})$ .

We have for any $\xi\in L_{\delta}^{2}(C_{K})$ : $$\forall\alpha\in C_{K}, \forall x\in C_{K}, (\theta_{a}(\alpha)\xi)(x)=\xi (\alpha^{-1}x)$$.

We then get for every $f\in L_{\delta}^{2}(\mathbb{A}_{K}/K^{\star})_{0}$, $\alpha\in C_{K}$ and $g\in C_{K}$ that :

$\begin{aligned}
\mathcal{E}(\theta_{a}(\alpha)f)(g) & = & |g|^{\frac{1}{2}}\sum_{q\in K^{\star}}(\theta_{a}(\alpha)f)(qg) \\
 & = & |g|^{\frac{1}{2}}\sum_{q\in K^{\star}}f(\alpha^{-1}qg) \\
 & = & |g|^{\frac{1}{2}}\sum_{q\in K^{\star}}f(q\alpha^{-1}g) \\
 & = & |\alpha |^{\frac{1}{2}}|\alpha^{-1}g|^{\frac{1}{2}}\sum_{q\in K^{\star}}f(q\alpha^{-1}g) \\
 & = & |\alpha |^{\frac{1}{2}}(\theta_{m}(\alpha)\mathcal{E}(f))(g)
\end{aligned}$

Thus we have that $\mathcal{E}\theta_{a}(\alpha)=|\alpha |^{\frac{1}{2}}\theta_{m}(\alpha)\mathcal{E}$. In other words, it shows that the natural representation $\theta_{a}$ of $C_{K}$ on $L_{\delta}^{2}(\mathbb{A}_{K}/K^{\star})_{0}$ corresponds, via the isometry $\mathcal{E}$, to the restriction of $|\alpha |^{\frac{1}{2}}\theta_{m}(\alpha )$ to the invariant subspace given by the range of $\mathcal{E}$.

\begin{defi}
We denote by $\mathcal{H}=L^{2}_{\delta}(C_{K})/\text{Im}(\mathcal{E})$ the cokernel of the map $\mathcal{E}$. Let us denote also $\underline{\theta}_{m}$ the quotient representation of $C_{K}$ on $\mathcal{H}$ and finally let us denote for a character $\chi$ of $C_{K,1}$, $\mathcal{H}_{\chi}=\{h\in \mathcal{H}/\forall g\in C_{K,1}, \theta_{m}(g)h=\chi(g)h\}$
\end{defi}
Since $N_{1}$ is a compact group, we get that :
\begin{prop}
The Hilbert space $\mathcal{H}$ splits as a direct sum $\mathcal{H}=\bigoplus_{\chi\in \widehat{C_{K,1}}}\mathcal{H}_{\chi}$ and the representation $\underline{\theta}_{m}$ decomposes as a direct sum of representation $\underline{\theta}_{m,\chi}:C_{K}\to\text{Aut}(\mathcal{H}_{\chi})$
\end{prop}

This situation gives rise to operators whose spectra will the critical zeroes of L functions and so the spectral interpretation of it.:

\begin{defi}
Let us define and note $D_{\chi}$ the infinitesimal generator of the restriction of $\underline{\theta}_{m,\chi}$ to $1\times\mathbb{R}^{\star}_{+}\subset C_{K}$, in other words we have for every $\xi\in\mathcal{H}_{\chi}$, $D_{\chi}\xi=\lim_{\epsilon\to 0}\frac{1}{\epsilon}(\underline{\theta}_{m,\chi}-1)\xi$
\end{defi}

Then the central theorem of the spectral realisation of the critical zeroes of the L functions as in \cite{trace} and in \cite{connesmarcolli} is :

\begin{theo}\label{theo:spconnes}
Let $\chi$, $\delta >1$, $\mathcal{H}_{\chi}$ and $D_{\chi}$ as above. Then $D_{\chi}$ has a discrete spectrum and $\text{Sp}(D_{\chi})\subset \imath\mathbb{R}$ is the set of imaginary parts of zeroes of the L-function with Gr\"{o}ssencharacter $\tilde{\chi}$ (the extension of $\chi$ to $C_{K}$) which have real part equal to $\frac{1}{2}$, ie $\rho\in\text{Sp}(D_{\chi})\Leftrightarrow L\left(\tilde{\chi},\frac{1}{2}+\rho\right)=0$ and $\rho\in\imath\mathbb{R}$. Moreover the multiplicity of $\rho$ in $\text{Sp}(D_{\chi})$ is equal to the largest integer $n<1+\frac{\delta}{2}$, $n\leq$ multiplicity of $\frac{1}{2}+\rho$ as a zero of L.
\end{theo}

\begin{proof}

We follow the proof already of \cite{trace} and in \cite{connesmarcolli} making it more precise with respect to our goal.

We first need to understand the range of $\mathcal{E}$, in order to do that, we consider its orthogonal in the dual space that is $L^{2}_{\delta}(C_{K})$. 

Since the subgroup $C_{K,1}$ of $C_{K}$  is the group the ideles classes of norm 1, $C_{K,1}$ is a compact group and acts by the representation $\theta_{m}$ which is unitary when restricted to $C_{K,1}$.

Therefore we can decompose $L^{2}_{\delta}(C_{K})$ and its dual $L^{2}_{-\delta}(C_{K})$ into the direct of the following subspaces 

$$L^{2}_{\delta,\chi_{0}}(C_{K})=\{\xi\in L^{2}_{\delta}(C_{K});\forall g\in C_{K}, \forall \gamma\in C_{K,1}, \xi(\gamma^{-1}g)=\chi_{0}(\gamma)\xi(g)\}$$

which correspond to the projections $P_{\chi_{0}}=\int_{C_{K}}\bar{\chi_{0}}(\gamma)\theta_{m}(\gamma)\text{d}_{1}\gamma$.

And for the dual : $$L^{2}_{-\delta,\chi_{0}}(C_{K})=\{\xi\in L^{2}_{-\delta}(C_{K});\forall g\in C_{K}, \forall \gamma\in C_{K,1}, \xi(\gamma g)=\chi_{0}(\gamma)\xi(g)\}$$

which correspond to the projections $P_{\chi_{0}}^{\textit{t}}=\int_{C_{K}}\bar{\chi_{0}}(\gamma)\theta_{m}(\gamma)^{\textit{t}}\text{d}_{1}\gamma$.

Here we have used $(\theta_{m}(\gamma)^{\textit{t}}\eta)(x)=\eta(\gamma x)$ which comes from the definition of the transpose $\langle\theta_{m}(\gamma)\xi,\eta\rangle=\langle\xi,\theta_{m}(\gamma)^{\textit{t}}\eta)\rangle$ using $\int_{C_{K}}\xi(\gamma^{-1}x)\eta(x)d^{\star}x=\int_{C_{K}}\xi(y)\eta(\gamma y)d^{\star}y$.

In these formulas one only uses the character $\chi_{0}$ as a character of the compact subgroup $C_{K,1}$ of $C_{K}$. One now chooses non canonically an extension $\tilde{\chi}_{0}$ of $\chi_{0}$ as a character of $C_{K}$ (ie we have $\forall\gamma\in C_{K,1}, \tilde{\chi}_{0}(\gamma )=\chi_{0}(\gamma )$). This choice is not unique and two choices of extensions only differ by a character that is principal (ie of the form $\gamma\mapsto |\gamma |^{is_{0}}$ with $s_{0}\in\mathbb{R}$). We fix a factorization $C_{K}=C_{K,1}\times\mathbb{R}^{\star}_{+}$ and fix $\tilde{\chi}_{0}$ as being equal to $1$ on $\mathbb{R}^{\star}_{+}$.

Then by definition, we can write any element $\eta$ of $L^{2}_{-\delta,\chi_{0}}(C_{K})$ in the form : $$\eta: g\in C_{K} \mapsto \eta(g)=\tilde{\chi}_{0}(g)\psi(|g|)$$ where $\int_{C_{K}}|\psi(|g|)|^{2}(1+(\log |g|)^{2})^{-\delta/2}d^{\star}g<\infty$.

A vector like this $\eta$ is in the orthogonal of the range of $\mathcal{E}$ if and only if : $$\forall f\in\mathcal{S}(\mathbb{A}_{K})_{0},\,\int_{C_{K}}\mathcal{E}(f)(x)\tilde{\chi}_{0}(x)\psi(|x|)d^{\star}x=0$$

Using Mellin inversion formula $\psi(|x|)=\int_{\mathbb{R}}\hat{\psi}(t)|x|^{\imath t}dt$, we can see formally that this last equality becomes equivalent to : $$\int_{C_{K}}\int_{\mathbb{R}}\mathcal{E}(f)(x)\tilde{\chi}_{0}(x)\hat{\psi}(t)|x|^{\imath t}dtd^{\star}x=\int_{\mathbb{R}}\Delta_{\frac{1}{2}+\imath t}(f)\hat{\psi}(t)dt$$

Those formal manipulations are justified  by the use of the approximate units with special properties which appear in a previous lemma \ref{lemm:approx} and the rapid decay of $\mathcal{E}(f)$ of the proposition \ref{prop:rapidecay}.

Thanks to the last formula, we are now looking for nice functions $f\in\mathcal{S}(\mathbb{A}_{K})_{0}$ on which to test the distribution $\int_{C_{K}}\Delta_{\frac{1}{2}+\imath t}\hat{\psi}(t)dt$.

For the finite places, we denote by $P$ the finite set of finite places where $\chi_{0}$ ramifies, we take $f_{0}:=\otimes_{\nu\notin P}1_{\mathcal{O}_{\nu}}\otimes f_{\chi_{0}}$ where $f_{\chi_{0}}$ is the tensor product over ramified  places of the functions equal to $0$ outside $\mathcal{O}_{\nu}^{\star}$ and to $\chi_{0},\nu$ on $\mathcal{O}_{\nu}^{\star}$.

Then by the definition of $\Delta_{s}'$, for any $f\in\mathcal{S}(\mathbb{C})$ we get that  $\langle\Delta_{s}',f_{0}\otimes f\rangle=\int_{C_{K}} f(x)\chi_{0;\infty}(x)|x|^{s}d^{\star}x$. Moreover if the set $P$ of finite ramified places of $\chi_{0}$ is not empty, we have $f_{0}(0)=0$ and $\int_{\mathbb{A}_{K}^{f}}f_{0}(x)dx=0$ so that $f_{0}\otimes f\in\mathcal{S}(\mathbb{A}_{K})_{0}$ for all $f\in\mathcal{S}(\mathbb{C})$. 

We can in fact take a function $f$ of the form $f(x)=b(x)\bar{\chi}_{0,\infty}(x)$ with $b\in\mathcal{C}^{\infty}_{c}(\mathbb{R}^{\star}_{+})$. 

So for any $s\in\mathbb{C}$ such that $\Re s>0$, $\langle\Delta_{s}',f_{0}\otimes f_{b}\rangle =\int_{\mathbb{R}^{\star}_{+}}b(x)|x|^{s}d^{\star}x$.

So when we pair the distribution $\int_{\mathbb{R}}\Delta_{\frac{1}{2}+\imath t}\hat{\psi}(t)dt$ again such functions, we get that :

$\langle\int_{\mathbb{R}}\Delta_{\frac{1}{2}+\imath t}\hat{\psi}(t)dt,f_{0}\otimes f_{b}\rangle =\iint_{C_{K}\times\mathbb{R}} L(\chi_{0},\frac{1}{2}+\imath t)\hat{\psi}(t)b(x)|x|^{\frac{1}{2}+\imath t}d^{\star}xdt$.

But one can see that, if $\chi_{0}|_{C_{K,1}}\neq 1$, $L(\chi_{0},\frac{1}{2}+\imath t)$ is an analytic function of $t$ so the product $L(\chi_{0},\frac{1}{2}+\imath t)\hat{\psi}(t)$ is a tempered distrubution and so is its Fourier transform. Thanks to the last equality, we have that the Fourier transform of $L(\chi_{0},\frac{1}{2}+\imath t)\hat{\psi}(t)$ paired on arbitrary functions which are smooth with compact support equals 0 and so the Fourier transform $L(\chi_{0},\frac{1}{2}+\imath t)\hat{\psi}(t)$ is equal to $0$.

If $\chi_{0}|_{C_{K,1}}= 1$, we need to impose the condition $\int_{\mathbb{A}_{K}} fdx=0$ ie  $\int_{\mathbb{R}^{\star}_{+}} b(x)|x|d^{\star}x=0$ but we can see that the space of functions $b(x)|x|^{\frac{1}{2}}\in\mathcal{C}^{\infty}_{c}(\mathbb{R}^{\star}_{+})$ with the condition $\int_{\mathbb{R}^{\star}_{+}} b(x)|x|d^{\star}x=0$ is dense in $\mathcal{S}(\mathbb{R}^{\star}_{+})$.

Let us now recall that for the equation $\phi(t)\alpha(t)=0$ with $\alpha$ a distribution on $\mathbb{S}^{1}$ and $\phi\in\mathcal{C}^{\infty}(\mathbb{S}^{1})$ which has finitely many zeroes denoted $x_{i}$ of order $n_{i}$ with $i\in I$ with $I$ a finite set , the distributions $\delta_{x_{i}},\delta_{x_{i}}',\ldots,\delta_{x_{i}}^{n_{i}-1}, i\in I$ form a basis of the space of solutions in $\alpha$ (of the equation $\phi(t)\alpha(t)=0$).

Now we can come back to our main study. Thanks to what we have shown before, we now know that for $\eta$ orthogonal to the range of $\mathcal{E}$ and such that $\theta_{m}^{t}(h)(\eta )=\eta$, we have that $\hat{\psi}(t)$ is a distribution with compact support satisfying the equation $L(\chi_{0},\frac{1}{2}+\imath t)\hat{\psi}(t)=0$.

Therefore thanks to what we have recalled, we get that $\hat{\psi}$ is a finite linear combination of the distributions $\delta_{t}^{(k)}$ with $t$ such that $L(\chi_{0},\frac{1}{2}+\imath t)=0$ and $k$ striclty less than the order of the zero of this $L$ function (necessary and sufficient to get the vanishing on the range of $\mathcal{E}$) and also $k<\frac{\delta -1}{2}$ (necessary and sufficient to ensure that $\psi$ belongs to $L^{2}_{-\delta}(\mathbb{R}^{\star}_{+})$, ie $\int_{\mathbb{R}^{\star}_{+}} (\log|x|)^{2k}(1+|\log|x||^{2})^{-\delta/2}d^{\star}x<\infty$).

Conversely, let $s$ be a zero of $L(\chi_{0},s)$ of order $k>0$. Then by the proposition \ref{prop:2.52} and the finiteness and the analyticity of $\Delta_{s}'$ for $\Re s>0$, we get, for $a\in [| 0,k-1|]$ and $f\in\mathcal{S}(\mathbb{A}_{K})_{0}$, that : $\left(\frac{\partial}{\partial s}\right)^{a}\Delta_{s}'(f)=0$. 

We also have that $\left(\frac{\partial}{\partial s}\right)^{a}\Delta_{s}'(f)=\int_{C_{K}}\mathcal{E}(f)(x)\chi_{0}(x)|x|^{s-\frac{1}{2}}(\log|x|)^{a}d^{\star}x$.

Thus $\eta$ belongs to the orthogonal of the range of $\mathcal{E}$ and such that $\theta_{m}^{t}(h)\eta=\eta$ if and only if it is a finite linear combination of functions of the form $$\eta_{t,a}(x)=\chi_{0}(x)|x|^{\imath t}(\log|x|)^{a}$$ where $L\left(\chi_{0},\frac{1}{2}+\imath t\right)=0$ and $a<$ order of the zero $t$ and $a<\frac{\delta -1}{2}$.

Therefore the restriction to the subgroup $\mathbb{R}^{\star}_{+}$ of $C_{K}$ of the transpose of $\theta_{m}$ is given in the above basis $\eta_{t,a}$ by $$\theta_{m}(\lambda)^{t}\eta_{t,a}=\sum_{b=0}^{a}C_{a}^{b}\lambda^{\imath t}(\log(\lambda))^{b}\eta_{t,b-a}$$

Therefore if $L(\chi_{0},\frac{1}{2}+\imath s)\neq 0$ then $\imath s$ does not belong to the spectrum of $D_{\chi_{0}}^{t}$. This determine the spectrum of the operator $D_{\chi_{0}}^{t}$ and so the spectrum of $D_{\chi_{0}}$. Therefore the theorem is proved.
\end{proof}

\subsection{The link between the points of the arithmetic site and the Dedekind zeta function}

Since the class number of $K$ of $1$, we observe that $C_{K,1}$ the ideles classes of norm 1 is given by : $$C_{K,1}=(K^{\star}(\prod_{\mathfrak{p}}\mathcal{O}_{\mathfrak{p}}^{\star}\times \mathbb{S}^{1}))/K^{\star}$$ 

We still denote $\mathcal{H}$ the Hilbert space associated by Connes in \cite{trace} to $\left(\mathbb{A}^{f}_{K}\times\mathbb{C}\right)/K^{\star}$ and whose definition was recalled in the last section.

In the last section we have seen that one can decompose $\mathcal{H}$ in the following way : $\mathcal{H}=\bigoplus_{\chi\in\widehat{C_{K,1}}}\mathcal{H}_{\chi}$ with $\mathcal{H}_{\chi}=\{h\in \mathcal{H}/\forall g\in C_{K,1}, \theta_{m}(g)h=\chi(g)h\}$.


Let us note $G=\left(K^{\star}\times\left(\prod_{\mathfrak{p}\enskip \text{prime}}\mathcal{O}^{\star}_{\mathfrak{p}}\times{1}\right)\right)/K^{\star}$.

We can observe that $C_{K,1}/G\simeq\mathbb{S}^{1}/\mathcal{U}_{K}$.

The main idea here is that we would like to have a spectral interpretation of $\zeta_{K}$ linked to the space of points of the arithmetic site $\left(\widehat{\mathcal{O}_{K}},\mathcal{C}_{\mathcal{O}_{K}}\right)$ over $\mathcal{C}_{K,\mathbb{C}}$ which is by theorem \ref{theo:ptsckc} $$\left(\mathbb{A}^{f}_{K}\times\mathbb{C}/\mathcal{U}_{K}\right)/\left(K^{\star}\times\left(\prod^{'}_{\mathfrak{p}\enskip \text{prime}}\mathcal{O}^{\star}_{\mathfrak{p}}\times{1}\right)\right)$$

In Connes' formalism (\cite{trace}, \cite{connesmarcolli}) and as recalled in the last section, $\mathcal{H}$ is an Hilbert space associated to the adele class space $\left(\mathbb{A}^{f}_{K}\times\mathbb{C}\right)/K^{\star}$ linked with the spectral interpretation of L functions. More precisely if we denote $\chi_{\text{trivial}}\in\widehat{C_{K,1}}$ the trivial character of $\widehat{C_{K,1}}$, then the results of \cite{trace} and \cite{connesmarcolli} show that $\mathcal{H}_{\chi_{\text{trivial}}}$ is associated to the spectral interpretation of $\zeta_{K}$.

\begin{theo}\label{theo:intsp}
We have $\mathcal{H}^{G}=\bigoplus_{\chi\in\widehat{S^{1}/\mathcal{U}_{K}}}\mathcal{H}^{G}_{\chi}$. Then as in \cite{trace} the space $\mathcal{H}^{G}_{\chi}$ corresponds to $L(\chi,\bullet )$, so in particular  when $\chi$ is trivial, $\mathcal{H}^{G}_{\chi}$ corresponds to $\zeta_{K}$ the Dedekind zeta function of $K$.
\end{theo}

\begin{proof}
We adapt here the same strategy as in the proof of theorem \ref{theo:spconnes}: 

In our case, let us consider $L^{2}_{\delta}(C_{K})^{G}$ (stable under the action of $G$) and $L^{2}_{-\delta}(C_{K})^{G}$ (stable under the action of $G$).

Since the subgroup $C_{K,1}G\simeq \mathbb{S}^{1}/\mathcal{U}_{K}$ thanks to $\theta_{m}$ (as recalled earlier $\theta_{m}$ denotes the regular representation of $C_{K}$ on $L^{2}_{\delta}(C_{K})$).

Therefore we can decompose $L^{2}_{\delta}(C_{K})^{G}$ and its dual $L^{2}_{-\delta}(C_{K})^{G}$ into the direct of the following subspaces 

$$L^{2}_{\delta,\chi_{0}}(C_{K})=\{\xi\in L^{2}_{\delta}(C_{K});\forall g\in C_{K}, \forall \gamma\in C_{K,1}/G, \xi(\gamma^{-1}g)=\chi_{0}(\gamma)\xi(g)\}$$

And for the dual : $$L^{2}_{-\delta,\chi_{0}}(C_{K})=\{\xi\in L^{2}_{-\delta}(C_{K});\forall g\in C_{K}, \forall \gamma\in C_{K,1}/G, \xi(\gamma g)=\chi_{0}(\gamma)\xi(g)\}$$

Let us also recall that to a character $\chi_{0}\in \widehat{C_{K,1}/G}\simeq\widehat{\mathbb{S}^{1}/\mathcal{U}_{K}}$, one can uniquely associate a Gr\"{o}{\ss}encharakter $\widetilde{\chi_{0}}$ (the conditions being a Gr\"{o}{\ss}encharakter and $K$ being class number $1$ give that the non archimedean part of $\widetilde{\chi_{0}}$ is completely determined by the archimedean part which is $\chi_{0}$).

From now on we can follow exactly the same strategy as the one used in the proof of \ref{theo:spconnes} :

Then by definition, we can write any element $\eta$ of $L^{2}_{-\delta,\chi_{0}}(C_{K})^{G}$ in the form : $$\eta: g\in C_{K} \mapsto \eta(g)=\tilde{\chi}_{0}(g)\psi(|g|)$$ where $\int_{C_{K}}|\psi(|g|)|^{2}(1+(\log |g|)^{2})^{-\delta/2}d^{\star}g<\infty$.

A vector like this $\eta$ is in the orthogonal of the range of $\mathcal{E}$ if and only if : $$\forall f\in\mathcal{S}(\mathbb{A}_{K})_{0},\,\int_{C_{K}}\mathcal{E}(f)(x)\tilde{\chi}_{0}(x)\psi(|x|)d^{\star}x=0$$

Using Mellin inversion formula $\psi(|x|)=\int_{\mathbb{R}}\hat{\psi}(t)|x|^{\imath t}dt$, we can see formally that this last equality becomes equivalent to : $$\int_{C_{K}}\int_{\mathbb{R}}\mathcal{E}(f)(x)\tilde{\chi}_{0}(x)\hat{\psi}(t)|x|^{\imath t}dtd^{\star}x=\int_{\mathbb{R}}\Delta_{\frac{1}{2}+\imath t}(f)\hat{\psi}(t)dt$$

Those formal manipulations are justified  by the use of the approximate units with special properties which appear in a previous lemma \ref{lemm:approx} and the rapid decay of $\mathcal{E}(f)$ of the proposition \ref{prop:rapidecay}.

Thanks to the last formula, we are now looking for nice functions $f\in\mathcal{S}(\mathbb{A}_{K})_{0}$ on which to test the distribution $\int_{C_{K}}\Delta_{\frac{1}{2}+\imath t}\hat{\psi}(t)dt$.

For the finite places, we denote by $P$ the finite set of finite places where $\widetilde{\chi_{0}}$ ramifies, we take $f_{0}:=\otimes_{\nu\notin P}1_{\mathcal{O}_{\nu}}\otimes f_{\widetilde{\chi_{0}}}$ where $f_{\widetilde{\chi_{0}}}$ is the tensor product over ramified  places of the functions equal to $0$ outside $\mathcal{O}_{\nu}^{\star}$ and to $\widetilde{\chi_{0}}_{\nu}$ on $\mathcal{O}_{\nu}^{\star}$.

Then by the definition of $\Delta_{s}'$, for any $f\in\mathcal{S}(\mathbb{C})$ we get that  $\langle\Delta_{s}',f_{0}\otimes f\rangle=\int_{C_{K}} f(x)\chi_{0}(x)|x|^{s}d^{\star}x$. Moreover if the set $P$ of finite ramified places of $\widetilde{\chi_{0}}$ is not empty, we have $f_{0}(0)=0$ and $\int_{\mathbb{A}_{K}^{f}}f_{0}(x)dx=0$ so that $f_{0}\otimes f\in\mathcal{S}(\mathbb{A}_{K})_{0}$ for all $f\in\mathcal{S}(\mathbb{C})$. 

We can in fact take a function $f$ of the form $f(x)=b(x)\chi_{0}(x)$ with $b\in\mathcal{C}^{\infty}_{c}(\mathbb{R}^{\star}_{+})$. 

So for any $s\in\mathbb{C}$ such that $\Re s>0$, $\langle\Delta_{s}',f_{0}\otimes f_{b}\rangle =\int_{\mathbb{R}^{\star}_{+}}b(x)|x|^{s}d^{\star}x$.

So when we pair the distribution $\int_{\mathbb{R}}\Delta_{\frac{1}{2}+\imath t}\hat{\psi}(t)dt$ again such functions, we get that :

$\langle\int_{\mathbb{R}}\Delta_{\frac{1}{2}+\imath t}\hat{\psi}(t)dt,f_{0}\otimes f_{b}\rangle =\iint_{C_{K}\times\mathbb{R}} L(\widetilde{\chi_{0}},\frac{1}{2}+\imath t)\hat{\psi}(t)b(x)|x|^{\frac{1}{2}+\imath t}d^{\star}xdt$.

But one can see that, if $\widetilde{\chi_{0}}$ is non trivial, $L(\widetilde{\chi_{0}},\frac{1}{2}+\imath t)$ is an analytic function of $t$ so the product $L(\widetilde{\chi_{0}},\frac{1}{2}+\imath t)\hat{\psi}(t)$ is a tempered distrubution and so is its Fourier transform. Thanks to the last equality, we have that the Fourier transform of $L(\widetilde{\chi_{0}},\frac{1}{2}+\imath t)\hat{\psi}(t)$ paired on arbitrary functions which are smooth with compact support equals 0 and so the Fourier transform $L(\widetilde{\chi_{0}},\frac{1}{2}+\imath t)\hat{\psi}(t)$ is equal to $0$.

If $\widetilde{\chi_{0}}$ is trivial, we need to impose the condition $\int_{\mathbb{A}_{K}} fdx=0$ ie  $\int_{\mathbb{R}^{\star}_{+}} b(x)|x|d^{\star}x=0$ but we can see that the space of functions $b(x)|x|^{\frac{1}{2}}\in\mathcal{C}^{\infty}_{c}(\mathbb{R}^{\star}_{+})$ with the condition $\int_{\mathbb{R}^{\star}_{+}} b(x)|x|d^{\star}x=0$ is dense in $\mathcal{S}(\mathbb{R}^{\star}_{+})$.

Let us now recall that for the equation $\phi(t)\alpha(t)=0$ with $\alpha$ a distribution on $\mathbb{S}^{1}$ and $\phi\in\mathcal{C}^{\infty}(\mathbb{S}^{1})$ which has finitely many zeroes denoted $x_{i}$ of order $n_{i}$ with $i\in I$ with $I$ a finite set , the distributions $\delta_{x_{i}},\delta_{x_{i}}',\ldots,\delta_{x_{i}}^{n_{i}-1}, i\in I$ form a basis of the space of solutions in $\alpha$ (of the equation $\phi(t)\alpha(t)=0$).

Now we can come back to our main study. Thanks to what we have shown before, we now know that for $\eta$ orthogonal to the range of $\mathcal{E}$ and such that $\theta_{m}^{t}(h)(\eta )=\eta$, we have that $\hat{\psi}(t)$ is a distribution with compact support satisfying the equation $L(\widetilde{\chi_{0}},\frac{1}{2}+\imath t)\hat{\psi}(t)=0$.

Therefore thanks to what we have recalled, we get that $\hat{\psi}$ is a finite linear combination of the distributions $\delta_{t}^{(k)}$ with $t$ such that $L(\widetilde{\chi_{0}},\frac{1}{2}+\imath t)=0$ and $k$ striclty less than the order of the zero of this $L$ function (necessary and sufficient to get the vanishing on the range of $\mathcal{E}$) and also $k<\frac{\delta -1}{2}$ (necessary and sufficient to ensure that $\psi$ belongs to $L^{2}_{-\delta}(\mathbb{R}^{\star}_{+})$, ie $\int_{\mathbb{R}^{\star}_{+}} (\log|x|)^{2k}(1+|\log|x||^{2})^{-\delta/2}d^{\star}x<\infty$).

Conversely, let $s$ be a zero of $L(\widetilde{\chi_{0}},s)$ of order $k>0$. Then by the proposition \ref{prop:2.52} and the finiteness and the analyticity of $\Delta_{s}'$ for $\Re s>0$, we get, for $a\in [| 0,k-1|]$ and $f\in\mathcal{S}(\mathbb{A}_{K})_{0}$, that : $\left(\frac{\partial}{\partial s}\right)^{a}\Delta_{s}'(f)=0$. 

We also have that $\left(\frac{\partial}{\partial s}\right)^{a}\Delta_{s}'(f)=\int_{C_{K}}\mathcal{E}(f)(x)\widetilde{\chi_{0}}(x)|x|^{s-\frac{1}{2}}(\log|x|)^{a}d^{\star}x$.

Thus $\eta$ belongs to the orthogonal of the range of $\mathcal{E}$ and such that $\theta_{m}^{t}(h)\eta=\eta$ if and only if it is a finite linear combination of functions of the form $$\eta_{t,a}(x)=\widetilde{\chi_{0}}(x)|x|^{\imath t}(\log|x|)^{a}$$ where $L\left(\widetilde{\chi_{0}},\frac{1}{2}+\imath t\right)=0$ and $a<$ order of the zero $t$ and $a<\frac{\delta -1}{2}$.

Therefore the restriction to the subgroup $\mathbb{R}^{\star}_{+}$ of $C_{K}$ of the transpose of $\theta_{m}$ is given in the above basis $\eta_{t,a}$ by $$\theta_{m}(\lambda)^{t}\eta_{t,a}=\sum_{b=0}^{a}C_{a}^{b}\lambda^{\imath t}(\log(\lambda))^{b}\eta_{t,b-a}$$

Therefore if $L(\widetilde{\chi_{0}},\frac{1}{2}+\imath s)\neq 0$ then $\imath s$ does not belong to the spectrum of $D_{\widetilde{\chi_{0}}}^{t}$. This determine the spectrum of the operator $D_{\widetilde{\chi_{0}}}^{t}$ and so the spectrum of $D_{\widetilde{\chi_{0}}}$. Therefore the theorem is proved. Let us remark that in the rest of the thesis and in the theorem we will make an abuse of notation and write $L(\chi_{0},\bullet)$ instead of $L(\widetilde{\chi_{0}},\bullet)$.
\end{proof}

\section{Link between $\text{Spec}\left(\mathcal{O}_{K}\right)$ and the arithmetic site}

In this section, $K$ will still denote an imaginary quadratic number field with class number 1.

We consider the Zariski topos $\text{Spec}\left(\mathcal{O}_{K}\right)$.

Let us denote for any prime ideal $\mathfrak{p}$ of $\mathcal{O}_{K}$,  $H(\mathfrak{p}):=\{ h\in K/ \alpha_{\mathfrak{p}} h\in\widetilde{\mathcal{O}_{K}}\}$ where $\alpha_{\mathfrak{p}}\in\mathbb{A}^{f}_{K}$ is the finite adele whose components are all equal to $1$ except at $\mathfrak{p}$ where the component vanishes.

\begin{defi}
Let us denote $\mathcal{S}_{K}$ the sheaf of sets on $\text{Spec}\left(\mathcal{O}_{K}\right)$ which assigns for each Zariski open set $U\subset\text{Spec}\left(\mathcal{O}_{K}\right)$ the set $\Gamma\left( U,\mathcal{S}_{K}\right) :=\{U\ni\mathfrak{p}\mapsto\xi_{\mathfrak{p}}\in H(\mathfrak{p})/\xi_{\mathfrak{p}}\neq 0 \enskip\text{for finitely many prime ideals}\enskip\mathfrak{p}\in U\}$. The action of $\mathcal{O}_{K}$ on the sections is done pointwise.
\end{defi}

\begin{theo}\label{theo:liensiteun}
The functor $T:\mathcal{O}_{K}\to\mathcal{S}h\left(\text{Spec}\left(\mathcal{O}_{K}\right)\right)$ which associates to the only object $\star$ of the small category $\mathcal{O}_{K}$ (the one already considered earlier where $\mathcal{O}_{K}$ is used as a mono\"{i}d with respect to the multiplication law), the sheaf $\mathcal{S}_{K}$ and to endomorphisms of $\star$ the action of $\mathcal{O}_{K}$ on $\mathcal{S}_{K}$ is filtering and so defines a geometric morphism $\Theta : \text{Spec}\left(\mathcal{O}_{K}\right)\to\widehat{\mathcal{O}_{K}}$. The image of a point $\mathfrak{p}$ of $\text{Spec}\left(\mathcal{O}_{K}\right)$ associated to the prime ideal $\mathfrak{p}$ of $\mathcal{O}_{K}$ is the point of $\widehat{\mathcal{O}_{K}}$ associated to the $\mathcal{O}_{K}$ module $H(\mathfrak{p})\subset K$.
\end{theo}
\begin{proof}
To check that the functor $T$ is filtering, we adapt in the very same way as in \cite{ccsitearith} the definition VII 8.1 of \cite{mlm} of the three filtering conditions for a functor and the lemma VII 8.4 of \cite{mlm} where those conditions are reformulated and apply it to our very special case where $\mathcal{O}_{K}$ is the small category which has only a single object $\star$ and $\mathcal{O}_{K}$ (the ring of integers of $K$) as endomorphism and where its image under $T$ is the object $T(\star )=\mathcal{S}$ of $\text{Spec}(\mathcal{O}_{K})$ to get that $T$ is filtering if and only if it respects the three following conditions:
\begin{enumerate}
\item For any open set $U$ of $\text{Spec}(\mathcal{O}_{K})$ there exists a covering $\{ U_{j}\}$ of $U$ and sections $\xi_{j}\in\Gamma (U_{j},\mathcal{S})$
\item For any open set $U$ of $\text{Spec}(\mathcal{O}_{K})$ and sections $c,d\in\Gamma (U,\mathcal{S})$, there exists a covering $\{ U_{j}\}$ of $U$ and for each $j$ arrows $u_{j},v_{j}:\star\to\star$ in $\mathcal{O}_{K}$ and a section $b_{j}\in\Gamma (U_{j},\mathcal{S})$ such that $c|_{U_{j}}=T(u_{j})b_{j}$ and $d|_{U_{j}}=T(v_{j})b_{j}$
\item Given two arrows $u,v:\star\to\star$ in $\mathcal{O}_{K}$ and a section $c\in\Gamma (U,\mathcal{S})$ with $T(u)c=T(v)c$, there exists a covering $\{ U_{j}\}$ of $U$ and for each $j$ an arrow $w_{j}:\star\to\star$ and a section $z_{j}\in\Gamma (U_{j},\mathcal{S})$ such that for each $j$, $T(w_{j})z_{j}=c|_{U_{j}}$ and $u\circ w_{j}=v\circ w_{j}\in\text{Hom}_{\mathcal{O}_{K}}(\star ,\star)$
\end{enumerate}
So to check that $T$ is filtering, all we have to do now is to check the three filtering conditions.
\begin{itemize}
\item[$\bullet$] Let us check (i).
Let $U$ be a non empty open set of $\text{Spec}(\mathcal{O}_{K})$, then the $0$ section, ie the section hose value at each prime ideal is $0$, is an element of $\Gamma (U,\mathcal{S})$ and so by considering $U$ itself as a cover of $U$ we have shown (i).
\item[$\bullet$] Let us check (ii)

Let $U$ be a non empty open set of $\text{Spec}(\mathcal{O}_{K})$.

Let $c,d\in\Gamma (U,\mathcal{S})$ two sections of $\mathcal{S}$ over $U$.

Then there exists a finite set $E\subset U$ of prime ideals of $\mathcal{O}_{K}$ such that both $c$ and $d$ vanish in the complement $V:=U\backslash E$ of $E$.

$V$ is a non empty set of $\text{Spec}(\mathcal{O}_{K})$ and let us note for each $\mathfrak{p}\in E$, $U_{\mathfrak{p}}:=V\cup\{\mathfrak{p}\}\subset U$.

By construction the collection $\{ U_{\mathfrak{p}}\}_{\mathfrak{p}\in E}$ form an open covering of $U$.

Then the restriction of the section $c$ and $d$ to $U_{\mathfrak{p}}$ are only determined by their value at $\mathfrak{p}$ since they vanish at every other point of $U_{\mathfrak{p}}$. Moreover given an element $b\in\mathcal{S}_{\mathfrak{p}}$, one can extend it uniquely to a section of $\mathcal{S}$ on on $U_{\mathfrak{p}}$ which vanishes on the complement of $\mathfrak{p}$. 

So finally for each $\mathfrak{p}$, thanks to the property (ii) of flatness of the functor associated to the stalk $\mathcal{S}_{\mathfrak{p}}=H_{\mathfrak{p}}$, as required we get that there exists arrows $u_{\mathfrak{p}},v_{\mathfrak{p}}\in\mathcal{O}_{K}$ and a section $b_{\mathfrak{p}}\in\Gamma (U_{\mathfrak{p}},\mathcal{S})$ such that $c|_{U_{\mathfrak{p}}}=T(u_{\mathfrak{p}})b_{\mathfrak{p}}$ and $d|_{U_{\mathfrak{p}}}=T(v_{\mathfrak{p}})b_{\mathfrak{p}}$. Since $\{ U_{\mathfrak{p}}\}_{\mathfrak{p}\in E}$ is an open cover of $U$, we finally get (ii).
\item[$\bullet$] Let us now check (iii)

Let $U$ be an open set of $\text{Spec}(\mathcal{O}_{K})$, let $c\in\Gamma (U,\mathcal{S})$ and $u,v\in\mathcal{O}_{K}$ such that $T(u)c=T(v)c$.
\begin{itemize}
\item[$\star$] Let us assume first that there is a prime ideal $\mathfrak{p}$ of $\mathcal{O}_{K}$ such that $c_{\mathfrak{p}}\neq 0$.

Then by property (iii) of flatness of the functor associated $\mathcal{S}_{\mathfrak{p}}=H_{\mathfrak{p}}$, let $\tilde{w}\in\mathcal{O}_{K}$ and $\tilde{z}_{\mathfrak{p}}\in H_{\mathfrak{p}}$ such that $T(w)\tilde{z}_{\mathfrak{p}}=c_{\mathfrak{p}}$ and $u\tilde{w}=v\tilde{w}$.

We cannot have $\tilde{w}=0$ because then we could have $c_{\mathfrak{p}}=0$ which is impossible.

So we have that $\tilde{w}=\neq 0$ and so that $u=v$.

And then we take $U$ itself as its own cover and with the notations of (iii) we take $z=c$ and $w=1\in\mathcal{O}_{K}$ and so (iii) is checked in this case.
\item[$\star$] Otherwise $c$ is the zero section.

In this case we take $U$ itself as its own cover and with the notation of (iii) we take $z=0$ and $w=0$.
\end{itemize}
\end{itemize}
We have thus shown that $T:\mathcal{O}_{K}\to\text{Spec}(\mathcal{O}_{K})$ respects the conditions (i), (ii) and (iii), so $T$ is filtering.

Then by theorem VII 9.1 of \cite{mlm} we get that  $T$ is flat.

Therefore by theorem VII 7.2 of \cite{mlm}, we get that $T$ defines a geometric morphism $\Theta :\text{Spec}(\mathcal{O}_{K})\to\widehat{\mathcal{O}_{K}}$.

Similarly to \cite{ccsitearith}, the image of a point $\mathfrak{p}$ of $\text{Spec}(\mathcal{O}_{K})$ is the point of $\widehat{\mathcal{O}_{K}}$ whose associated flat functor $F:\mathcal{O}_{K}\to\mathfrak{Sets}$ is the composition of the functor $T:\mathcal{O}_{K}\to\text{Spec}(\mathcal{O}_{K})$ with the stalk functor at $\mathfrak{p}$. This last funcotr associates to any sheaf on $\text{Spec}(\mathcal{O}_{K})$ its stlak at the point $\mathfrak{p}$ viewed as a set, so we get that $F$ is the flat functor from $\mathcal{O}_{K}$ to $\mathfrak{Sets}$ associated to the stalk $\mathcal{S}_{\mathfrak{p}}=H_{\mathfrak{p}}$.

All is proven.
\end{proof}

\begin{theo}\label{theo:liensitedeux}
Let us note $\Theta^{\star}(\mathcal{C}_{\mathcal{O}_{K}})$ the pullback of the structure sheaf of $\left(\widehat{\mathcal{O}_{K}},\mathcal{C}_{\mathcal{O}_{K}}\right)$. Then:
\begin{enumerate}
\item The stalk of $\Theta^{\star}(\mathcal{C}_{\mathcal{O}_{K}})$ at the prime $\mathfrak{p}$ is the semiring $\mathcal{C}_{H_{\mathfrak{p}}}$ and at the generic point it is $\mathbb{B}$.
\item The sections $\xi$  of $\Theta^{\star}(\mathcal{C}_{\mathcal{O}_{K}})$ on an open set $U$ of $\text{Spec}(\mathcal{O}_{K})$ are the maps $U\ni\mathfrak{p}\mapsto\xi_{\mathfrak{p}}\in\mathcal{C}_{H_{\mathfrak{p}}}$ which are either equal to $\{ 0\}$ outside a finite set or everywhere equal to the constant section $\xi_{\mathfrak{p}}=\emptyset\in\mathcal{C}_{H_{\mathfrak{p}}},\,\forall\mathfrak{p}\in U$
\end{enumerate}
\end{theo}

\begin{proof}
\begin{enumerate}
\item The result follows from the fact that the stalk of $\Theta^{\star}(\mathcal{C}_{\mathcal{O}_{K}})$ at the prime $\mathfrak{p}$ is the same as the stalk of the sheaf $\mathcal{C}_{\mathcal{O}_{K}}$ at the $\Theta (\mathfrak{p})$ (and so associated to $H_{\mathfrak{p}}$) of $\widehat{\mathcal{O}_{K}}$.

For $\{ 0\}$ we consider the stalk of $\mathcal{C}_{\mathcal{O}_{K}}$ at the point of $\widehat{\mathcal{O}_{K}}$ associated to the $\mathcal{O}_{K}$-module $\{ 0\}$ which is so $\mathcal{C}_{\{ 0\}}=\{\emptyset,\{ 0\}\}\simeq\mathbb{B}$.
\item It follows from theorem 6.1 and the definition of pullback.
\end{enumerate}
\end{proof}

\section{The square of the arithmetic site for $\mathbb{Z}[\imath]$}

In this section we will only treat the case of $\mathbb{Z}[\imath]$, the case for $\mathbb{Z}[j]$ being similar replace $[1,\imath]$ by the segment $[1,j]$.

That being said, before beginning investigating tensor products in the case of $\mathbb{Z}[i]$, we must change our point of view for an equivalent one which is functionnal, ie we will switch from convex sets to some restriction of the opposite of their support function. Although we only have an abstract description for now, I think it will be useful for the future and in other cases (for example $\mathbb{Z}[\sqrt{2}]$) to switch to the functional point of view.

\begin{defi} Let us note $\mathcal{F}_{\mathbb{Z}[\imath]}$ the set of all piecewise affine convex functions of the form $$\left\{\begin{aligned} [1,\imath ] [1,\imath]=\{1-t+it, t\in [0,1]\} & \longrightarrow &  \mathbb{R}^{+}_{\text{max}}\\ x+\imath y=1-t+\imath t & \longmapsto & \max_{(a,b)\in \Sigma}\langle (x,y),(a,b)\rangle =ax+by=a+(b-a)t\end{aligned}\right.$$ where $\Sigma$ is the set of summits (in fact thanks to the symmetries by $\imath$ and $-1$ we can only take the summits in the upper right quarter of the complex plane) of an element of $\mathcal{C}_{\mathbb{Z}[\imath]}$ (when $\Sigma$ is empty, the function associated is constant equal to $-\infty$).
\end{defi}

The easy proof of the following proposition is left to the reader.

\begin{prop} Endowed with the operations $max$ (punctual maximum) and $+$ (punctual addition), $(\mathcal{F}_{\mathbb{Z}[\imath]},max,+)$ is an idempotent semiring.
\end{prop}

We can now show that the viewpoints of the convex geometry and of those special functions are equivalent.
\begin{prop}\label{prop:faiscstrucfct}
$\left(\mathcal{C}_{\mathbb{Z}[\imath]},\text{Conv}(\bullet \cup \bullet), +\right)$ and $(\mathcal{F}_{\mathbb{Z}[\imath]},max,+)$ are isomorphic semirings through the isomorphism $$\Phi :\left\{\begin{aligned}\mathcal{C}_{\mathbb{Z}[\imath]} & \longrightarrow & \mathcal{F}_{\mathbb{Z}[\imath]} \\ C & \longmapsto \left\{\begin{aligned}[1,\imath] & \longrightarrow & \mathbb{R}^{+}_{\text{max}}\\ x+\imath y & \longmapsto & \max_{(a,b)\in\Sigma_{C}}\langle (x,y),(a,b)\rangle \end{aligned}\right.& \end{aligned}\right.$$ where $\Sigma_{C}$ stands for the set of summits of $C$ 
\end{prop}
\begin{proof}
This map $\Phi$ is immediately a surjective morphism between $\left(\mathcal{C}_{\mathbb{Z}[\imath]},\text{Conv}(\bullet \cup \bullet), +\right)$ and $(\mathcal{F}_{\mathbb{Z}[\imath]},max,+)$.

Let us now show that $\Phi$ is injective.

Let $C,C'\in\mathcal{C}_{\mathbb{Z}[\imath]}\backslash\{\emptyset ,\{ 0\}\}$ with $C\neq C'$.

Let $c' \in C'$ such that $c' \notin C$.

We identify $\mathbb{C}$ and $\mathbb{R}^2$, then thanks to Hahn-Banach theorem, there exists $\phi\in(\mathbb{R}^{2})^{\star}$ such that $\forall c\in C, \phi (c)<\phi (c')$.

But thanks to the canonical euclidian scalar product, we can identify $(\mathbb{R}^{2})^{\star}$ with $\mathbb{R}^{2}$, so let $\vec{u}\in\mathbb{R}^{2}$ such that $\phi =\langle\vec{u},\bullet\rangle$, so we have that $\forall c\in C, \langle\vec{u},c\rangle <\langle\vec{u},c'\rangle$.

But since $C$ is compact, let $\gamma\in C$ such that $\langle\vec{u},\gamma\rangle <\langle\vec{u},c'\rangle =\sup_{c\in C}\langle\vec{u},c\rangle <\langle\vec{u},c'\rangle$

Thanks to the symmetry of $C,C'$ by $\mathcal{U}_{K}$ and the identification of $\mathbb{C}$ and $\mathbb{R}^2$, we can assume that $\vec{u},\gamma ,c'\in\mathbb{C}/\mathcal{U}_{K}$.

Then finally we have that $(\Phi(C))(\vec{u})\leq \langle\vec{u},\gamma\rangle <\langle\vec{u},c'\rangle\leq (\Phi (C'))(\vec{u})$, so we have the injectivity in this case.

For the other cases, let us remark that $\Phi (\emptyset )\equiv -\infty$ the constant function equal to $'\infty$ by convention, $\Phi (\{ 0\} )=0$ the constant function equal to zero by direct calculation, and that for $C\in\mathcal{C}_{\mathcal{O}_{K}}\backslash\{\emptyset ,\{ 0\}\}$, if we take $c$ a summit of $C$ of maximal module among the summits of $C$, we immediately get that $(\Phi (C))(c)=|c|^{2}>0$ and so the injectivity is proved.

All in all, we indeed have that $\left(\mathcal{C}_{\mathbb{Z}[\imath]},\text{Conv}(\bullet \cup \bullet), +\right)$ and $(\mathcal{F}_{\mathbb{Z}[\imath]},max,+)$ are isomorphic semirings.
\end{proof}

Let us now determine $\mathcal{F}_{\mathbb{Z}[\imath]}\otimes_{\mathbb{B}}\mathcal{F}_{\mathbb{Z}[\imath]}$.

Viewing $(\mathcal{F}_{\mathbb{Z}[\imath]},\max )$ as a $\mathbb{B}$-module, we can define $\mathcal{F}_{\mathbb{Z}[\imath]}\otimes_{\mathbb{B}}\mathcal{F}_{\mathbb{Z}[\imath]}$ in the following way (see also \cite{semimod} and \cite{ccsitearith}):

\begin{defi}\label{defi:defprodtensnred} $(\mathcal{F}_{\mathbb{Z}[\imath]}\otimes_{\mathbb{B}}\mathcal{F}_{\mathbb{Z}[\imath]},\oplus)$ is the $\mathbb{B}$-module constructed as the quotient of $\mathbb{B}$-module of finite formal sums $\sum e_{i}\otimes f_{i}$ (we can remark that no coefficients are needed since $\mathcal{F}_{\mathbb{Z}[\imath]}$ is idempotent) by the equivalence relation $$\sum e_{i}\otimes f_{j}\sim\sum e_{j}'\otimes f_{j}' \Leftrightarrow \forall \Psi, \sum\Psi (e_{i},f_{i})=\sum\Psi (e_{j}',f_{j}')$$ where $\Psi$ is any bilinear map from $\mathcal{F}_{\mathbb{Z}[\imath]}\times\mathcal{F}_{\mathbb{Z}[\imath]}$ to any arbitrary $\mathbb{B}$-module and where $\oplus $ is just the formal sum.
\end{defi}

Since $(\mathcal{F}_{\mathbb{Z}[\imath]},\max , +)$ is moreover an idempotent semiring, we can see that the law $+$ of $\mathcal{F}_{\mathbb{Z}[\imath]}$ induces a new law again noted $+$ on $\mathcal{F}_{\mathbb{Z}[\imath]}\otimes_{\mathbb{B}}\mathcal{F}_{\mathbb{Z}[\imath]}$ in the following way : 

\begin{prop}\label{prop:prodtsemrg}
Let $a\otimes b$ and $a'\otimes b'\in \mathcal{F}_{\mathbb{Z}[\imath]}\otimes_{\mathbb{B}}\mathcal{F}_{\mathbb{Z}[\imath]}$, we can define $+$ such that $(a\otimes b) +(a'\otimes b')=(a+a')\otimes (b+b')$. In this way $+$ is well defined and it turns $(\mathcal{F}_{\mathbb{Z}[\imath]}\otimes_{\mathbb{B}}\mathcal{F}_{\mathbb{Z}[\imath]},\oplus,+)$ into an idempotent semiring.
\end{prop}
\begin{proof}
Let $(a,b)\in\mathcal{F}_{\mathbb{Z}[\imath]}\times\mathcal{F}_{\mathbb{Z}[\imath]}$
We define the application $\Sigma_{a,b} :\left\{\begin{aligned} \mathcal{F}_{\mathbb{Z}[\imath]}\times\mathcal{F}_{\mathbb{Z}[\imath]} & \to & \mathcal{F}_{\mathbb{Z}[\imath]}\otimes\mathcal{F}_{\mathbb{Z}[\imath]} \\ 
(a',b') & \mapsto & (a+a')\otimes (b+b') \end{aligned}\right.$

Let $(a',b'),(a'',b')\in\mathcal{F}_{\mathbb{Z}[\imath]}\times\mathcal{F}_{\mathbb{Z}[\imath]}$, 

Then we have :

$\Sigma_{a,b}(\max (a', a''),b')=(a+\max(a',a''))\otimes (b+b')=(max(a+a',a+a''))\otimes(b+b')=((a+a')\otimes(b+b'))\oplus((a+a'')\otimes(b+b'))=\Sigma_{a,b}(a',b')\oplus\Sigma_{a,b}(a'',b')$

So $\Sigma_{a,b}$ is $\mathbb{B}$-linear in the first variable. One can show in the same way that $\Sigma_{a,b}$ is $\mathbb{B}$-linear in the second variable so finally that $\Sigma_{a,b}$ is a $\mathbb{B}$-bilinear map from $\mathcal{F}_{\mathbb{Z}[\imath]}\times\mathcal{F}_{\mathbb{Z}[\imath]}$ to $\mathcal{F}_{\mathbb{Z}[\imath]}\otimes\mathcal{F}_{\mathbb{Z}[\imath]}$ so it can be factorized by the universal property of tensor product by a linear map $\sigma_{a,b}:\mathcal{F}_{\mathbb{Z}[\imath]}\otimes\mathcal{F}_{\mathbb{Z}[\imath]}\to\mathcal{F}_{\mathbb{Z}[\imath]}\otimes\mathcal{F}_{\mathbb{Z}[\imath]}$

And consequently we denote for all $a'\otimes b'\in \mathcal{F}_{\mathbb{Z}[\imath]}\otimes\mathcal{F}_{\mathbb{Z}[\imath]}$, $(a\otimes b)+(a'\otimes b'):=_{\text{def}}\sigma_{a,b}(a'\otimes b')$.

So $+$ is well defined on elementary tensors and so after for all tensors. We deduce from this that $(\mathcal{F}_{\mathbb{Z}[\imath]}\otimes\mathcal{F}_{\mathbb{Z}[\imath]} ,\oplus , +)$ is a semiring. 
 
\end{proof}

\begin{prop}\label{prop:actprodtens}
$(\mathbb{Z}[\imath])^{2}$ acts on $(\mathcal{F}_{\mathbb{Z}[\imath]}\otimes\mathcal{F}_{\mathbb{Z}[\imath]} ,\oplus , +)$ and the action preserves the semiring structure.
\end{prop}

\begin{proof}
Let $(\alpha,\beta)\in(\mathbb{Z}[\imath])^{2}$ and let $p\in\mathcal{F}_{\mathbb{Z}[\imath]}\otimes\mathcal{F}_{\mathbb{Z}[\imath]}$.

Let $I$ be a finite set and $f_{i},g_{i}\in\mathcal{F}_{\mathbb{Z}[\imath]}$ for all $i\in I$ such that $p=\bigoplus_{i\in I}f_{i}\otimes g_{i}$.

Then we define the action of $(\alpha,\beta)$ on $p$ by $(\alpha,\beta)\bullet p=\sum_{i\in I}\Phi(\alpha\bullet\Phi^{-1}(f_{i}))\otimes\Phi(\beta\bullet\Phi^{-1}(g_{i}))$ where $\Phi$ is the isomorphism between $\mathcal{C}_{\mathbb{Z}[\imath]}$ and $\mathcal{F}_{\mathbb{Z}[\imath]}$.

With this definition, the action of $(\mathbb{Z}[\imath])^{2}$ on $\mathcal{F}_{\mathbb{Z}[\imath]}\otimes\mathcal{F}_{\mathbb{Z}[\imath]}$ is directly compatible with the law $\oplus$ and so preserves the structure of $\mathbb{B}$-module.

Let $a\otimes b,a'\otimes b' \in \mathcal{F}_{\mathbb{Z}[\imath]}\otimes\mathcal{F}_{\mathbb{Z}[\imath]}$, then we have $a\otimes b + a'\otimes b'= (a+a')\otimes (b+b')$.

And for $(\alpha,\beta)\in(\mathbb{Z}[\imath])^{2}$, we have $(\alpha,\beta)\bullet ((a+a')\otimes (b+b'))=\Phi(\alpha\bullet\Phi^{-1}(a+a'))\otimes\Phi(\beta\bullet\Phi^{-1}(b+b'))$.

But $\alpha\bullet\Phi^{-1}(a+a')=\alpha\bullet\Phi^{-1}(a)+\alpha\bullet\Phi^{-1}(a')$ and $\beta\bullet\Phi^{-1}(b+b')=\beta\bullet\Phi^{-1}(b)+\beta\bullet\Phi^{-1}(b')$.

So we have $(\alpha,\beta)\bullet ((a+a')\otimes (b+b'))=(\alpha,\beta)\bullet a\otimes b + (\alpha,\beta)\bullet a'\otimes b'$.

the action of $(\mathbb{Z}[\imath])^{2}$ on $\mathcal{F}_{\mathbb{Z}[\imath]}\otimes\mathcal{F}_{\mathbb{Z}[\imath]}$ is directly compatible with the law $+$. 
\end{proof}

Thanks to this last proposition, we can therefore view $(\mathcal{F}_{\mathbb{Z}[\imath]}\otimes\mathcal{F}_{\mathbb{Z}[\imath]} ,\oplus , +)$ as an idempotent semiring in the topos $\widehat{(\mathbb{Z}[\imath])^{2}}$ (the topos of sets with an action of $(\mathbb{Z}[\imath])^{2}$ where the composition of arrows is the multiplication component by component). It allows us to define the unreduced square of the arithmetic site for $\mathbb{Z}[\imath]$ as follows :

\begin{defi}\label{defi:unredsq}
The unreduced square $\left( \widehat{(\mathbb{Z}[\imath])^{2}},\mathcal{F}_{\mathbb{Z}[\imath]}\otimes\mathcal{F}_{\mathbb{Z}[\imath]} \right)$ is the topos $\widehat{(\mathbb{Z}[\imath])^{2}}$ with the structure sheaf $(\mathcal{F}_{\mathbb{Z}[\imath]}\otimes\mathcal{F}_{\mathbb{Z}[\imath]} ,\oplus , +)$ viewed as an idempotent semiring in the topos.
\end{defi}

The idempotent semiring $(\mathcal{F}_{\mathbb{Z}[\imath]}\otimes\mathcal{F}_{\mathbb{Z}[\imath]} ,\oplus , +)$ is not necessarily a multiplicative cancellative semiring. In the case it is not, we can send it into a multiplicative cancellative semiring in the following way :

Let us set $\mathcal{P}:=\mathcal{F}_{\mathbb{Z}[\imath]}\otimes\mathcal{F}_{\mathbb{Z}[\imath]}$.

Let us denote $\mathcal{R}$ the idempotent semiring (with laws $\oplus$ and $+$ being defined component wise) $\mathcal{R}:=\mathcal{P}\times \mathcal{P}/\sim$ where $\sim$ is the equivalence relation defined as follows $$(a,b)\sim (a',b')\Leftrightarrow \exists c\in\mathcal{P}, a+b'+c=a'+b+c$$

\begin{prop}
The semiring $\mathcal{R}$ is multiplicatively cancellative.
\end{prop}

\begin{proof}
Let $(a,b),(a',b'),(c,d)\in\mathcal{R}$ with $(c,d)\neq (-\infty,-\infty)$ such that $(c,d)+(a,b)=(c,d)+(a',b')$, so we have $(a+c,b+d)=(a'+c,b'+d)$.

So we have $a+c+b'+d=a'+c+b+d$, ie $a+b'+(c+d)=a'+b+(c+d)$.

So in $\mathcal{R}$, we have $(a,b)=(a',b')$ and so $\mathcal{R}$ is multiplicatively cancellative.
\end{proof}

\begin{defi}\label{defi:prdtensred}
Let us denote $\mathcal{F}_{\mathbb{Z}[\imath]}\hat{\otimes}\mathcal{F}_{\mathbb{Z}[\imath]}$ the image of $\mathcal{P}=\mathcal{F}_{\mathbb{Z}[\imath]}\otimes\mathcal{F}_{\mathbb{Z}[\imath]}$ by the application $$\gamma : \left\{\begin{aligned}\mathcal{P} & \to & \mathcal{R} \\ a & \mapsto & (a,0)\end{aligned}\right.$$

 It is an idempotent multiplicatively cancellative semiring.
\end{defi}

\begin{prop}
The reduced tensor product of $\mathcal{F}_{\mathbb{Z}[\imath]}$ by $\mathcal{F}_{\mathbb{Z}[\imath]}$ is given by $\mathcal{F}_{\mathbb{Z}[\imath]}\hat{\otimes}\mathcal{F}_{\mathbb{Z}[\imath]}$, it satisfies the following universal property. For any multiplicative cancellative ring $R$ and any homomorphism $\rho:\mathcal{F}_{\mathbb{Z}[\imath]}\otimes\mathcal{F}_{\mathbb{Z}[\imath]}\to R$ such that $\rho^{-1}(\{0\})=\{(-\infty,-\infty )\}$, then there exists a unique homomorphism $\rho':\mathcal{F}_{\mathbb{Z}[\imath]}\hat{\otimes}\mathcal{F}_{\mathbb{Z}[\imath]}\to R$ such that $\rho=\rho'\circ\gamma$.
\end{prop}

\begin{proof}
Let $R$ a multiplicative cancellative ring and a homomorphism $\rho:\mathcal{F}_{\mathbb{Z}[\imath]}\otimes\mathcal{F}_{\mathbb{Z}[\imath]}\to R$ such that $\rho^{-1}(\{0\})=\{(-\infty,-\infty )\}$.

Let $a,b\in\mathcal{F}_{\mathbb{Z}[\imath]}\otimes\mathcal{F}_{\mathbb{Z}[\imath]}$ such that $a=b$ in $\mathcal{F}_{\mathbb{Z}[\imath]}\hat{\otimes}\mathcal{F}_{\mathbb{Z}[\imath]}$.

Then there exists $c\in\mathcal{F}_{\mathbb{Z}[\imath]}\otimes\mathcal{F}_{\mathbb{Z}[\imath]}\backslash\{ (-\infty,-\infty)\}$ such that $a+c=b+c$.

Then $\rho(a+c)=\rho(b+c)$, so $\rho(a)\times_{R}\rho(c)=\rho(b)\times_{R}\rho(c)$.

Since $\rho^{-1}(\{0\})=\{(-\infty,-\infty)\}$, $\rho(c)\neq 0_{R}$.

And so since $R$ is multiplicatively cancellative, we have $\rho(a)=\rho(b)$, so the image of an element of $\mathcal{F}_{\mathbb{Z}[\imath]}\otimes\mathcal{F}_{\mathbb{Z}[\imath]}$ by the application $\rho$ depends only on the class of this latter element in $\mathcal{F}_{\mathbb{Z}[\imath]}\hat{\otimes}\mathcal{F}_{\mathbb{Z}[\imath]}$ and so we can take $\rho':\mathcal{F}_{\mathbb{Z}[\imath]}\hat{\otimes}\mathcal{F}_{\mathbb{Z}[\imath]}\ni \gamma(a)\mapsto\rho(a)$. We have shown that the application $\rho'$ is well defined and we have $\rho =\rho'\circ\gamma$. Therefore the result is proved.

\end{proof}

\begin{prop}
The action of $\mathbb{Z}[\imath]\times\mathbb{Z}[\imath]$ on $\mathcal{F}_{\mathbb{Z}[\imath]}\otimes\mathcal{F}_{\mathbb{Z}[\imath]}$ induces an action on $\mathcal{F}_{\mathbb{Z}[\imath]}\hat{\otimes}\mathcal{F}_{\mathbb{Z}[\imath]}$ which is compatible with the semiring structure.
\end{prop}

\begin{proof}
Let $a,b\in\mathcal{F}_{\mathbb{Z}[\imath]}\otimes\mathcal{F}_{\mathbb{Z}[\imath]}$ such that in $\mathcal{F}_{\mathbb{Z}[\imath]}\hat{\otimes}\mathcal{F}_{\mathbb{Z}[\imath]}$, $a$ is equal to $b$ (ie $\gamma(a)=\gamma(b)$).

Then let $c\in\mathcal{F}_{\mathbb{Z}[\imath]}\otimes\mathcal{F}_{\mathbb{Z}[\imath]}$ such that $a+c=b+c$.

Then for any $(\alpha,\beta)\in\mathbb{Z}[\imath]\times\mathbb{Z}[\imath]$, we have $(\alpha,\beta)\bullet(a+c)=(\alpha,\beta)\bullet(b+c)$ and so $(\alpha,\beta)\bullet a+ (\alpha,\beta)\bullet c=(\alpha,\beta)\bullet b+(\alpha,\beta)\bullet c$ and finally $(\alpha,\beta)\bullet a$ equal to $(\alpha,\beta)\bullet b$ in $\mathcal{F}_{\mathbb{Z}[\imath]}\hat{\otimes}\mathcal{F}_{\mathbb{Z}[\imath]}$ (ie $\gamma((\alpha,\beta)\bullet a)=\gamma((\alpha,\beta)\bullet b)$.

Consequently the action of $\mathbb{Z}[\imath]\times\mathbb{Z}[\imath]$ is compatible with the relation $\sim$ used to define $\mathcal{F}_{\mathbb{Z}[\imath]}\hat{\otimes}\mathcal{F}_{\mathbb{Z}[\imath]}$ and since the action of $\mathbb{Z}[\imath]\times\mathbb{Z}[\imath]$ was compatible with the semiring structure of $\mathcal{F}_{\mathbb{Z}[\imath]}\otimes\mathcal{F}_{\mathbb{Z}[\imath]}$, the induced action of $\mathbb{Z}[\imath]\times\mathbb{Z}[\imath]$ is compatible with the semiring structure on $\mathcal{F}_{\mathbb{Z}[\imath]}\hat{\otimes}\mathcal{F}_{\mathbb{Z}[\imath]}$.

\end{proof}

\begin{defi}\label{defi:redsquare}
The reduced square $\left( \widehat{(\mathbb{Z}[\imath])^{2}},\mathcal{F}_{\mathbb{Z}[\imath]}\hat{\otimes}\mathcal{F}_{\mathbb{Z}[\imath]} \right)$ is the topos $\widehat{(\mathbb{Z}[\imath])^{2}}$ with the structure sheaf $(\mathcal{F}_{\mathbb{Z}[\imath]}\hat{\otimes}\mathcal{F}_{\mathbb{Z}[\imath]} ,\oplus , +)$ viewed as an idempotent semiring in the topos.
\end{defi}

\end{document}